%% file: integration.tex
\documentclass{amsart}
\usepackage[utf8]{inputenc}
\title{Uniform Mordell--Lang conjecture for Semiabelian Varieties}
\author{Zhaobo Han}
\email{hanzbtom@math.ucla.edu}
\address{Department of Mathematics, UCLA, Los Angeles, CA 90095, USA}
\author{Wenbin Luo}
\email{wbluo@math.ecnu.edu.cn}
\address{School of Mathematical Sciences, Shanghai Key Laboratory of PMMP, East China Normal University, 500 Dongchuan Road, Shanghai, 200241, China}
\author{Jiawei Yu}
\email{2201110030@pku.edu.cn}
\address{School of Mathematical Sciences, Peking University, Road Yiheyuan No.5, 100871, Beijing, China}
\usepackage{amssymb}
\usepackage{amsmath}
\usepackage{amsthm}
\usepackage{bbm}
\usepackage{hyperref}
\usepackage{quiver}
\usepackage[all]{xy}
\usepackage{tikz-cd}
\usepackage{tikz}
\usetikzlibrary{shapes.geometric, arrows}
\tikzstyle{start} = [rectangle, minimum width=2cm, minimum height=1cm, text centered, draw=black]
\tikzstyle{process} = [rectangle, minimum width=2cm, minimum height=1cm, text centered, draw=black]
\tikzstyle{arrow} = [thick,->,>=stealth]

\usepackage{mathtools}
\usepackage[toc,page]{appendix}

\input{macros.tex}

\newtheorem{theorem}{Theorem}[section]
\newtheorem{corollary}[theorem]{Corollary}
\newtheorem{lemma}[theorem]{Lemma}
\newtheorem{definition}[theorem]{Definition}

\newtheorem{definitionproposition}[theorem]{Definition-Proposition}
\newtheorem{remark}[theorem]{Remark}
\newtheorem{proposition}[theorem]{Proposition}

\makeatletter
\@addtoreset{equation}{section}
\makeatother

\begin{document}
\maketitle
\begin{abstract}
    We prove the uniform Mordell--Lang conjecture for semiabelian varieties, that is, for a subvariety $X$ of a semiabelian variety, its intersection with a finite-rank subgroup of rational points, can be covered by at most $c$ cosets contained in $X$, where $c$ depends only on the dimension of the semiabelian variety, the degree of $X$, and the rank of the subgroup. Our proof is new even in the abelian case previously proven by Gao--Ge--Kühne.
\end{abstract}
\tableofcontents
\section{Introduction}
The Mordell conjecture, proved by Faltings \cite{Faltings1983Mordell}, asserts the finiteness of rational points on smooth projective curves of genus at least $2$ over a number field. Vojta gives a new proof \cite{vojta} via Diophantine approximation. His method was developed to prove the Mordell--Lang conjecture \cite{Faltings1991Mordell-Lang}, which is a higher dimensional generalization of both the Mordell conjecture and the Manin--Mumford conjecture.

Recently, there has been breakthroughs on the uniformity problem of the aforementioned conjectures. In \cite{DGH,kühne2024equi}, Dimitrov--Gao--Habegger and Kühne prove the uniform Mordell--Lang conjecture for curves, namely, the number of rational points can be bounded from above in terms of the curve's genus ($\geq 2$) and the Mordell-Weil rank, which answers a question of Mazur \cite[pp. 223]{Mazur00uniform}. In the case of genus $2$, this was previously obtained by DeMarco--Krieger--Ye \cite{DKY2020genus2} in some cases. In a rather new article of Yu--Yuan--Zhou \cite{yu2026quantitativitynumberrationalpoints}, an explicit bound is given via an approach different from \cite{DGH}. 

The general uniform Mordell--Lang conjecture, which concerns arbitrary subvarieties in abelian varieties, was obtained by Gao--Ge--Kühne \cite{GGK}.

\vskip 0.3em

{\it The goal of the current article} is to generalize the results of \cite{GGK} to semiabelian varieties. 
Here we give a brief description. Let $K$ be an algebraically closed field of characteristic $0$. A \emph{semiabelian variety} $G$ over $K$ is a torus extension of an abelian variety, that is, we have an exact sequence of group varieties:
$$0\rightarrow \mathbb{G}_{m}^t\rightarrow G\xrightarrow{} A\rightarrow0,$$
where $A$ is an abelian variety of dimension $g$ over $K$. Here we identify an algebraic torus with $\mathbb G_m^t$ since $K$ is algebraically closed. Note that the Mordell--Lang conjecture for semiabelian variety is due to McQuillan \cite{McQuillan95semiab}.
Now we fix a toric compactification $\mathbb G_m^t\hookrightarrow V$ of $\mathbb G_m^t$, which induces a compactification $G\hookrightarrow\ovl G:=G\times V/\mathbb G_m^t$.
Let $M$  be the line bundle associated with the boundary divisor $\overline G\setminus G$, and $L$ be the pull-back of an ample line bundle on $A$ to $\ovl G$. The main result of the article is as follows:
\begin{theorem}[Uniform Mordell--Lang for semiabelian varieties]\label{thm_UML}
    There exists a constant $c(g+t,d)$ satisfying the property as follows. Let $\Gamma\subset G(K)$ be a subgroup of finite rank. Let $X\subset G$ be a closed subvariety with $d=\deg_{L+M}(\ovl X)$. Then there are at most $c(g+t,d)^{\rk(\Gamma)+1}$ cosets $H_1,H_2,\dots$ contained in $X$ such that
$$X(K)\cap\Gamma=\bigcup H_i(K)\cap\Gamma.$$
\end{theorem}
In this article, we only consider the compactification $\mathbb G_m^t\hookrightarrow (\mathbb P^1)^t$, and the degrees induced by different compactifications are comparable.

In the case that $t=0$, our result is exactly \cite[Theorem 1.1]{GGK}, which removes the dependence on the Faltings height $h_{\mathrm{Fal}}(A)$ of the abelian variety in a previous work of R\'emond \cite{Remond00UML}, although the latter provides explicit constants. In the case of algebraic tori, the known uniformity result of this kind is due to Evertse--Schlickewei \cite{ESSLinear2002} for hyperplanes, and R\'emond \cite{Remond02tori} for general subvarieties. Both results also gives explicit constants.

Instead of just a generalization of the abelian case, after introducing several new methods, we present a concise treatment to the proofs of some results in \cite{ZG1}, \cite{kühne2024equi}, and \cite{GGK} in this semiabelian setting, where the first two works are required in the proof of the last.
\subsection{Outline of the proof and new ingredients}
To obtain the uniformity in Theorem \ref{thm_UML}, a general philosophy is that one may consider the problem on the moduli space, which is also the main idea of \cite{DGH2021Umlpencil,DGH,GGK}. We roughly follow the same philosophy. Our strategy can be summarized in the workflow as follows.
\begin{equation}\label{eq_workflow}\allowdisplaybreaks\begin{tikzcd}\allowdisplaybreaks
     &\boxed{\parbox{7em}{\centering\text{Non-degeneracy}\\\text{of the universal}\\\text{ subscheme}}} \arrow[rd] & & \boxed{\parbox{7em}{\centering\text{Equidistribution}\\\text{on families}}}\arrow{ld}\\
     & &\boxed{\parbox{9em}{\centering\text{Uniform Bogomolov}\\\text{with ``extra term"}}} \arrow[d]& \boxed{\parbox{7em}{\centering\text{Diophantine}\\\text{Approximation}}}\arrow{ld}\\ & &\boxed{\parbox{10em}{\centering\text{Uniform Mordell--Lang}\\ \text{(points version)}}}\arrow[d]& \boxed{\parbox{5em}{\centering\text{Relative}\\\text{Ueno locus}}}\arrow{ld}\\ & &\boxed{\text{Uniform Mordell--Lang}}\\
    \end{tikzcd}
\end{equation}
Note that the counting of ``small" points is done by the first row, ``large points" by the second row, and positive dimensional cosets by the third row. One may see the proof of Proposition \ref{ML1} for the precise definition of the terminologies. From the second row, our treatment is different from \cite{GGK}, so this is a new proof even in the abelian case.

We now explain the diagram \eqref{eq_workflow} in reverse order. 

In order to demonstrate how we work on the moduli space, we consider a simplified situation as follows. Let $\mc G\rightarrow S$ be a family of semiabelian varieties parametrized by a quasi-projective variety over $\C$, given by the exact sequence 
$$0\rightarrow \mathbb G_{m,S}^t\rightarrow \mc G\rightarrow \mc A\rightarrow 0$$
where $\mc A$ is an abelian scheme over $S$ with a polarization $L$. 

Assume that $\mc X\subset \mc G$ is a closed subvariety such that for every $s\in S(\C)$, the triple $(\mc G_s, L_s, \mc X_s)$ is as in Theorem \ref{thm_UML}. We will obtain the result after applying our argument to suitable such $(\mc G,L,\mc X)$'s (in practice, we will replace $L$ by $L+ [-1]^*L$).

Recall that for a closed subvariety $X\subset G$, the union of positive dimensional cosets contained in $X$ is called the \emph{Ueno locus}, whose complement in $X$ is denoted as $X^\circ.$ 

In \cite{GGK}, by a recursive argument, Theorem \ref{thm_UML} can be reduced to a uniform estimate of $\#(X^\circ\cap \Gamma)$, namely, the abelian case of Corollary \ref{coro_MLforpointsgeneralpol}. Their argument is based on the Poincar\'e reducibility and sophisticated degree estimates in \cite{Bogomolv1980Ueno,MW93Minimal,Remond00UML}, which are absent in our case.

Instead, we offer a novel treatment by applying the \emph{geometric Zilber--Pink conjecture} for connected mixed Shimura varieties, which is proven in \cite{BU}. 
In Proposition \ref{prop_relative_ueno}, we prove a finiteness result for what we call the \emph{relative Ueno locus}, namely, the union $\displaystyle\bigcup_{s\in S(\C)}(\mc X\setminus\mc X_s^\circ)$. The proof strategy is to take a sufficiently large fiber power $\mc X^{[m]}:=\underbrace{\mc X\times_S\cdots\times_S\mc X}_{m\text{ times}}$ of $\mc X$ so that we can reduce to the case that $\mc G$ maps injectively to the modular spaces of semiabelian varieties (with the abelian quotient having certain polarization and level structure).

Combining with the discussion in \S \ref{subsection_proof_main_theorem}, we reduce Theorem \ref{thm_UML} to its point counting version, i.e., the following result (see also Corollary \ref{MLforpoints}):
\begin{theorem}[Uniform Mordell--Lang conjecture for points]\label{thm_UML_pts}
There exists a constant $c'(g+t,d)$ satisfying the property as follows. Let $G,X$ and $\Gamma$ be as in Theorem \ref{thm_UML}. Then
$$\#(X^\circ(\ovl\Q)\cap\Gamma)\le c'(g+t,d)^{\rk(\Gamma)+1}.$$
\end{theorem}

We remark that our reduction from $K$ to $\ovl \Q$ is also due to invoking the relative Ueno locus via a specialization argument (see the proof of Corollary \ref{coro_MLforpointsgeneralpol}).

Theorem \ref{thm_UML_pts} is a generalization of \cite[Theorem 1.1]{DGH} and \cite[Theorem 1.1']{GGK}. The proof of \cite[Theorem 1.1']{GGK} follows the ball and cone packing argument as in Vojta's proof \cite{vojta}. After implementing the quantitative version of Vojta's inequality and Mumford's inequality obtained in \cite{Remond00UML,Remond00Inquality}, along with a reformulation by David--Philippon \cite{DP07Minoration}, Gao--Ge--Kühne reduce \cite[Theorem 1.1']{GGK} to the so-called \emph{new gap principle} \cite[Theorem 1.2]{GGK}.

In the semiabelian case, while Vojta's inequality with explicit constants is proven by R\'emond \cite[Theorem 4.1]{remond}, the Mumford inequality is not available. To overcome this, we adapt a variant of Vojta's proof to the Mordell--Lang conjecture, which was first introduced by Yuan \cite{yuan2025vojtasproofmordellconjecture} in the curve case. One may also find the treatment in \cite{yu2026quantitativitynumberrationalpoints}, and \cite{Yu2026uniform} in the function fields case. We generalize the argument to the higher dimensional case, which provides
an alternative to Mumford's inequality. 

As a trade-off, we need to give a much more delicate estimate for points of small canonical height $\widehat h(\cdot)$ (see \eqref{eq_canonical_height} for the definition of the canonical height). Namely, we prove the \emph{uniform Bogomolov conjecture with extra term} as follows. Let $X\subset G$ be a closed subvariety. We say $X$ \emph{generates} $A$ if $A$ itself is the smallest abelian subvariety containing $\pi_{G/A}(X)-\pi_{G/A}(X).$
\begin{theorem}[Uniform Bogomolov conjecture with extra term]\label{SUBC}
There exists a positive constant $c_3=c_3(g+t,d)>0$ satisfying the property as follows. Let $G,L$ and $X$ be as in Theorem \ref{thm_UML}, with $L$ being a pull-back of a principal polarization on $A$. Assume $X$ generates $A$, and $X$ has finite stabilizer. Then there exists a reduced subscheme $Z\subsetneq X$ of degree $$\deg Z\le c_1(g+t,d),$$ and a point $g_0\in G(\ovl\QQ)$ such that for any $g_1\in G(\ovl\QQ)$, we have
$$\{x\in X^\circ(\ovl\Q):\widehat{h}(x-g_1)\le c_3(g+t,d)\widehat h(g_1-g_0)\}\subset Z(\ovl\Q).$$
Here $c_1(g+t,d)$ is the constant in Theorem \ref{bogomolov}.
\end{theorem}

In the case of curves embedded in their Jacobians, Yuan proved a similar result in \cite{Yuan2026bigness}, via Yuan-Zhang's theory of adelic line bundles \cite{YuanZhang} and previous works \cite{Zhang93admissible,zhang2010Gross,Cinkir2011,deJong2018} due to Zhang, Cinkir and de Jong. 
\begin{remark}
Let us clarify some terminologies.

Classically, the \emph{uniform Bogomolov conjecture} is  the height inequality in Theorem \ref{SUBC} by $\widehat h(x)\leq c_3'(g+t,d)$ for some other constant $c_3'$. 
In the abelian case, this is precisely \cite[Theorem 1.3]{GGK}. Before their result, \cite{DP07Minoration,DKY2020genus2,Kuhne} proved some cases. Notice that this version cannot imply uniform Mordell--Lang, even for curves; see \cite[below Theorem~1.3]{GGK}. Indeed, to achieve uniform Mordell--Lang, one needs a stronger version of {\it uniform Bogomolov}, called the {\it new gap principle} in \cite[Theorem~4.1]{ZG_Survey} and \cite[Theorem 1.2]{GGK}. Yuan's \cite[Theorem 1.1]{Yuan2026bigness} further strengthens this uniform Bogomolov type result by introducing an extra term.

During the preparation of the current article, we were told by Hultberg that using the GVF theory \cite{yaacov2022GVF}, a generalization of \cite[Theorem 1.2, Theorem 1.3]{GGK} to semiabelian varieties has been proven \cite{hultberg2026newgap} with a stronger condition that $X$ generates $G$. The key to the proof is a reduction to Gao--Ge--Kühne \cite{GGK} via the relative Faltings-Zhang method established by Luo--Yu \cite{luo2025gbc} on geometric Bogomolov conjecture for semiabelian varieties.

However, in our case, the direct generalization of \cite[Theorem 1.2]{GGK} to semiabelian varieties can {\bf not} deduce Theorem~\ref{thm_UML} due to the absence of Mumford's inequality. The version with the extra term is hence necessary.
\end{remark}
In the case of algebraic tori, these results are known much earlier due to Bombieri--Zannier \cite{BZ95Tori} (see also \cite{DP99tori,AD06tori}). 

To prove our uniform Bogomolov conjecture with the extra term, we implement Yuan--Zhang's theory of adelic line bundles on quasi-projective varieties.

The rough idea is as follows. Since $G$ acts on the Hilbert scheme by sending the modular point $[X]$ associated with $X$ to $[X-g_1]$ for $g_1\in G.$ Imagine that we have a ``quotient" of the Hilbert scheme by $G$, we may choose a ``section" of the ``quotient", therefore all translation of $X$ are represented by a single modular point. Since we assign strong positivity to the adelic line bundle defining the modular height of subvarieties in \S\ref{subsec_hilbert_modular}, we get a height inequality in the form of Proposition \ref{prop_height_inequality_translation}. This reduces the uniform Bogomolov conjecture with the extra term to a uniform Bogomolov type result as in Theorem \ref{bogomolov}.

In practice, instead of constructing such ``quotient" and ``section", we use a recursive argument on the universal Hilbert scheme as in \S\ref{subsec_proof_subc}. 

To prove Theorem \ref{bogomolov}, the strategy is more or less standard. We need the following two main ingredients:
\begin{enumerate}
    \item A study on the rank of Betti maps for a family of semiabelian varieties, which is a generalization of \cite{ZG1};
    \item An equidistribution theorem for a family of semiabelian varieties, generalizing both \cite{kühne2024equi} and \cite{Kuhne}.
\end{enumerate} 
Note that the notion of \emph{non-degeneracy}, namely, the fullness of the rank of the Betti map, is equivalent to the nonvanishing of the measure associated with the equidistribution. Hence (2) relies on (1).

The study of Betti maps has a long history that goes back to Masser and Zannier \cite{MZ:torsionanomalous, zannier2012unlikely, MasserZannierTorsionPointOnSqEC, MASSER2014116, MasserZannierRelMMSimpleSur}. Let $\mathcal A_g$ be the universal abelian variety over the moduli space $\mathbb A_g$ of abelian varieties of a certain polarization type and level structure, and $\mathfrak H_g^+$ be the Siegel upper half space, which is the universal covering of $\mathbb A_g$. The abelian Betti map $(\mc A_g)_{\mathfrak H_g^+}\rightarrow \R^{2g}/\Z^{2g}$ is a real analytic, fiberwise isomorphism as real Lie groups over $\mathfrak H_g^+$, such that each fiber over $\R^{2g}/\Z^{2g}$ is complex analytic. In \cite{ACZ2020Betti}, this map is used to prove the torsion points are dense on sections of certain abelian scheme. In the study of uniformity problems in Mordell--Lang conjecture, what we care the most is the rank of the Betti map restricted on a subvariety $\mc X\subset \mc A_g$ after passing to $\mathfrak H_g^+$. Gao gives a group scheme theoretic criterion to test the generic Betti rank \cite[Theorem 1.1]{ZG1}. This time we work on the family of semiabelian varieties: $$\mc P_{g,t}^\times :=\mc P_{g,1}^\times \times_{\mc A_g} \cdots\times_{\mc A_g}\mc P_{g,1}^\times\rightarrow (\mc A_g^{\vee})^{[t]},$$ where $\mc A_g^\vee$ is the dual of $\mc A_g$ and $\mc P_{g,1}^\times \rightarrow \mc A_g\times_{\mathbb A_g}\mc A_g^\vee$ is the universal Poincaré torsor. Using the fact that each fiber of $\mc P_{g,t}^\times$ is isomorphic to $(\C^\times)^t\times \R^{2g}/\Z^{2g}$ as real Lie group, we construct the semiabelian Betti map (see \S\ref{subsec_Betti_maps} for details).

We provide a lightweight but almost self-contained proof for a generalization of Gao's result to $\mc P_{g,t}^\times$ (see \S \ref{subsection_degeneracy_criterion}), which may have further applications in the forthcoming article of Zhaobo Han on the relative Manin--Mumford conjecture for semiabelian varieties, whose abelian case is proved by Gao--Habegger \cite{GH2026RMM}.

Moreover, let $\mc X\subset \mc P_{g,t}^\times$ be a subvariety, which maps surjectively to some irreducible subvariety $S\subset(\mc A_g^\vee)^{[t]}$. We obtain a non-degeneracy result of the self product $\mc X^{[m]}$ for $m\gg 0,$ generalizing \cite[Theorem 1.3]{ZG1}.
The actual result is stronger, see Theorem \ref{theo_nondeg_gen} which works for families over Hilbert scheme. Here we give a weaker version to demonstrate.
\begin{theorem}
    Let $\ovl\eta$ be the geometric generic point of $S$.
    Assume that 
    \begin{enumerate}
        \item For every $s\in S$, $\mc X_s$  generates $(\mc A_g)_s$;
        \item $\mc X_{\ovl \eta}$ is irreducible and has finite stabilizer.
    \end{enumerate}
    Then $\mc X^{[m]}:=\underbrace{\mc X\times_S\cdots\times_S\mc X}_{m\text{ times}}$ is non-degenerate for every $m\gg 0.$
\end{theorem}

Now let $(\mc G/S,L,\mc X)$ be as in the beginning of the subsection. This time we assume that $\mc G, S$ and $\mc X$ are defined over $\ovl {\Q}$. We first construct an adelic line bundle $\ovl M$ on the compactification $\ovl{\mc G}:=(\mc G\times (\mathbb P^1)^t)/\mathbb G_m^t$, with the underlying line bundle $M:=\mc O_{\ovl {\mc G}}({\ovl{\mc G}/\mc G}).$ Let $\ovl L$ be the adelic line bundle constructed in \cite[\S 6]{YuanZhang}. Then we have the associated height function $\widehat h_{\ovl L}$ and $\widehat h_{\ovl M}:\mc G(\ovl {\Q})\rightarrow \R_{\geq 0}$. Denote that $\widehat h:=\widehat h_{\ovl M}+\widehat h_{\ovl L}.$
\begin{theorem} Assume that the abelian quotient of $\mc G$ is principally polarized, and $\mc X$ is non-degenerate.
Let $x_i\in\XXX\ (n=1,2\dots)$ be a generic sequence of algebraic points. If
$$\lim_{i\to\infty}\widehat{h}(x_i)=0,$$
then the average Galois orbit of the sequence $(x_i)$ equidistributes to
a probability measure $\mu$ on $\mc X(\C).$
\end{theorem}
We will give the explicit description of $\mu$ in Theorem \ref{theorem_equidistribution}, which is an intertwined wedge product of currents given by the Betti map, called the \emph{Betti currents}. 

The proof of the equidistribution is done by implementing a general equidistribution result of Balla\"{y}--Sombra \cite{BS}, after proving several positivity results for $\ovl M$ and $\ovl L$ using an infinitesimal trick. 
\subsection{Organization of the article}
In \S \ref{section_Betti}, we first give the definition of Betti maps and Betti currents, including the data of universal coverings of $\mc P_{g,t}^\times$.

In \S \ref{section_Shimura}, we explain the connected mixed Shimura variety structure on $\mc P_{g,t}^\times$, along with a description of certain types of connected mixed Shimura subvarieties of $\mc P_{g,t}^\times.$ Moreover, we study the Shimura quotient by a normal subgroup, which will be used in the proof of the non-degeneracy.

We then present a relatively short proof for a generalization of \cite[Theorem 1.1]{ZG1} in \S \ref{section_degeneracy}, that is, a criterion for degeneracy. Our proof implements the geometric Zilber--Pink conjecture and the equivariancy of Betti maps shown in \S \ref{section_Betti}.

To prove the non-degeneracy of fiber power, we give a generalization of Gao's isogeny trick (see \cite[Proof of (10.4)]{ZG1}) in \S \ref{sec_nondeg_fiberproduct}. 

We then briefly review Yuan-Zhang's theory of adelic line bundles \cite{YuanZhang} in  \S\ref{section_adelic}.

In \S\ref{section_canonical}, we construct the canonical adelic line bundle $\ovl M$ and $\ovl L$ for our use and explain their relationship with the Betti map. The connection with non-degeneracy is discussed in \S\ref{section_intersection}.

In \S\ref{section_UBC}, we first prove the uniform Bogomolov conjecture with a bound of height in terms of subvarieties's modular height, using the equidistribution result. We the obtain our uniform Bogomolov conjecture with extra term via an optimization proposition.

In \S\ref{section_UML}, we first use a refined ball packing argument given by the uniform Bogomolov conjecture with extra term to deduce the points counting version of uniform Mordell--Lang conjecture (Theorem \ref{ML}). In \S\ref{subsection_relative_ueno}, we introduce the relative Ueno locus, and prove the finiteness result, i.e., Proposition \ref{prop_relative_ueno}. Combining with Theorem \ref{MLforpoints}, we then conclude the proof of Theorem \ref{ML}.

The Appendix \S\ref{Appendix} contains only the construction of the isogeny used in \S \ref{sec_nondeg_fiberproduct}, which is an easy consequence of Poincaré reducibility for abelian variety and algebraic tori.
\subsection*{Acknowledgement}
We thank Ziyang Gao, Junyi Xie and Xinyi Yuan for their consistent supports and valuable discussions. We thank Kaiyuan Gu and Chenxin Huang for some communications on their preprints. We also thank Nuno Hultberg for sharing his preprint.

\subsection*{AI usage}
All ideas in the article are due to authors' brains.
In the early stage of the preparation, Wenbin Luo uses \textbf{DeepSeek V4} to help polish some part of the preprint. After seeing the AI develop much faster than we expect, we decide not to use AI in any part of this work.
\section*{Notation and conventions}
In this article, a \emph{variety} refers to an integral separated scheme of finite type over a field.

For two line bundles $L_1$ and $L_2$ on a scheme, we denote by $L_1+L_2$ for their tensor product.

Let $Y$ be a scheme and $X$ be a $Y$-scheme. We denote by $\pi_{X/Y}$ the structure morphism $X\rightarrow Y$ if not specified.

Let $S$ be a scheme. A \emph{semiabelian scheme} over $\mc G$ over $S$ is a group $S$-scheme given by an exact sequence
$$0\rightarrow \mathbb{G}_{m,S}^t\rightarrow \mathcal G\rightarrow \mathcal A\rightarrow 0$$
of group $S$-schemes, where $\mathcal A/S$ is an abelian scheme of relative dimension $g$. We remark that this is a semiabelian scheme whose toric part is split and of constant rank in the common sense\footnote{Note that in this article, we always work over an algebraically closed field, so that we can always identify a torus with $\mathbb G_m^t$. Moreover, such identification can be considered as a ``level structure", so that our moduli space of semiabelian varieties defined in \S\ref{sec_univ_cov} is fine due to the Weil--Barsotti formula \cite[\S III.18]{Oort1966group}.}. 

We say a semiabelian scheme $\GGG$ is \emph{(principally) polarized} if its abelian quotient is so.
The same definition applies to level structures.

Let $\mc X\subset \mc G$ be a subscheme, and $\mc G'$ be a semiabelian subscheme of $\mc G$. We denote by $\mc X/\mc G'$ the image of $\mc X$ in $\mc G/\mc G'$, which is a constructible set. Its dimension is defined as the dimension of its Zariski closure.

\section{Betti maps and Betti currents}\label{section_Betti}
Let $\Delta$ be a simply connected complex analytic space. 
Let $\mathcal G_\Delta\rightarrow \Delta$ be a family of connected complex algebraic groups.
Fix a point $s\in \Delta$, we say $b:\GGG_\Delta\rightarrow (\GGG_{\Delta})_s$ is a \emph{Betti map} if the following conditions are satisfied:
\begin{enumerate}
    \item[(B1)] $b$ is real analytic.
    \item[(B2)] For any $x\in (\GGG_\Delta)_s$, $b^{-1}(x)$ is complex analytic.
    \item[(B3)] For any $s'\in S$, $b|_{(\GGG_\Delta)_{s'}}:(\GGG_\Delta)_{s'}\rightarrow (\GGG_\Delta)_s$ is a group isomorphism.
\end{enumerate}
Note that the conditions (B1) and (B3) require that any two fibers are isomorphic as real Lie groups. 

In the case that $\mc G_\Delta$ is a family of abelian varieties, the idea of using Betti maps to solve Diophantine problems dates back to works of Masser and Zannier \cite{MZ:torsionanomalous, zannier2012unlikely, MasserZannierTorsionPointOnSqEC, MASSER2014116, MasserZannierRelMMSimpleSur}. More recently, it has seen many more arithmetic applications, joint with the Betti forms of Mok \cite[pp. 374]{Mok91Bettiform}. 
For example, Andr\'e--Corvaja--Zannier proves the density of torsion points on sections of certain abelian schemes \cite[Theorem 2.3.2]{ACZ2020Betti}, Gao--Habegger and Cantat--Gao--Habegger--Xie prove the geometric Bogomolov conjecture in characteristic $0$ \cite{GH2019HeighIneq,CGHX2021GBC}, and the series of works \cite{DGH2021Umlpencil,kühne2024equi,GGK} on uniform Mordell--Lang conjecture relies on the study of rank of Betti maps due to Gao \cite{ZG1}.


In this article, we generalize these constructions to families of semiabelian varieties, whose torus part is of rank $t\geq 0$ and abelian quotient has dimension $g\geq 0$. Moreover, we assume that the abelian quotient has a principal polarization and level-$4$ structure. When $g=t=1$, the Betti map was constructed by Bertrand--Masser--Pillay--Zannier \cite{BMPZ}.


In \S \ref{subsection_alg_grps} and \S\ref{sec_univ_cov}, we will first introduce the moduli space of such semiabelian varieties, and its universal covering, following the construction in \cite{BE}. We then define the Betti maps and its associated currents in \S\ref{subsec_Betti_maps} and \S \ref{subsection_betti_currents}. Compare with previous works, we in addition introduce the toric Betti currents $\omega_{\tor}$, which is a difference of semipositive $(1,1)$-currents in contrast with the abelian case (see \S \ref{subsec_betti_adelic_relation} for a further discussion).
\subsection{Algebraic groups}\label{subsection_alg_grps}
We first fix the following notation for algebraic groups:
\begin{itemize}
    \item For any positive integer $p,q,r>0$, we denote by $\mr M_{p,q}$ the $\Q$-vector group of $p\times q$ matrices equipped with the addition group law, which is isomorphic to $\mathbb G_{a}^{pq}.$ 
    \item Set $J:=\begin{pmatrix}0\qquad I_g\\-I_g\quad0\end{pmatrix}$.
    Let\begin{align*}
    \mr{Sp}_{2g}&:=\{h\in  \mr{GL}_{2g}\,:\, h^\top J h=J\},\\
    \mr{GSp}_{2g}&:=\left\{h\in \mr{GL}_{2g}\,:\,\displaystyle h^\top J h=cJ\text{  for some $c\in\mathbb{Q}^{\times}$}\right\}.
    \end{align*}
    In particular, we denote that $\mu(g):=c$ for $g\in \mr{GSp}_{2g}.$
    Let $$\mr{Sp}_{2g}(1+4\Z):=\{h\in\mr{Sp}_{2g}(\Z)\,:\, h\equiv I_{2g}(\mod 4\Z)\}$$ be an arithmetic subgroup of $\mr{Sp}_{2g}(\Z)$.

    \item Let $P_{g,t}$ be the $\Q$-group $$\left\{\begin{pmatrix}
            \mu(h)I_{t} & * &*\\
            0& h& *\\
            0& 0& 1
        \end{pmatrix}\in \mr{GL}_{t+2g+1}\,:\, h\in \mr{GSp}_{2g}\right\}.$$
        We may and do write $P_{g,t}=(\mr M_{t,1}\oplus \mr M_{2g,1})\rtimes \mr M_{t,2g}\rtimes \mr{GSp}_{2g}$ if we denote by $(z,y,x,g)$ the element $\begin{pmatrix}
            \mu(h)1_{t} & xJ &z\\
            0& h& y\\
            0& 0& 1
        \end{pmatrix}$
    for $z\in \mr{M}_{t,1}, y\in \mr{M}_{2g,1},x\in\mr{M}_{t,2g}$ and $h\in \mr{GSp}_{2g}.$
    We have the following commutative diagram:
    \[\begin{tikzcd}
	P_{g,t} && \\
	& {(\mr M_{2g,1}\oplus \mr M_{t,2g})\rtimes \mr{GSp}_{2g}} & {\mr M_{2g,1}\rtimes \mr{GSp}_{2g}} \\
	& {\mr M_{t,2g}\rtimes \mr{GSp}_{2g}} & \mr{GSp}_{2g}
 	\arrow["p_{\pi}"', from=1-1, to=2-2]
	\arrow["{p_{\ab}}", curve={height=-18pt}, from=1-1, to=2-3]
	\arrow["p_{\mr{sab}}"', curve={height=18pt}, from=1-1, to=3-2]
	\arrow["p_{\mr{Siegel}}"', curve={height=80pt}, from=1-1, to=3-3]
	\arrow[from=2-2, to=2-3]
	\arrow[from=2-2, to=3-2]
	\arrow[from=2-3, to=3-3]
	\arrow[from=3-2, to=3-3]
\end{tikzcd}.\]
\end{itemize}
The following will be frequently used in next section.
\begin{lemma}\label{lemm_normality_weight-2}
    Any $\Q$-vector subgroup of $\mr M_{t,1}$ is normal in $P_{g,t}.$ 
\end{lemma}
\begin{proof}
    This this a direct consequence of the matrix representation of $P_{g,t}.$
\end{proof}
\subsection{Universal coverings}\label{sec_univ_cov}
Now set
    $\mathfrak H_g^+:=\{\tau\in \mr M_{g,g}(\C)\,:\, \tau^\top=\tau, \mr{Im}(\tau)>0\}$, that is, the Siegel upper half space. 
    Define that \begin{align*}\mathcal{X}^+_{g,t}:&=\mr M_{t,1}(\C)\times\mr M_{g,1}(\C)\times \mr M_{t,g}(\C)\times \mathfrak H_g^+
    \end{align*}
    which will serve as the universal covering of the moduli space of semiabelian varieties of our concern.
    We denote by $P_{g,t}(\R)^+$ the neutral component of $P_{g,t}(\R).$ Then $P_{g,t}(\R)^+\mr M_{t,1}(\C)$ acts on $\mc X_{g,t}^+$:
    \begin{equation}\label{eq_group_action}
        \begin{pmatrix}
            \mu(g) \mr{I}_t& x_1& x_2 &  z\\
            0& A & B& y_1\\
            0& C & D &y_2\\
            0& 0 & 0 & 1
        \end{pmatrix}(w,v,u,\tau)\longmapsto\left(w',v',u',\tau'\right).
    \end{equation}
where $(x_2, -x_1)=(x_1, x_2)J^{-1}\in \mr{M}_{t,2g}(\R), \begin{pmatrix}y_1 \\ y_2\end{pmatrix}\in \mr{M}_{2g,1}(\R),z\in \mr M_{t,1}(\C), g=\begin{pmatrix}A\quad B \\ C\quad D\end{pmatrix}\in \mr{GSp}_{2g}(\R)^+$, $w\in \mr{M}_{t,1}(\C),v\in \mr{M}_{g,1}(\C),u\in \mr{M}_{t,g}(\C),\tau\in \mathfrak H_g^+$ and $(w',v',u',\tau')$ is given by the following formulas:
    \[
    \begin{aligned}
    \tau' &= (A\tau + B)(C\tau + D)^{-1}, \quad (\text{Möbius transform})\\
    u'   &= \bigl( \mu\, u + x_1 \tau + x_2 \bigr) (C\tau + D)^{-1}, \\
    v'   &= A v + y_1 - (A\tau + B)(C\tau + D)^{-1} (C v + y_2) \\
     &= A v + y_1 - \tau' (C v + y_2), \\
    w'   &= \mu\, w + x_1 v + z - u' (C v + y_2).
    \end{aligned}
\]
Consider the commutative diagram
\[\begin{tikzcd}
	\mc X^+_{g,t} && \\
	& {\mr M_{g,1}(\C)\times \mr M_{t,g}(\C)\times \mathfrak H^+_g} & {\mr M_{g,1}(\C)\times \mathfrak H^+_g} \\
	& {\mr M_{t,g}(\C)\times \mathfrak H^+_g} & \mathfrak H^+_g
 	\arrow["\widetilde\pi"', from=1-1, to=2-2]
	\arrow["{\widetilde\pi_{\ab}}", curve={height=-18pt}, from=1-1, to=2-3]
	\arrow["\widetilde\pi_{\mr{sab}}", curve={height=20pt}, from=1-1, to=3-2]
	\arrow["\widetilde\pi_{\mr{Siegel}}"', curve={height=90pt}, from=1-1, to=3-3]
	\arrow[from=2-2, to=2-3]
	\arrow[from=2-2, to=3-2]
	\arrow[from=2-3, to=3-3]
	\arrow[from=3-2, to=3-3]
\end{tikzcd}.\]
The action of $P_{g,t}(\R)^+\mr M_{t,1}(\C)$ on $\mc X_{g,t}^+$ induces the following actions of
\begin{itemize}
    \item $(\mr M_{2g,1}(\R)\oplus \mr M_{t,2g}(\R))\rtimes \mr{GSp}_{2g}(\R)^+$ on ${\mr M_{g,1}(\C)\times \mr M_{t,g}(\C)\times \mathfrak H^+_g}$,
    \item $\mr M_{t,2g}(\R)\rtimes \mr{GSp}_{2g}(\R)^+$ on ${\mr M_{t,g}(\C)\times \mathfrak H^+_g}$,
    \item $\mr M_{2g,1}(\R)\rtimes \mr{GSp}_{2g}(\R)^+$ on ${\mr M_{g,1}(\C)\times \mathfrak H^+_g}$,
    \item and $\mr{GSp}_{2g}(\R)^+$ on $\mathfrak H^+_g$
\end{itemize}
via $p_\pi$, $p_{\sab}$, $p_{\ab}$ and $p_{\mr{Siegel}}$ respectively.
Then one obtain the following moduli spaces 
\begin{itemize}
    \item $\mc P_{g,t}^\times:=(\mr M_{t,1}(\Z)\oplus \mr M_{2g,1}(\Z))\rtimes \mr M_{t,2g}(\Z)\rtimes \mr{Sp}_{2g}(1+4\Z))\backslash \mc X_{g,t}^+,$
    \item $(\mc A_g^{\vee})^{[t]}:=(\mr M_{t,2g}(\Z)\rtimes \mr{Sp}_{2g}(1+4\Z)^+)\backslash({\mr M_{t,g}(\C)\times \mathfrak H^+_g})$,
    \item $\mc A_g:=(\mr M_{2g,1}(\Z)\rtimes \mr{Sp}_{2g}(1+4\Z))\backslash ({\mr M_{g,1}(\C)\times \mathfrak H^+_g}),$
    \item and $\mathbb A_g:=\mr{Sp}_{2g}(1+4\Z)\backslash \mathfrak H^+_g.$
\end{itemize}
Here $\mathbb A_g$ is the fine moduli space of principally polarized abelian varieties with level-$4$ structure, $\mc A_g$ is the universal abelian variety over $\mathbb A_g$, and $\mc A_g^\vee$ the dual of $\mc A_g$. If $t=1$, $\mathcal P_{g,1}^\times$ is the universal Poincaré torsor over $\mathbb A_g$.
Similarly as above, we have the commutative diagram:
\[\begin{tikzcd}
	\PPP^\times_{g,t} && \\
	& {\AAA_g\times_{\mathbb A_g}(\AAA_g^\vee)^{[t]}} & \AAA_g \\
	& {(\AAA_g^\vee)^{[t]}} & \mathbb A_g
 	\arrow["\pi"', from=1-1, to=2-2]
	\arrow["{\pi_{\ab}}", curve={height=-18pt}, from=1-1, to=2-3]
	\arrow["\pi_{\mr{sab}}"', curve={height=18pt}, from=1-1, to=3-2]
	\arrow["\pi_{\mr{Siegel}}"', curve={height=80pt}, from=1-1, to=3-3]
	\arrow["", from=2-2, to=2-3]
	\arrow["", from=2-2, to=3-2]
	\arrow[from=2-3, to=3-3]
	\arrow[from=3-2, to=3-3]
\end{tikzcd}.\]
The fibration $\pi_{\sab}:\mc P_{g,t}^\times\rightarrow (\mc A_g^\vee)^{[t]}$ parametrizes all semiabelian varieties of toric rank $t$, with the abelian quotient of dimension $g$ carrying principal polarization and level-$4$ structure \cite[Proposition 4.6]{BE}.

\subsection{Betti maps}\label{subsec_Betti_maps}
The action of $P^{\mr{der}}_{g,t}(\R)^+\mr{M}_{t,1}(\C)$ on $\mc X_{g,t}^+$ defined by \eqref{eq_group_action} is transitive. Moreover, we have the bijection
\begin{equation}\label{eq_betti_matrix}
\begin{aligned}
 \mathcal{X}^+_{g,t}&\xlongrightarrow{\varphi_b} \mr M_{t,1}(\C)\times \mr M_{2g,1}(\R)\times \mr{M}_{t,2g}(\R)\times \mathfrak H_g^+,\\
    (z,y,x,h)\cdot(0,0,0,\tau_0)&\longmapsto (z,y,x,h\cdot \tau_0). 
\end{aligned}
\end{equation}
where $(z,y,x,g)\in P^{\mr{der}}_{g,t}(\R)^+\mr{M}_{t,1}(\C)=(\mr{M}_{t,1}(\C)\oplus \mr M_{2g,1}(\R))\rtimes \mr{M}_{t,2g}(\R)\rtimes \mr{Sp}_{2g}(\R)$, and $\tau_0=\sqrt{-1}I_g\in \mathfrak H_g^+.$ Note that $\varphi_b$ identifies $\mc X_{g,t}^+$ as a semi-algebraic space.
\begin{remark}
    Here we add an explanation of the bijection in \eqref{eq_betti_matrix}. For any $\upsilon\in \mc X_{g,t}^+$, there exists $(z,y,x,h)\in P_{g,t}^{\mr{der}}(\R)^+\mr M_{t,1}(\C)$ such that $(z,y,x,h)\cdot (0,0,0,\tau_0)=\upsilon$, and for any other such $(z',y',x',h')$, we have $(z,y,x)=(z',y',x').$ 
    This is the special case of the general argument in \cite[(4.1)]{ZGAxLindemann}.
\end{remark}

We define the \emph{universal Betti-map} $\widetilde{b}=(\widetilde{b}_{\tor},\widetilde b_{\ab})$ to be the composition of $\varphi_b$ with the projection to $\mr{M}_{t,1}(\C)\oplus \mr M_{2g,1}(\R).$ 
Here $\widetilde b_{\tor}$ denotes the further projection to $\mr M_{t,1}(\C)$, and $\widetilde b_\ab$ denotes the $\mr M_{2g,1}(\R)$ part.

We provide the following explicit formula:
\begin{equation}\label{eq_Betti_map}
\begin{aligned}
 \mathcal{X}^+_{g,t}&\longrightarrow \C^t\times \R^{2g},\\
    (w,v,u,\tau)&\longmapsto (w-u\mr{Im}(\tau)^{-1}\mr{Im}(v),(\mr{Re}(v)-\mr{Re}(\tau)\mr{Im}(\tau)^{-1}\mr{Im}(v),-\mr{Im}(\tau)^{-1}\mr{Im}(v))). 
\end{aligned}
\end{equation}
Hence $\widetilde b_{\ab}$ is the pull-back of the Betti map on $\mr{M}_{g,1}(\C)\times \mathfrak H_g^+$ (see \cite[\S3 and \S4]{ZG1}).
\begin{remark}
    Our formula for $\widetilde b_\ab$ is slightly different with \cite{ZG1} since we use the real coordinate $v=y_1-\tau y_2$ instead of $v=y_1+\tau y_2$. 
\end{remark}

The following proposition will be used in the subsequent section.
\begin{proposition}\label{prop_betti_equivariant}
The universal Betti map $\widetilde b$ is $P_{g,t}(\R)^+\mr{M}_{t,1}(\C)$-equivariant in the sense that for any $q=(z,y,x,h)\in (\mr{M}_{t,1}(\C)\oplus \mr M_{2g,1}(\R))\rtimes \mr{M}_{t,2g}(\R)\rtimes \mr{GSp}_{2g}(\R)^+=P_{g,t}(\R)^+\mr{M}_{t,1}(\C)$, and $\upsilon\in \mc X_{g,t}^+$,
\[{ \begin{aligned} \widetilde{b}_{\ab}(q\cdot\upsilon) &= h \widetilde{b}_{\ab}(\upsilon)+y, \\ \widetilde{b}_{\tor}(q\cdot\upsilon) &= \mu(h) \, \widetilde{b}_{\tor}(\upsilon)+x \widetilde{b}_{\ab}(\upsilon) + z.  \end{aligned}}\]
In particular, the map $\C^t\times \R^{2g}\rightarrow \C^t\times \R^{2g}, \widetilde b(\upsilon)\mapsto \widetilde b(q\cdot \upsilon)$ is a diffeomorphism.
\end{proposition}
\begin{proof}
    The formula follows from direct matrix multiplication computation via \eqref{eq_betti_matrix}. We may also refer to \cite[pp. 15]{ZGAxLindemann} for a general discussion. Therefore for any $q$, the map $\widetilde b(\upsilon)\mapsto \widetilde b(q\cdot\upsilon)$ is differentiable. As it has the obvious inverse $\widetilde b(\upsilon)\mapsto \widetilde b(q^{-1}\cdot\upsilon)$ which is also differentiable, we conclude the proof.
\end{proof}

A consequence is that the map $\widetilde b$ descends to a map $b=(b_\tor,b_\ab)$ from $\mc P_{g,t}^\times \times_{(\mc A_g^{\vee})^{[t]}} (\mr{M}_{t,2g}(\R)\times \mathfrak H_g^+)=(\mr{M}_{t,1}(\Z)\oplus \mr M_{2g,1}(\Z))\backslash \mc X_{g,t}^+$ to $(\C^\times)^t\times \R^{2g}/\Z^{2g}$. Moreover, $b$ satisfies the conditions (B1-3) of Betti maps in the beginning of this section.

Now we consider a principally polarized semiabelian scheme $\mc G$ over a quasi-projective variety $S$ over $\C$. Then up to a finite cover, we have the modular map 
\[
\begin{tikzcd}
  \mc G \ar[d] \ar[r,"\iota"] & \mc P_{g,t}^\times \ar[d] \\
  S \ar[r,"\iota_S"] &  (\mc A_g^\vee)^{[t]}
  \ar[ul, phantom, "\ulcorner", pos=0.45] 
\end{tikzcd}
\]
Then $\iota_S$ lift to $\widetilde S\rightarrow \mr M_{t,2g}(\R)\times \mathfrak H_g^+$ where $\widetilde S$ is the universal covering of $S(\C)$ (viewed as a complex analytic space). By abuse of notation, we still denote by $b$ its the pull-back to $\mc G_{\widetilde S}.$

\subsection{Betti currents}\label{subsection_betti_currents}
In this subsection, for simplicity, we use $\wedge$ to denote both wedge product and matrix multiplication as in \cite[Lemma 2.3]{DGH}.

We first review the Betti form used in previous works on abelian scheme. Here we call it the \emph{abelian Betti form} $\omega_{\ab}$ on $\mc A_g$, which is descended from the following \emph{universal abelian Betti form}:
$$\widetilde\omega_{\ab}:=2\sqrt{-1}\partial\bar\partial(\mr{Im}(v^\top)\mr{Im}(\tau)^{-1}\mr{Im}(v))$$
on $\mr{M}_{g,1}(\C)\times \mathfrak H_g^+ \ni (v,\tau).$ Here we list some properties of $\omega_{\ab}$:
\begin{itemize}
    \item This is a semipositive $(1,1)$-form on $\mc A_g$ \cite[Lemma 2.3]{DGH};
    \item If we denote by $[N]$ the multiplication by $N$ map on $\mc A_g$, then $[N]^*\omega_\ab=N^2\omega_\ab$ \cite[Lemma 2.6]{DGH}. 
    \item The form $\widetilde \omega_\ab$ is invariant under $\mr M_{2g,1}(\R)\rtimes \mr{Sp}_{2g}(\R)^+$ \cite[Proof of Lemma 2.6]{DGH}. This guarantees the descent to $\omega_\ab$.
\end{itemize}

In the rest of this article, we also denote by $\omega_\ab$ its pull-back to $\mc P_{g,t}^\times$ by abuse of notation.

Now we start to introduce the \emph{toric Betti currents}.
We first define the \emph{universal toric Betti forms} on $\mathcal P_{g,t}^\times \times_{(\mc A_g^\vee)^{[t]}}(\mr M_{t,2g}(\R)\times \mathfrak H_g^\times)$:\begin{align*}
    \widetilde \omega_{\tor, i}^{(p)}&:=\frac{2}{p}\sqrt{-1} \partial \bar\partial \log((|b_{\tor,i}|^{-p}+1)(1+ |b_{\tor,i}|^p)) \quad (i=1,\dots,t),\\
    \widetilde \omega_{\tor}^{(p)}&:=\sum_{1\leq i\leq t} \widetilde \omega_{\tor,i}^{(p)}
\end{align*}
where $b_{\tor,i}$ denotes the composition of $b_\tor$ with the $i$-th projection $\C^t\rightarrow \C$. 

We then define the \emph{universal Betti currents}:
\begin{equation}\label{eq_betti_currents}
     \begin{aligned}
\widetilde \omega_{\tor, i}&:=\sqrt{-1} \partial \bar\partial \big\lvert\log|b_{\tor,i}|\big\rvert\quad (i=1,\dots,t),\\
    \widetilde \omega_{\tor}&:=\sum_{1\leq i\leq t} \widetilde\omega_{\tor,i}.\\
\end{aligned}
\end{equation}
We can see that $\widetilde \omega_{\tor,i}^{(p)}$ (resp. $\widetilde \omega_{\tor}^{(p)}$) converges to $\widetilde \omega_{\tor, i}$ (resp. $\widetilde\omega_{\tor}$) as $p\rightarrow \infty.$ 

Then we have \begin{itemize}
    \item $\widetilde \omega_{\tor,i}$ (resp. $\widetilde \omega_\tor$) descend to a current $\omega_{\tor, i}$ (resp. $\omega_\tor$) on $\mc P_{g,t}^\times$, which are called the \emph{toric Betti currents}.
    \item The currents $\omega_{\tor,i}$ and $\omega_{\tor}$ are differences of semipositive $(1,1)$-currents. 
\end{itemize}
The proof will be given in Proposition \ref{prop_dsh_toric_betti} via the theory of adelic line bundles. 

Finally, we define the \emph{Betti current}:
$$\omega:=\omega_\ab+\omega_\tor.$$

There are some other facts we have to point out.
Notice that we have the real 2-forms on $\R^{2g}/\Z^{2g}$:
\begin{align*}
    \beta_{\ab}&:=2dy_1^\top\wedge dy_2,
\end{align*}
where $y_1$ and $y_2\in \R^g$ such that $v=y_1-\tau y_2$.
By \cite[Lemma 2.4]{DGH},
\begin{equation}\label{eq_currents_pull_back_via_betti}
    \omega_{\ab}=b^*\beta_{\ab}.
\end{equation}
\section{$\mc P_{g,t}^\times$ as a (mixed) Shimura variety}\label{section_Shimura}
In this subsection, we first review the definition of the connected mixed Shimura variety defined by Pink \cite{PinkThesis} and explain the Shimura data for $\mc P_{g,t}^\times$ following \cite{BE}.

As building blocks of bi-algebraic topology (we will explain this in section \ref{section_degeneracy}) on Shimura varieties and of unlikely intersection problems, the special subvarieties are of main concern. We will give a characterization of special subvarieties of a certain type in \S \ref{subsec_special_sub}.

We remark that in our case the Ribet section (see \cite[\S 5]{BE} for the definition) is excluded. For general discussion on special subvarieties of $\mc P_{g,t}^\times$, we refer to \cite{GuHuang}.

Moreover, we will discuss some specific case of the Shimura quotient by a normal group in \S\ref{subsec_normal_subgrp}, which will be used in our criterion for non-degeneracy.
\subsection{Mixed Shimura varieties}
A \emph{mixed Shimura datum} $(P,\mc X,h)$ consists of 
\begin{itemize}
    \item a connected linear algebraic group $P$ over $\Q$;
    \item a left homogeneous space $\mc X$ under the subgroup $P(\R)U(\C)\subset P(\C)$, where $U\subset P$ is a normal subgroup (which will be called the weight $-2$ part);
    \item $h:\mc X\hookrightarrow \mr{Hom}(\mathbb S,P_\C)$ is a $P(\R)U(\C)$-equivariant map, where $\mathbb S:=\mr{Res}_{\C/\R}(\mathbb G_m)$ is the Deligne torus.
\end{itemize} such that the conditions in \cite[Definition 2.1]{PinkThesis} are satisfied. We omit $h$ if there is no ambiguity. We then say the mixed Shimura datum $(P,\mc X^+)$ is \emph{connected} if $\mc X^+\xhookrightarrow{h} \mr{Hom}(\mathbb S,P_\C)$ is an orbit under the subgroup $P(\R)^+U(\C)\subset P(\C),$ where $P(\R)^+$ is the connected component of $P(\R)$. 

A \emph{connected mixed Shimura variety} $M$ associated with $(P,\mc X^+)$ is of form $\Gamma\backslash\mc X^+$ for some congruence subgroup $\Gamma\subset P(\Q)\cap P(\R)_+$ where $P(\R)_+$ is the stabilizer of $\mc X^+\subset \mr{Hom}_\C(\mathbb S_\C,P_\C).$ We will denote by $\uni:\mc X^+\rightarrow M$ the uniformization map for any Shimura varieties that show up in this article, if there is no ambiguity. Hence our $\mc P_{g,t}^\times$ is a connected mixed Shimura variety with the Shimura datum $(P_{g,t},\mc X_{g,t}^+)$ such that $\mr M_{t,1}< P_{g,t}$ is the weight $-2$ part.

A \emph{Shimura morphism} $(P,\mc X^+)\rightarrow (P',{\mc X'}^+)$ of Shimura data is a homomorphism $f: P\rightarrow P'$ of algebraic groups over $\Q$ which induces a map $\mc X^+\rightarrow \mc X'^+, x\mapsto f\circ x.$ If $f$ is further injective, we say $(P,\mc X^+)$ is a \emph{Shimura subdatum} of $(P',\mc X'^+).$

We denote by $\mc R_u(P)$ the unipotent radical of $P$.
By weight reason, we have the following:
\begin{proposition}\cite[Proposition 2.9]{ZGAxLindemann}\label{prop_weight_decomp}
    Let $f:(P,\mc X^+)\rightarrow (P',\mc X‘^+)$ be a Shimura morphism of connected mixed Shimura data. Then
    $$f(U)\subset U'\quad\text{and}\quad f(\mc R_u(P))\subset \mc R_u(P'),$$
    where $U'$ is the normal subgroup of $P'$ of weight $-2$ part.
\end{proposition}

If $N$ is a normal subgroup of $P$, Pink constructed the \emph{quotient mixed Shimura datum} $(P,\mc X^+)/N$ in \cite[2.9]{PinkThesis}. The underlying group is $P/N$, and the underlying space is simply denoted as $\mc X^+/N$. This induces an associated \emph{quotient mixed Shimura variety} $M_{/N}$ and the quotient map $\rho_{/N}:M\rightarrow M_{/N}$ which is called the \emph{Shimura quotient}.




\subsection{Special subvarieties}\label{subsec_special_sub}

Let $Q$ be a connected algebraic subgroup of $P_{g,t}$. We denote that 
\begin{itemize}
    \item $R_Q:=Q\cap \big((\mr{M}_{t,1}\oplus \mr M_{2g,1})\rtimes \mr M_{t,2g}\big)$
    \item $G_Q:=Q/R_Q=p_{\mr{Siegel}}(Q)$.
    \item $U_Q:=Q\cap \mr M_{t,1}$
    \item $V_Q:=p_{\pi}(Q)\cap (\mr M_{2g,1}\oplus \{0\})=(Q/U_Q)\cap (\mr M_{2g,1}\oplus \{0\}),$
    \item $V^{\ab}_Q:=p_{\ab}(Q)\cap \mr M_{2g,1},$ 
    \item and $W_Q:=p_{\sab}(Q)\cap \mr{M}_{t,2g}.$
\end{itemize}

Now we further assume that $(Q,\mc Y^+)$ is a connected mixed Shimura subdatum of $(P_{g,t},\mc X^+_{g,t})$. The connected mixed Shimura subvariety $M\subset \mc P_{g,t}^\times$ associated with $(Q,\mc Y^+)$ is called a \emph{special subvariety}.

By Proposition \ref{prop_weight_decomp}, we have \begin{itemize}
    \item The unipotent radical $\mc R_u(Q)=R_Q$;
    \item The weight $-2$ part of $Q$ is $U_Q$;
    \item The reductive part of $Q$ is $G_Q$.
\end{itemize}

Since $Q$ is a connected algebraic group over $\Q$, by \cite[Theorem 2.7]{Platonov1994AlgGrp}, we have the Levi decomposition $\mc R_u(Q)\rtimes G_Q,$ where the section $G_Q\rightarrow Q$ is given by
$$(0,g)\mapsto (q_0,0)(0,g)(q_0,0)^{-1}$$
where $g\in G_Q$ and $q_0\in ((\mr{M}_{t,1}\oplus \mr M_{2g,1})\rtimes \mr{M}_{t,2g})(\Q).$



In this subsection, we aim to prove the following:
\begin{theorem}\label{theorem_special_subvariety}
    Assume that \begin{equation}\label{eq_split_unipotent_radical}
        V^{\ab}_Q=V_Q,\text{ i.e., } \mc R_u(p_{\pi}(Q))=V^{\ab}_Q\oplus W_Q.
    \end{equation}
    Then $M\rightarrow \pi_{\sab}(M)$ is a torsion translation of semiabelian subscheme of $\mc P_{g,t}^\times|_{\pi_{\sab}(M)}.$
\end{theorem}
We first prove the following straightforward fact:
\begin{lemma}\label{lemm_split_unipotnent_radical}
    The condition \eqref{eq_split_unipotent_radical}, is equivalent to $$\pi(M)=\pi_{\ab}(M)\times_{\pi_{\mr{Siegel}}(M)}\pi_{\sab}(M).$$
\end{lemma}
\begin{proof}
    Consider the quotient $\mc A_1:=(\mc A_g|_{\pi_{\mr{Siegel}}(M)})\times_{\pi_{\mr{Siegel}}(M)}{\pi_{\sab}(M)}\rightarrow \mc A_2:=\mc A_g|_{\pi_{\mr{Siegel}}(M)}$, which is a Shimura quotient given by the data:
    \begin{align*}
        &((\mr M_{2g,1}\oplus W_Q)\rtimes G_Q, (p_{\pi}(q_0)+\mr M_{2g,1}(\R)\oplus W_Q(\R))\times \widetilde\pi_{\mr{Siegel}}(M))\\&\kern 12em\longrightarrow (\mr M_{2g,1}\rtimes G_Q, (p_{\ab}(q_0)+\mr M_{2g,1}(\R))\times \widetilde\pi_{\mr{Siegel}}(M))
    \end{align*}
    where the section $G_Q\rightarrow (\mr M_{2g,1}\oplus W_Q)\rtimes G_Q$ is a conjugation by $p_{\pi}(q_0).$ Then $\pi(M)\subset \mc A_1$ is the preimage of $\pi_{\ab}(M)\subset \mc A_2$ by condition \eqref{eq_split_unipotent_radical}.

    Conversely, since $\pi(M)$ is a connected mixed Shimura subvariety of $\mc A_g\times_{\mathbb A_g} (\mc A_g^{\vee})^{[t]},$ associated with the Shimura subdatum $(p_{\pi}(Q),\widetilde \pi(\mc Y^+)),$ by \cite[Propoosition 3.4 and its proof]{ZG2017}, we have that
    \begin{itemize}
        \item The group $p_{\pi}(Q)$ is the conjugate of $(V_Q\oplus W_Q)\rtimes G_Q < (\mr M_{2g,1}\oplus \mr M_{t,2g})\rtimes \mr{GSp}_{2g}$ by $(p_{\pi}(q_0),1);$
        \item Via the semi-algebraic identification \eqref{eq_betti_matrix},
        $$\widetilde \pi(\mc Y^+)=(p_{\pi}(q_0)+V_Q(\R)\oplus W_Q(\R))\times \widetilde \pi_{\mr{Siegel}}(\mc Y^+).$$
    \end{itemize}
    Similarly, we have $\widetilde\pi_{\sab}(\mc Y^+)=(p_{\sab}(q_0)+W_Q(\R))\times \widetilde \pi_{\mr{Siegel}}(\mc Y^+).$
    By the condition \eqref{eq_split_unipotent_radical}, we see  
    that $\pi(M)=\pi_{\ab}(M)\times_{\pi_{\mr {Siegel}}(M)}\pi_{\sab}(M).$
\end{proof}

The following is a well-known but useful fact. We include the proof for the reader's convenience.
\begin{lemma}\label{lemma_abelian_section}
    Let $S$ be a projective variety over an algebraically closed field $k$. 
    Let $0\rightarrow \mathbb G_{m,S}^t\rightarrow \mc G\rightarrow \mc A\rightarrow 0$ be a semiabelian scheme. Assume that we have a section $\sigma: \mc A\rightarrow \mc G$ of $\mc A$. Then $\mc G=\mc A\times \mathbb G_m^t$ and $\sigma (\mc A)=\mc A\times\{x\}$ for some $x\in \mathbb G_m^t$.
\end{lemma}
\begin{proof}
    Consider \begin{align*}
        h:\mc A^{[2]}=\mc A\times_S\mc A&\longrightarrow \mathbb G_m^t\\
        (a,b)&\longmapsto \sigma(a+b)\sigma (a)^{-1}\sigma(b)^{-1}
    \end{align*}
    Since $\mc A^{[2]}$ is projective, $h$ has to be a constant map. We set $x\in \mathbb G_m^t$ to be the image.

    After replacing $\sigma$ with its translation by the section $\{-x\}\times S\subset \mathbb G_m^t\times S\subset \mc G$, we may assume that $x=0.$ Then we see that $\sigma$ is a group homomorphism.

    We thus have the homomorphism $\mc G\rightarrow \mc A\times \mathbb G_m^t, g\mapsto (\pi(g),g\sigma(\pi(g))^{-1})$ which concludes the proof.
\end{proof}
\begin{proof}[Proof of Theorem \ref{theorem_special_subvariety}]
    Put $P:=((\mr M_{t,1}\oplus V_Q)\rtimes W_Q)\rtimes G_Q$ and $\mc X^+:=\mc X_{g,t}^+|_{\widetilde \pi(M)}.$ Then $(P,\mc X^+)$ is the connected mixed Shimura datum of $M_P:=\mc P_{g,t}^\times |_{\pi(M)}.$ By Lemma \ref{lemm_split_unipotnent_radical}, $M_P$ is also a semiabelian $\pi_{\sab}(M)$-subscheme of $\mc P_{g,t}^\times|_{\pi_{\sab}(M)}$ up to a torsion translation.

    Let $t'=\dim U_Q$ and $\mathbb G_m^{t'}$ be the subtorus of $\mathbb G_m^t$ given by the saturation of the lattice $\Z^t\cap U_Q$. Since $U_Q$ is normal in $P$ by Lemma \ref{lemm_normality_weight-2},
    we have the commutative diagram of $\pi_{\sab}(M)$-schemes
    \[\begin{tikzcd}
	M_{/U_Q} & (M_P)_{/U_Q}=M_P/\mathbb G_m^{t'}\\
	 & {\pi(M)}
	\arrow[hook, from=1-1, to=1-2]
	\arrow["\mr{id}",from=1-1, to=2-2]
	\arrow[from=1-2, to=2-2]
\end{tikzcd}\]
    where $M_{/U_Q}$ (resp. $(M_P)_{/U_Q}$) is the Shimura quotient of $M$ (resp. $M_P$) by $U_Q$. Since $M$ is the preimage of $M_{/U_Q}$ in $M_P$, it suffices to prove that $M_{/U_Q}$ is a torsion translation of an abelian $\pi_{\sab}(M)$-subscheme of $(M_P)_{/U_Q}.$

    Applying Lemma \ref{lemma_abelian_section} to each fiber of $M_P/\mathbb G_m^{t'}$ over $\pi_{\mr{Siegel}}(M)$, we obtain that \begin{enumerate}
    \item[(i)] $M_P/\mathbb G_m^{t'}$ is isomorphic to $\pi(M)\times \mathbb G_m^{t-t'}$ as semiabelian schemes over $\pi_{\sab}(M)$,
    \item[(ii)] and the inclusion $M_{/U_Q}\hookrightarrow M_P/\mathbb G_m^{t'}$ is a translation of $\pi(M)$ by $\sigma$, where $\sigma:\pi_{\mr{Siegel}}(M)\rightarrow \mathbb G_m^{t-t'}$ is a morphism of varieties.
    \end{enumerate} 

    Now it suffices to prove that $\sigma$ is a constant map to a torsion point.

    Note that the Levi decomposition $P=\mc R_u(P)\rtimes G_Q$ gives that
    $$\mc X^+=(q_0+ U_P(\C)\mc R_u(P)(\R))\times \widetilde\pi_{\mr{Siegel}}(M).$$ 
    
    By (i), we see that $V_Q\oplus W_Q$ is a normal subgroup of $P/U_P$, hence so is its preimage $\mc R_u(Q)$. After taking the quotient of $M_P$ and $M$ by $\mc R_u(Q)$, we obtain that 
    \begin{align*}
        \mc Y^+_{/\mc R_u(Q)}&=\{\widetilde q_0\}\times \widetilde \pi_{\mr{Siegel}}(M),\\
        \mc X^+_{/\mc R_u(Q)}&=(\widetilde q_0+ (U_P/U_Q)(\C))\times \widetilde\pi_{\mr{Siegel}}(M).
    \end{align*}
    Then $\sigma$ is exactly the map $$\pi_{\mr{Siegel}}(M)\simeq\uni(\mc Y^+_{/\mc R_u(Q)})\rightarrow \uni(\mc X^+_{/\mc R_u(Q)}).$$
    Therefore we conclude the proof using the fact that $q_0$ is a $\Q$-point.
\end{proof}
\subsection{Normal subgroups}\label{subsec_normal_subgrp}
We keep the same notation as in \S \ref{subsec_special_sub}. The target is to give a geometric interpretation to the quotient by a normal subgroup. The key here is that this can only be done after we take a modification of the quotiented group in the sense of \eqref{eq_subgroup_closure}.

From now on, let $(Q,\mc Y^+)$ be a connected mixed Shimura subdatum of $(P_{g,t},\mc X_{g,t}^+).$ Assume that the condition \eqref{eq_split_unipotent_radical} holds. Let $N\trianglelefteq Q$ be a connected normal subgroup whose reductive part is semi-simple.

\begin{lemma}
    We have $N\cap (\mr{M}_{t,1}\oplus \mr M_{2g,1})=U_N\oplus V_N <\mr{M}_{t,1}\oplus \mr M_{2g,1}.$
\end{lemma}
\begin{proof}
    Let $V':=N\cap (\mr{M}_{t,1}\oplus \mr M_{2g,1}).$ For any $\widetilde y\in \mc Y^+,$ $\uni(U_N(\C)V'(\R)\widetilde y)$ is a weakly special subvariety of $(\mc P_{g,t}^\times)_s$ for $s=\pi_{\sab}(\uni(\widetilde y))$. By \cite[Proposition 7.2]{ZGAxLindemann}, $\uni(U_N(\C)V'(\R)\widetilde y)$ is a translate of subgroups, hence $V'=U_N\oplus V_N.$
\end{proof}

\begin{proposition}\label{prop_normality_triviality} 
    \begin{enumerate}
        \item $U_N\oplus V^{\ab}_N$ is a normal subgroup of $Q$. Equivalently, it is a $p_{\sab}(Q)$-submodule of $U_Q\oplus V_Q.$
        \item $p_{\sab}(N)$ acts trivially on $(U_Q\oplus V_Q)/(U_N\oplus V^{\ab}_N).$
    \end{enumerate}
\end{proposition}
\begin{proof}
    (1) The Levi-decomposition $Q=R_Q\rtimes G_Q$ induces a Levi decomposition $W_Q\rtimes G_Q$ of $p_{\sab}(Q)$. 

    By \cite[Proposition 3.5]{ZG2017}, $V^{\ab}_N$ is $G_Q$-equivariant. By the representation of $P_{g,t}$, we have the following diagram 
    \[\begin{tikzcd}
	G_Q\times (U_Q\oplus V_Q)&G_Q\times U_Q & \\
	U_Q\oplus V_Q & U_Q
	\arrow["\mathrm{id}\times{\mr{pr}_{U_Q}}", from=1-1, to=1-2]
	\arrow[from=1-1, to=2-1]
    \arrow["{\mr{pr}_{U_Q}}",from=2-1, to=2-2]
	\arrow[from=1-2, to=2-2]
    \end{tikzcd}\]
    where the vertical arrows are group actions of $G_Q$.
    Hence $U_N\oplus V_N^{\ab}$ is $G_Q$-equivariant by the normality of $U_N$ due to Lemma \ref{lemm_normality_weight-2}.
    
    Now it suffices to prove that $U_N\oplus V^{\ab}_N$ is $W_Q$-equivariant, i.e., $U_N\oplus V^{\ab}_N$ is normal in $R_Q$. 
    Since the action of $W_Q$ on $V_Q$ is trivial, the action of $W_Q$ on $(U_Q\oplus V_Q)$ is given by $$x\cdot (z,y)\longmapsto (z+f_x(y),y)$$
    where $(x \in W_Q)\mapsto (f_x\in \mr{Hom}(V_Q,U_Q))$ is a group homomorphism.
    Hence it suffices to prove that $f_x(y)\in U_N$ for any $y\in V^{\ab}_N$ and $x\in W_Q$.

    Since $V^{\ab}_N=p_{\ab}(N)\cap \mr{M}_{2g,1}=p_{\ab}(N\cap \mc R_u(Q))$, there exists $n=(z_0,y,x_0)\in N\cap \mc R_u(Q)\subset (U_Q\oplus V^{\ab}_Q)\rtimes W_Q$ such that $p_{\ab}(n)=y.$
    Then $$(0,0,x)n(0,0,x)^{-1}=(z_0+f_x(y),y,x+x_0)\cdot (0,0,-x)=(z_0+f_x(y),y,x_0)$$
    which exactly gives that $f_x(y)\in U_N.$

    (2) Since $p_{\ab}(N)\trianglelefteq p_{\ab}(Q)$, $G_N$ acts trivially on $V_Q/V_N^\ab$. Similarly as in (1), by the representation of $P_{g,t}$, $G_N$ acts trivially on $(U_Q\oplus V_Q)/(U_N\oplus V_N^\ab).$
    
    Now it suffices to prove that $W_N$ acts trivially on $(U_Q\oplus V_Q)/(U_N\oplus V^{\ab}_N),$ i.e.,
    for any $x\in W_N$ and $(z',y')\in U_Q\oplus V_Q$, \begin{equation}\label{eq_poincare_pairing}
        x\cdot (z',y')-(z',y')\in U_N\oplus V^{\ab}_N\Leftrightarrow f_{x}(y')\in U_N.
    \end{equation}

    Since $W_N=p_{\sab}(\mc R_u(N))\subset (U_Q\oplus V_Q)\rtimes W_Q$, there exists an element $n=(z,y,x)\in \mc R_u(N)$. Let $q=(0,y',0)\in Q$. By the normality of $N$, the commutator
    \begin{align*}
        nq n^{-1}q^{-1}=(f_x(y'),y',0)q^{-1}=(f_x(y'),0,0)\in U_N,
    \end{align*}
    which concludes the proof.
\end{proof}
In the following, we set \begin{equation}\label{eq_subgroup_closure}
    N':=(U_N\oplus V^{\ab}_N)\cdot N
\end{equation}
which is, by the normality of $U_N\oplus V_N^\ab$, naturally a normal subgroup of $Q$ as well.

Now we explain the Shimura quotient of $M$ by $N'$ in two steps.
The first is a sequel of Proposition \ref{prop_normality_triviality}(1). We first recall that in \eqref{eq_betti_matrix}, we have the semi-algebraic identification of $\mc X_{g,t}^+$ with $\mr M_{t,1}(\C)\times \mr M_{2g,1}(\R)\times \mr M_{t,2g}(\R)\times \mathfrak H_g^+.$ After translations via the matrix multiplication, we have semi-algebraic identifications of
$\widetilde \pi_{\sab}(\mc Y^+)$ with $W_Q(\R)\times \widetilde \pi_{\mr{Siegel}}(M)$, and 
$\mc Y^+$ with $U_Q(\C)\times V_Q^\ab(\R)\times \widetilde \pi_{\sab}(M).$
\begin{corollary}\label{coro_semiabelian_subscheme}
    Let $M_0:=\uni(U_N(\C)\times V^{\ab}_N(\R)\times \widetilde\pi_{\sab}(\mc Y^+)).$ Then $M_0$ is a semiabelian $\pi_{\sab}(M)$-subscheme of $M$.
\end{corollary}
\begin{proof}
    Notice that $\pi(M_0)=\uni(V^{\ab}_N(\R)\times \widetilde \pi_{\sab}(\mc Y^+))$ is an abelian $\pi_{\sab}(M)$-subscheme of $\pi(M)$ by \cite[Rappels 4.4.3]{PD}.
    Then $M_R:=\pi|_M^{-1}(\pi(M_0))$ is a semiabelian $\pi_{\sab}(M)$-subscheme of $M$.
    
    Taking the Shimura quotient of $M$ by $U_N$, we obtain the following commutative diagram:
    \[\begin{tikzcd}
	\rho_{/U_N}(M_0)& \rho_{/U_N}(M_R) &M_{/U_N}\\
	 & \pi(M_0) & {\pi(M)}
	\arrow[hook, from=1-1, to=1-2]
    \arrow[hook, from=1-2, to=1-3]
	\arrow["\mr{id}",from=1-1, to=2-2]
	\arrow[from=1-2, to=2-2]
    \arrow[hook, from=2-2, to=2-3]
    \arrow[from=1-3, to=2-3]
\end{tikzcd}.\]
Applying Lemma \ref{lemma_abelian_section} to each fiber of $\rho_{/U_N}(M_R)$ over $\pi_{\mr{Siegel}}(M)$, we obtain that \begin{enumerate}
    \item[(a)] $\rho_{/U_N}(M_R)=\pi(M_0)\times \mathbb G_m^{t_0}$ where $t_0=\dim U_Q-\dim U_N$;
    \item[(b)] $\rho_{/U_N}(M_0)$ is a translation of $\pi(M_0)$ by $\sigma_N$, where $\sigma_N:\pi_{\mr{Siegel}}(M)\rightarrow \mathbb G_m^{t_0}$ is a morphism of varieties.
\end{enumerate}
By Proposition \ref{prop_normality_triviality}, $V^{\ab}_N$ is normal in $Q/U_N$. We may then further take the quotient of $M_{/U_N}$ by $V^{\ab}_N$. Then by a similar argument as in the proof of Theorem \ref{theorem_special_subvariety}, we obtain that $\sigma_N$ is a constant map to $0$. Therefore $\rho_{/U_N}(M_0)$ is an abelian $\pi_\sab(M)$-subscheme of $\rho_{/U_N}(M)$, its preimage $M_0$ is thus a semiabelian $\pi_{\sab}(M)$-subscheme of $M_R$, which concludes the proof.
\end{proof}

Let $M':=M/M_0\rightarrow \pi_{\sab}(M)$ where $M/M_0$ is understood as a quotient of group schemes over $\pi_{\sab}(M).$ Note that $M'$ is also the Shimura quotient of $M$ by $U_N\oplus V^{\ab}_N.$ Hence the quotient $\rho_{/N'}$ factors through $M'$.

On the other hand, we have the Shimura quotient $$\rho_{/p_{\sab}(N)}:\pi_{\sab}(M)\rightarrow \pi_{\sab}(M)_{/p_{\sab}(N)}.$$
\begin{proposition}
    For any $y_\sab\in \pi_{\sab}(M)_{/p_{\sab}(N)}$, $$M'|_{\rho_{/p_{\sab}(N)}^{-1}(y_{\sab})}\rightarrow \rho_{/p_{\sab}(N)}^{-1}(y_{\sab})$$ is an isotrivial semiabelian scheme.
\end{proposition} 
\begin{proof}
    Let $S$ be the image of $\rho_{/p_{\sab}(N)}^{-1}(y_{\sab})$ in $\mathbb A_g,$ $\mc A=\pi_{\ab}(M)|_S$, $\mc B=\rho^{-1}_{/p_{\sab}(N)}(y_\sab)$, and $\mc G=\mc P_{g,t}^\times|_{\mc A\times_S\mc B}.$ Then $\mc G':=M_0|_{\mc B}$ is also a semiabelian subscheme of $\mc G$ over $\mc B$, i.e., we have the following commutative diagram
    \begin{equation*}
    \begin{tikzcd}
    0 \arrow[r] & T \arrow[d] \arrow[r] & \mc G' \arrow[d] \arrow[r] & \mc A'\times_S\mc B \arrow[d] \arrow[r] & 0 \\
    0 \arrow[r] & \mathbb G_m^t \arrow["\rho_{\tor}",d]  \arrow[r] & \mc{G} \arrow["\rho",d]  \arrow[r] & \mc A\times_S\mc B \arrow["\rho_{\ab}",d] \arrow[r] & 0,\\
    0 \arrow[r] & \mathbb G_m^t/T \arrow[r] & \mc{G}/\mc G'   \arrow[r] & (\mc A/\mc A')\times_S\mc B  \arrow[r] & 0,
    \end{tikzcd}
    \end{equation*}
    where $\mc A'=\pi_{\ab}(M_0)|_S.$
    Now it suffices to prove that $\mc G/\mc G'$ is isotrivial as $M'|_{\mc B}\subset \mc G/\mc G'$ as semiabelian subscheme.
    
    Let $\mc C$ be the isotrivial part of $\pi_{\sab}(M)|_S\rightarrow S$.
    Since $B$ is a weakly special subset, by \cite[Proposition 3.3]{ZG2017}, it is a translation of an abelian subscheme of $\pi_{\sab}(M)|_S$ by a torsion section and by a constant section of $\mc C\rightarrow S$. 

    By the theorem of the Fixed Part (see \cite[Theorem 1]{YA} or \cite[Corollaire 4.1.2]{PD}), $\mc A/\mc A'$ is isotrivial due to Proposition \ref{prop_normality_triviality}(2). After taking a finite cover of $S$, we may assume that $\mc A/\mc A'=A_0\times S$ for some abelian variety $A_0$ over $\C$. Then up to an isogeny, we may further assume that $\mc A=\mc A'\times A_0$. 
    
    For each $b\in \mc B$, we denote by $\gamma_b$ (resp. $\ovl \gamma_b$) the extension class in $\mr{Ext}^1_{\C}(\mc A_b,\mathbb G_m^t)$ (resp. $\mr{Ext}_\C^1(A_0,\mathbb G_m^t/T)$) associated to $\mc G_b$ (resp. $(\mc G/\mc G')_b$), which satisfies the formula
    $$\rho_{\tor,*}\gamma_b=\rho^*_{\ab}\ovl \gamma_b \in \mr{Ext}^1(\mc A_b,\mathbb G_m^t/T).$$
    We first assume that $\mc B$ is an abelian subscheme of $\pi_{\sab}(M)|_S$. Again by Proposition \ref{prop_normality_triviality}(2) (more precisely, the formula \eqref{eq_poincare_pairing}), $$\gamma_b\in \mr{Im}(\mr{Ext}^1_\C(\mc A_b,T)\rightarrow \mr{Ext}^1_\C(\mc A_b,\mathbb G_m^t)),$$
    which forces that $\ovl \gamma_b$ is torsion by applying the functor $\mr{Ext}^1(\mc A_b,\cdot)$ to $T\rightarrow \mathbb G_m^t\rightarrow \mathbb G_m^t/T$. 
    
    In the general case, $\ovl\gamma_b$ is independent of $b$ up to torsion (determined by the constant section).
    We then conclude the proof by the Weil--Barsotti formula \cite[\S III.18]{Oort1966group}.
\end{proof}

To conclude, we have the commutative diagram:
\begin{equation}\label{eq_two_steps_quotient}\begin{tikzcd}
	{M} & M'=M/M_0 & {M_{/N'}} \\
	{\pi_{\sab}(M)} & {\pi_{\sab}(M)} & {\pi_{\sab}(M)_{/p_{\sab}(N)}}
	\arrow["\rho_{/(U_N\oplus V_N^{\ab})}", from=1-1, to=1-2]
	\arrow[from=1-1, to=2-1]
	\arrow[from=1-2, to=1-3]
	\arrow[from=1-2, to=2-2]
	\arrow["\rho_{/p_{\sab}(N)}", from=1-3, to=2-3]
	\arrow[from=2-1, to=2-2]
	\arrow[from=2-2, to=2-3]
\end{tikzcd}.\end{equation}

\section{Degeneracy loci}\label{section_degeneracy}
In this section, we give a generalization of Gao's study on generic rank of Betti maps \cite{ZG1}.
\subsection{Weak Optimality and Geometric Zilber--Pink}
In this subsection we recall the definition of weak optimality defined in \cite[Section 7]{ZG1} and a Geometric Zilber--Pink result from \cite{BU}. These will be crucial in our characterization of degeneracy loci in the next subsection.

Let $M$ be a Shimura variety associated with the mixed Shimura datum $(P,\mc X^+).$ 
By \cite[Proposition 4.1]{ZGAxLindemann}, there exists a complex algebraic variety $\mc X^\vee$ such that we have an embedding $\mc X^+\hookrightarrow \mc X^\vee$ as an Euclidean open and semi-algebraic subset. We first recall some facts of the bi-algebraic geometry on $\mc X^+.$
\begin{definition}\label{def_algebraicity_uni_cover}
Let $\tilde{Y}$ be an analytic subvariety of $\mc X^+$.
\begin{enumerate}
\item $\widetilde{Y}$ is called an \emph{irreducible algebraic subset} of $\mc X^+$ if it is a complex analytically irreducible component of $Z\cap \mc X^+$ for some algebraic closed subset $Z\subset \mc X^\vee$;
\item $\widetilde{Y}$ is called \emph{algebraic} if it is a finite union of irreducible algebraic subsets of $\mc X^+$;
\end{enumerate}
\end{definition}

\begin{definition}
    \begin{enumerate}
        \item An irreducible subvariety $Z$ of $M$ is said to be \emph{bi-algebraic} if one (and hence any) complex analytic irreducible component of $\uni^{-1}(Z)$ is algebraic in $\mc X^\vee$;
        \item For any subset $\Omega\subset M$, we denote by $\Omega^{\biZar}$ the smallest bi-algebraic subset containing $\Omega$. Define the \emph{weakly defect} of $\Omega$ as $$\delta_{\mr{ws}} (\Omega)=\dim(\Omega^\biZar)-\dim(\Omega^\Zar);$$
        \item Let $Y$ be a subvariety of $M$. We say an irreducible Zariski closed subvariety $Z\subset Y$ is \emph{weakly optimal} if for any irreducible Zariski closed subset $Z'\supsetneq Z$ of $Y$, $\delta_{\mr{ws}}(Z)<\delta_{\mr{ws}}(Z').$
    \end{enumerate}
\end{definition}
We recall the definition of weakly special subsets of $M$ in \cite[Definition 4.1.(b)]{Pink2005}, more precisely, the equivalent definition due to \cite[\S 5.1]{ZGAxLindemann}.
\begin{definition} \label{def_Weakly_Special}
A subvariety $Y$ of $M$ is called \emph{weakly special} if there exist\begin{enumerate}
    \item  a connected mixed Shimura subdatum $(Q,\mc Y^+)$ of $(P,\mc X^+)$,
    \item a connected normal subgroup $N$ of $Q$ whose reductive part is semisimple,
    \item and a point $\tilde{y}\in\mc Y^+$
\end{enumerate}
such that $Y=\uni(N(\R)^+U_N(\C)\tilde{y})$, where $U_N=N\cap U$ and $U$ is the normal subgroup of $P$ of weight $-2$ part in \cite[Definition 2.1]{PinkThesis}. Equivalently, $Y$ is a fiber of the Shimura quotient map $\rho_{/N}.$

Note that weakly special subsets are precisely the bi-algebraic subsets due to \cite{UY2011Char} and \cite[Corollary 8.3]{ZGAxLindemann}
\end{definition}
\begin{remark}
Let $(Q,\mc Y^+)$ be a connected mixed Shimura subdatum of $(P_{g,t},\mc X^+_{g,t}),$ and $N$ be a normal subgroup of $Q$. Then $U_N$ in Definition \ref{def_Weakly_Special} is exactly $N\cap \mr{M}_{t,1}.$
\end{remark}

Now we recall the geometric Zilber--Pink Theorem due to \cite{BU}. In the case of $Y(1)^N$, this is obtained by Habegger--Pila \cite{HP2016atypical}, which is generalized to any pure Shimura varieties due to Daw--Ren \cite{DawRen2018Ax}. Gao proved the geometric Zilber--Pink in the case of connected mixed Shimura variety of Kuga type \cite{ZGMixedAF}, under the name of finite result à la Bogomolov.
\begin{theorem}\cite[Theorem 7.1]{BU}\label{Geometric_Zilber_Pink}
Let $Y\subset M$ be a subvariety. There exists a finite subset $$\Sigma_Y\subset\left\{(Q,\mathcal Y^+,N):\begin{aligned}
(Q,\mc Y^+)\text{ is a connected mixed Shimura}\\\text{ subdatum of }(P,\mc X^+),\text{ and } N \trianglelefteq Q\end{aligned}\right\}$$
such that for any weakly optimal subvariety $Z$ of $Y$, there exist $(Q,\mc Y^+,N)\in \Sigma_Y$ and $\widetilde y\in \mc Y^+$, such that 
$$Z^\biZar=\uni(N(\R)^+U_N(\C)\widetilde y).$$
\end{theorem}
\subsection{Degeneracy loci}
We will generalize \cite[\S 6-7]{ZG1} to $\mc P_{g,t}^\times$. 


We first fix some conventions and notations:
\begin{itemize}
    \item For any irreducible subset $Y$ in a connected mixed Shimura variety, we denote by $\widetilde Y$ any irreducible component of $\uni^{-1}(Y).$ Moreover, if $x\in Y$, we denote by $\widetilde y$ any point in $\widetilde Y$ such that $\uni(\widetilde y)=x$.
    \item For a weakly special subset $Y$, we sometimes write $Y=\uni(N(\R)^+U_N(\C)\widetilde y)$ without claiming the ambient Shimura subdatum $(Q,\mc Y^+)$ if there is no ambiguity.
\end{itemize}

\begin{definitionproposition}
    Let $Y=\uni(N(\R)^+U_N(\C)\widetilde y)\subset \mc P^\times_{g,t}$ be a weakly special subset.
    Then half of the \emph{the Betti rank} $$\frac{1}{2}\max_{\widetilde y\in \widetilde Y}\mr{rank}_\R(d\widetilde b|_{\widetilde Y})_{\widetilde y}$$ is independent of the choice of $\widetilde y$, which we denote as $h_N$. 
\end{definitionproposition}
\begin{proof}
    For any $\widetilde y'\in \mc Y^+$, we have $q\cdot \widetilde y=\widetilde y'$ for some $q\in Q(\R)^+U_Q(\C)$. 
    Consider the map $\varphi:\mc X_{g,t}^+\rightarrow \mc X^+_{g,t}, x\mapsto q\cdot x.$
    Note that $$N(\R)^+U_N(\C)\widetilde y'=q N(\R)^+U_N(\C)q^{-1} q\cdot \widetilde y=qN(\R)^+U_N(\C)\widetilde y=\varphi(\widetilde Y)$$
    where the first equality is due to the normality of $N.$
    By the equivariance of the Betti map due to Proposition \ref{prop_betti_equivariant}, we obtain the diffeomorphism
    $\widetilde {\varphi}:\C^{t}\times \R^{2g}\rightarrow \C^{t}\times \R^{2g}$ such that $\widetilde b\circ \varphi=\widetilde\varphi\circ\widetilde b$, which concludes the proof.
\end{proof}

\begin{definition}
    Let $Y\subset \mathcal P_{g,t}^\times$ be a closed subvariety.
    We define the $r$\emph{-th degeneracy locus} $Y^{\deg}(r)$ as the union of all weakly optimal subvarieties $Z$ of $Y$ such that \begin{equation}\label{eq_defect_condition}\dim Z> h_N+r\end{equation}
    where $Z^\biZar=\uni(N(\R)^+U_N(\C)\widetilde y)$ for some $(Q,\mc Y^+,N)\in \Sigma_Y$ and $\widetilde y\in \mc Y^+.$
\end{definition}
We first prove the Zariski-closedness of the degeneracy loci.
\begin{lemma}\label{lemma_biZar_weakly_optimal}
    Let $Z\subset Y$ be a weakly optimal subvariety. Then $Z$ is an irreducible component of $Z^{\biZar}\cap Y$.  
\end{lemma}
\begin{proof}
    Let $Z'$ be an irreducible component $Z^\biZar\cap Y$ containing $Z$.
    Since $Z\subset Z'\subset Z^\biZar$, we obtain that $Z^\biZar\subset Z'^\biZar\subset Z^\biZar$, i.e., $Z^\biZar=Z'^\biZar$. 
    Hence $\delta_{\mr{ws}}(Z')=\dim (Z^{\biZar})-\dim Z'$. Assume that $Z\subsetneq Z'$, then $\delta_{\mr{ws}}(Z')>\delta_{\mr{ws}}(Z)$ by the weak optimality. Therefore $\dim Z'<\dim Z$, which is a contradiction. We thus obtain that $Z'=Z.$
\end{proof}

\begin{proposition} We have the equality
    \begin{equation}\label{eq_deg_loci_closed}
        \begin{aligned}
        Y^{\deg}(r)&=\bigcup_{(Q,\mc Y^+,N)\in \Sigma_Y} \{x\in Y\cap M:\dim_x  \rho^{-1}_{/N}\rho_{/N}(x)\cap (Y\cap M)> h_N+r \}
        \end{aligned}
    \end{equation}
    where $M$ is the Shimura subvariety given by $(Q,\mathcal Y^+)$ and $\rho_{/N}:M\rightarrow M_{/N}$ is the Shimura quotient map.
    In particular $Y^{\deg}(r)$ is Zariski-closed.
\end{proposition}
\begin{proof}
    Since $\dim_x  \rho^{-1}_{/N}\rho_{/N}(x)\cap (Y\cap M)$ is exactly the dimension of $\uni(N(\R)^+ U_N(\C)\widetilde y)\cap Y$'s irreducible component $Y'$ which contains $x$, we have 
    $$\text{Right hand side of \eqref{eq_deg_loci_closed}}=\bigcup_{\substack{(Q,\mc Y^+,N)\in \Sigma_Y, \widetilde y\in \mc Y^+, Y'\subset \uni(N(\R)^+ U_N(\C)\widetilde y)\cap Y,\\ Y'{\text{ irreducible and }} \dim Y'>h_N+r}} Y'$$
    
    We thus obtain the direction $"\subset"$ by Theorem \ref{Geometric_Zilber_Pink} and Lemma \ref{lemma_biZar_weakly_optimal}. 

    Conversely, let $Y'$ be an irreducible component of $\uni(N(\R)^+U_N(\C)\widetilde y)\cap Y$ with $\dim Y'\geq h_N+r$. Then $Y'^{\biZar}\subset \uni(N(\R)^+U_N(\C)\widetilde y)$. Therefore the generic Betti rank of $Y'^{\biZar}$ is not larger than $2h_N$, which implies that $Y'$ satisfies the condition \eqref{eq_defect_condition}. This concludes the proof.
\end{proof}

\begin{theorem}\label{theorem_non_deg_shimura_quotient}
    Let $Y\subset \mc P^\times_{g,t}$ be an irreducible closed subvariety. Then $Y^{\deg}(r)=Y$ if and only if there exists a triplet $(Q,\mc Y^+,N)\in \Sigma_Y$ such that
    $\dim Y-\dim \rho_{/N}(Y)> h_N+r.$
\end{theorem}
\begin{proof}
    This is a direct consequence of \eqref{eq_deg_loci_closed} and the irreducibility of $Y$.
\end{proof}

We now relate the degeneracy loci with the generic rank of the Betti map restricted on $\widetilde Y$. We start with a key lemma. 
\begin{lemma}\label{lemm_weakly_optimal_maximality}
    Let $Z\subset Y$ be an irreducible closed subvariety. Assume that $\dim Z>h_N+r$ and $Z$ is maximal for this property. Then $Z$ is weakly optimal.
\end{lemma}
\begin{proof}
    Let $Z'$ be an irreducible closed subvariety of $Y$ such that $Z\subsetneq Z'$ and $\delta_{\mr{ws}}(Z)\geq \delta_{\mr{ws}}(Z')$. Since $$\dim Z^\biZar- h_N\leq \dim Z'^{\biZar}-h_{N'},$$ we obtained that $$\dim Z'^\biZar-\dim Z'\leq\dim Z^{\biZar}-\dim Z<\dim Z^{\biZar}-h_N-r\leq \dim Z'^{\biZar}-h_{N'}-r$$
    which implies that $\dim Z'>h_{N'}+r$, contradicting to the maximality of $Z$.
\end{proof}
\begin{proposition}\label{prop_bettirank_deg} Let $Y\subset \mc P_{g,t}^\times$ be an irreducible closed subvariety. If $r\geq 0$, then
    $Y^\deg(r)=Y$ if and only if $\displaystyle\max_{\widetilde y\in Y^{\mr{sm}}}\mr{rank}_{\R}(\widetilde b|_{\widetilde Y})_{\widetilde y}< 2 (\dim Y-r).$
\end{proposition}
\begin{proof}
    "$\Rightarrow$": Since $Y^{\deg}(r)=Y$, we see that there exists a triplet $(Q,\mc Y^+,N)\subset \Sigma_Y$ such that $\dim Y-\dim \rho_{/N}(Y)>h_N+r$ by Theorem \ref{theorem_non_deg_shimura_quotient}. For any $x\in Y^{\mr{sm}}$, let $Z$ be the irreducible component of $\rho_{/N}^{-1}\rho_{/N}(x)\cap Y$ containing $x$, hence $\dim Z>h_N+r$ for general $x$. 
    Then $$\begin{aligned}\mr{rank}_{\R}(d\widetilde b|_{\widetilde Y})_{\widetilde y}&\leq \mr{rank}_{\R}(d\widetilde b|_{\widetilde Z})_{\widetilde y}+2(\dim Y-\dim Z)\\&\leq \mr{rank}_\R(d\widetilde b|_{N(\R)^+U_N(\C)\widetilde y})_{\widetilde y}+2(\dim Y-\dim Z)\\&=2(h_N+\dim Y-\dim Z)<2(\dim Y-r).\end{aligned}$$

    "$\Leftarrow$": Assume that $\displaystyle\ell:=\max_{\widetilde y\in Y^{\mr{sm}}}\mr{rank}_{\R}(d\widetilde b|_{\widetilde Y})_{\widetilde y}< 2 (\dim Y-r)$. Then for any $x\in Y^{\mr{sm}}$, there exists a complex analytic set $\widetilde W\subset \widetilde b^{-1}(\widetilde b(\widetilde y))\cap \widetilde Y$ such that  $\dim \widetilde W\geq \dim Y-\ell>r$.
    Let $Z:=\uni(\widetilde W)^\Zar$. By the Ax-Schaunuel theorem \cite[Theorem 1.1]{GaoKlingler2024AxSchanuel}, we have 
     $$\dim \widetilde W^{\Zar}+\dim Z\geq \dim \widetilde W+\dim Z^{\biZar}$$
    Since $\widetilde W^{\Zar}\subset \widetilde b^{-1}(\widetilde b(\widetilde{x}))\cap (N(\R)^+U_N(\C)\widetilde y)$, and the latter is algebraic and of dimension $\dim Z^\biZar-h_N$, we get $\dim Z\geq \dim W+h_N>r+h_N.$ By Lemma \ref{lemm_weakly_optimal_maximality}, $Z$ is contained in $Y^{\deg}(r)$, we are done.
\end{proof}
\subsection{Degeneracy criterion}\label{subsection_degeneracy_criterion}
Let $Y\subset \mc P_{g,t}^\times$ be a closed subvariety, and $M$ be the smallest special subvariety containing $Y$.

\begin{theorem}\label{theorem_quotient_subgrp_closure} Assume that $$\pi(M)=\pi_{\ab}(M)\times_{\pi_{\mr{Siegel}}(M)}\pi_{\sab}(M).$$Let $(Q,\mc Y^+)$ be the associated connected mixed Shimura datum of $M$. If $Y^{\deg}(r)=Y$, then there exists a normal subgroup $N\trianglelefteq Q$ with semisimple reductive part satisfying $N=(U_N\oplus V_N^\ab)\cdot N$, such that 
$$\dim Y-\dim \rho_{/N}(Y)>h_N+r.$$
Moreover, \begin{equation}\label{eq_betti_rank}
h_{N}=\dim U_N+\frac{1}{2}\dim V^{\ab}_N,
\end{equation}
which is also the relative dimension of $M_0$ over $\pi_{\sab}(M)$ constructed in Corollary \ref{coro_semiabelian_subscheme}.
\end{theorem}
\begin{proof}
By Theorem \ref{theorem_non_deg_shimura_quotient}, there exists $(Q^*,\mc Y^{+,*},N^*)\in \Sigma_Y$ such that $$\dim Y-\dim \rho_{/N^*}(Y)>h_{N^*}+r.$$
Let $M^*=\uni(\mc Y^{+,*})$. Since $Y\subset M$, we have $M\subset M^*$. With out loss of generality, we may assume that $\mc Y^+\subset \mc Y^{+,*}$ and $Q<Q^*.$ Let $N$ be the neutral component of $Q\cap N^*.$ We may assume that the reductive part of $N$ is semisimple after replacing $N$ with $N\cap p_{\mr{Siegel}}^{-1}(p_{\mathrm{Siegel}}(N)^{\mr{der}})$, which is still normal in $Q$.

We have the following commutative diagram:
\[\begin{tikzcd}
    \widetilde Y &\mc Y^+ &\mc Y^{+,*}\\
	Y & M & M^*
	\arrow[from=1-1, to=1-2]
	\arrow[from=1-1, to=2-1]
	\arrow[hook,from=1-2, to=1-3]
	\arrow[from=1-2, to=2-2]
	\arrow[from=1-3, to=2-3]
	\arrow[from=2-1, to=2-2]
	\arrow[hook,from=2-2, to=2-3]
\end{tikzcd}\]
We denote by $\widetilde \rho_{/N}: \mc Y^+\rightarrow \mc Y^+/N$ (resp. $\widetilde \rho_{/N^*}: \mc Y^+\rightarrow \mc Y^{*,+}/N^*$) the map on the covering space induced by $\rho_{/N}$ (resp. $\rho_{/N^*}$).

Notice that the fiber of $\widetilde\rho_{/N}|_{\widetilde Y}$ are of form $$(U_N(\C)N(\R)^+\widetilde y)\cap \widetilde Y=((U_{N^*}(\C)N^*(\R)^+\widetilde y)\cap \mc Y^+)\cap \widetilde Y=((U_{N^*}(\C)N^*(\R)^+\widetilde y)\cap \widetilde Y,$$
while the fiber of $\widetilde \rho_{/N^*}|_{\widetilde Y}$ take the form of the right hand side. Therefore $\dim \rho_{/N^*}(Y)=\dim \rho_{/N}(Y)$. Since $N<N^*$, by definition, we have $h_N\leq h_{N^*}.$
Hence $$\dim Y-\dim \rho_{/N}(Y)=\dim Y-\dim \rho_{/N^*}(Y)>h_{N^*}+r\geq h_N+r,$$ we may thus replace $(Q^*,\mc Y^{+,*},N^*)$ with $(Q,\mc Y^+, N).$

On the other hand, we set $N':=(U_N\oplus V_N^\ab)\cdot N$ which is normal by Proposition \ref{prop_normality_triviality}.(1). By the definition of $h_N$ and the equivariancy of Betti map due to Proposition \ref{prop_betti_equivariant}, we have \begin{equation*}
    h_{N'}=h_{N}=\dim U_N+\frac{1}{2}\dim V^{\ab}_N.
\end{equation*}
Hence we may further replace $N$ with $N'$, which concludes the proof
\end{proof}

\section{Non-degeneracy of fiber powers}\label{sec_nondeg_fiberproduct}
Let $S$ be a quasi-projective variety over $\C$, and $\mc G$ be a semiabelian scheme given by the exact sequence
$$0\rightarrow \mathbb{G}_{m,S}^t\rightarrow \mathcal G\rightarrow \mathcal A\rightarrow 0$$
of group $S$-schemes, where $\mathcal A/S$ is an abelian scheme of relative dimension $g$. We first define the following ``modular dimension" for the family $\mc G$ of semiabelian varieties.
\begin{definition}
    We denote by $\mr{Mod}.\dim_{\mc G}(S)$ the smallest dimension of quasi-projective variety $S'$ such that after taking a finite cover of $S$, there exists a morphism $f:S\rightarrow S'$ such that for each $s'\in S'$
    $$\mc G|_{f^{-1}(s')}\rightarrow f^{-1}(s')$$
    is an isotrivial family of semiabelian varieties.
\end{definition}

Assume that $\mc A$ has a principal polarization and level-$4$ structure. Then we have the modular map 
\[
\begin{tikzcd}
  \mc G \ar[d] \ar[r,"\iota"] & \mc P_{g,t}^\times \ar[d] \\
  S \ar[r,"\iota_S"] &  (\mc A_g^\vee)^{[t]}
  \ar[ul, phantom, "\ulcorner", pos=0.45] 
\end{tikzcd}
\]
Then we have $\mr{Mod}.\dim(\mc G/S)=\dim \iota_{S}(S).$
\begin{definition}
    Let $\mc X\subset \mc G$ be an irreducible closed subvariety.
    We say $\mc X$ is non-degenerate if for $Y:=\ovl{\iota(\mc X)}$
    $$\max_{\widetilde y\in Y^{\mr{sm}}}\mr{rank}_{\R}(d\widetilde b|_{\widetilde Y})_{\widetilde y}=2\dim \mc X.$$
\end{definition}

In this subsection, we aim to prove that for a suitable subvariety $\mc X$, its large enough fiber power is non-degenerate. Assume that $\mc X\rightarrow S$ is surjective. We recall the notations that
$$\mc X^{[m]}:=\underbrace{\mc X\times_S\cdots\times_S\mc X}_{m\text{ times}},$$
and the modular map for the fiber power: $$\iota^{[m]}:\mc G^{[m]}\rightarrow (\mc P^\times_{g,t})^{[m]}\hookrightarrow \mc P^\times_{mg,mt}.$$
\begin{theorem}\label{theo_nondeg_gen}
    Let $\ovl\eta$ be the geometric generic point of $S$.
    Assume that 
        \begin{enumerate}
        \item $\iota^{[m]}|_{\mc X^{[m]}}$ is generically finite for $m\gg 0$.
        \item $\mc X_s$ generates $\mathcal A_s$ for each $s\in S;$
        \item $\mc X_{\ovl \eta}$ is irreducible and has finite stabilizer.
    \end{enumerate}
    Then $\mc X^{[m]}$ is non-degenerate for every $m\gg 0.$
\end{theorem}
\subsection{Reduction step}
In order to prove Theorem \ref{theo_nondeg_gen}, it suffices to prove the following:
    \begin{theorem}\label{theorem_nondeg_group_quotient}
        We keep the same hypothesis as in Theorem \ref{theo_nondeg_gen}.
        For any $m\gg0$, let $\mc H\rightarrow \mc G^{[m]}$ be a semiabelian subscheme. Then either of the following holds:
    \begin{enumerate}
        \item[(i)] For general $s\in S$, $$\dim (\mc X^m_s/\mc H_s)\geq \dim (\mc X_s^m)-\dim \mc H_s+\dim S$$
        (see Notation and Conventions for the definition of $\mc X_s^m/\mc H_s$),
        \item[(ii)] or $\mr{Mod}.\dim _{\mc G^{[m]}/\mc H}(S)=\mr{Mod}.\dim_{\mc G}(S)$.
    \end{enumerate}
    \end{theorem}
\begin{proof}[Proof of Theorem \ref{theorem_nondeg_group_quotient} $\Rightarrow$ Theorem \ref{theo_nondeg_gen}]
   Let $M\subset \mc P_{mg,mt}^\times$ be the smallest special subvariety containing $Y=\ovl{\iota^{[m]}(\mc X^{[m]})}$.
    By hypothesis (2) in Theorem \ref{theo_nondeg_gen}, we obtain that 
    $\mc X^m_s$ generates $\mc A^m_s$ for any $s\in S$. Therefore
    $M$ satisfies the condition \eqref{eq_split_unipotent_radical} due to Lemma \ref{lemm_split_unipotnent_radical}. 
   Hence $M$ is a semiabelian subscheme of $\mc P_{mg,mt}^\times|_{\pi_{\sab}(M)}$ by Theorem \ref{theorem_special_subvariety}. Let $\mc G_{M}$ be the neutral component of $(\iota^{[m]})^{-1}(M\cap \iota^{[m]}(\mc G^{[m]})).$

    Let $(Q,\mc Y^+)$ be the Shimura subdatum of $M$. Assume that $\mc X^{[m]}$ is not non-degenerate, by Proposition \ref{prop_bettirank_deg}, $Y^{\deg}(0)=Y.$ Then there exists a normal subgroup $N$ of $Q$ with semisimple reductive part satisfying $N=(U_N\oplus V_N^\ab)\cdot N$ such that \begin{equation}\label{eq_deg_contradiction}
        \dim Y-\dim\rho_{/N}(Y)>h_N
    \end{equation} due to Theorem \ref{theorem_quotient_subgrp_closure}.

    Let $M_0$ be the semiabelian subscheme of $M$ constructed in Corollary \ref{coro_semiabelian_subscheme}. 
    Put $\mc H$ to be the neutral component of $(\iota^{[m]})^{-1} (M_0\cap \iota^{[m]} (\mc G^{[m]}))\subset \mc G_M.$

In the case (i) of Theorem \ref{theorem_nondeg_group_quotient}, by the discussion on two steps quotient of $\rho_{/N}$ in \eqref{eq_two_steps_quotient}, \begin{align*}
    \dim \rho_{/N}(Y)&\geq \dim (Y/M_0)_{\iota_S(s)}= \dim (Y/\iota^{[m]}(\mc H))_{\iota_S(s)}\\
    &\geq\dim (\mc X_s^m/\mc H_s)\\
    &\geq \dim (\mc X_s^m)-\dim \mc H_s +\dim S\\
    &\geq\dim Y_{\iota_S(s)}-\dim (M_0)_{\iota_S(s)}+\dim \iota_S(S)\\
    &=\dim Y-h_N
\end{align*}
for general $s\in S$, where the first and second equality is due to hypothesis (1) in Theorem \ref{theo_nondeg_gen}, and the last equality is due to \eqref{eq_betti_rank}. This is a contradiction to \eqref{eq_deg_contradiction}.

As for the case (ii), 
since $\pi_{\mc G^{[m]}/\mc A^{[m]}}(\mc G_M)=\mc A^{[m]}$ by condition (2) of Theorem \ref{theo_nondeg_gen}, $\mc G$ is isogeneous to $\mc G_M\times \mathbb G_m^{t'}$ for some $0\leq t'\leq t.$ Hence $\mc G/\mc H$ is isogeneous to $(\mc G_M/\mc H)\times \mathbb G_m^{t'}$. Therefore $$\mr{Mod}.\dim_{\mc G_M/\mc H}(S)=\mr{Mod}.\dim_{\mc G/\mc H}(S)\overset{\text{by (ii)}}{=}\mr{Mod}.\dim_{\mc G}(S)=\dim \iota_S(S).$$ Thus by the discussion before \eqref{eq_two_steps_quotient}, we obtain that $$\dim \rho_{/_{p_{\sab}(N)}}(\iota_S(S))=\dim \iota_S(S)$$ due to the minimality of $\mr{Mod}.\dim _{\mc G/\mc H}(S).$ Moreover, \eqref{eq_two_steps_quotient} again gives that
$$\dim \rho_{/N}(Y)=\dim (Y/M_0)|_{\iota_S(S)}=\dim Y-h_N$$
which also contradicts \eqref{eq_deg_contradiction}.
\end{proof}

From now on to the end of this section, let $\mc H$ be a semiabelian subscheme of $\mc G^{[m]}$ given by the following commutative diagram:
    \begin{equation}\label{eq_semiabelian_subscheme}
    \begin{tikzcd}
    0 \arrow[r] & T \arrow[d] \arrow[r] & \mathcal{H} \arrow[d] \arrow[r] & \mc A_{\mc H} \arrow[d] \arrow[r] & 0 \\
    0 \arrow[r] & (\mathbb G_m^t)^m \arrow[r] & \mc{G}^{[m]} \arrow[r] & \mathcal A^{[m]} \arrow[r] & 0
    \end{tikzcd},
    \end{equation}
\subsection{Quotient via spliting isogeny}\label{subsection_induction_step}
This part will serve as a induction step for the proofs in the subsequent subsections.
In this subsection, we consider the setting up as follows.
Let \begin{align*}
    &0\rightarrow \mc T_1\rightarrow \mc G_1\rightarrow \mc G_0\rightarrow 0\\
    \text{ and }&0\rightarrow \mc T_1^\perp\rightarrow \mc G_1^\perp\rightarrow \mc G_0\rightarrow 0
\end{align*} be exact sequences of semiabelian schemes over $S.$ Here, each $\mc T_1$ (resp. $\mc T_1^\perp$) is not necessarily the toric part of $\mc G_1$ (resp. $\mc G_1^\perp$).
Let $\mc T:=\mc T_1\times_S \mc T_1^\perp$ and $\mc G:=\mc G_1\times_{\mc G_0}\mc G_1^\perp$. Denote the projections
$$\begin{cases}
    p_{1,\mc T}:\mc T\rightarrow \mc T_1,\\
    p_{1,\mc T}^\perp:\mc T\rightarrow \mc T_1^\perp,\\
    p_1:\mc G\rightarrow \mc G_1,\\
    p_1^\perp:\mc G\rightarrow \mc G_1^\perp.
\end{cases}$$

Let $\mc X_1\subset \mc G_1$ and $\mc X_1^\perp\subset\mc G_1^\perp$ be closed subvarieties such that 
\begin{itemize}
    \item[(i)] $\pi_{\mc G_1/\mc G_0}(\mc X_1)=\pi_{\mc G_1^\perp/\mc G_0}(\mc X_1^\perp)=:\mc X_0$,
    \item[(ii)] $\mc X_0\rightarrow S$ is dominant,
    \item[(iii)] $\mr{Stab}_{\mc T_{1,\ovl\eta}}(\mc X_{1,\ovl\eta})$ is finite,
\end{itemize}

\begin{proposition}\label{prop_isogeny_induction}
     Let $\mc B\subset \mc T$ be a semiabelian subscheme. Assume that there exists an isogeny $h_{\mc T}:\mc T\rightarrow \mc T$ satisfying the following:
    \begin{itemize}
        \item[(a)] $h_{\mc T}(\mc B)=\mc B_1\times_S \mc B_1'^\perp$, where $\mc B_1:=p_{1,\mc T}(h_{\mc T}(\mc B))$ and $\mc B_1'^\perp=p_{1,\mc T}^\perp(h_{\mc T}(\mc B))$;
        \item[(b)] There exists another isogeny $h_{1,\mc T}:\mc T_1\rightarrow \mc T_1$ such that $h_{1,\mc T}p_{1,\mc T}=p_{1,\mc T}h_{\mc T}$, i.e., the following diagram:
        \[
\begin{tikzcd}
  \mc T \ar[d,"p_{1,\mc T}"'] \ar[r,"h_{\mc T}"] & \mc T \ar[d,"p_{1,\mc T}"] \\
  \mc T_1 \ar[r,"h_{1,\mc T}"'] &  \mc T_1
\end{tikzcd}
\]
    \end{itemize}
    Let $\mc B_1^\perp$ be the neutral component of $h_{\mc T}^{-1}(\epsilon_S\times_S \mc B_1'^\perp)$. Put $\mc X:=\mc X_1\times_{\mc X_0}\mc X_1^\perp \subset \mc G$.
    Then
    $$\dim ((\mc X/\mc B)_w)\geq \dim (\mc X_1)_w-\dim (\mc B_1)_s+ \dim ((\mc X_1^\perp/\mc B_1^\perp)_w)+\begin{cases}
    1  \qquad \text{if }\mc B_1\not=0,\\
    0   \qquad \text{otherwise}
\end{cases}$$
for general $w\in \mc X_0$ and $s=\pi_{\mc G_0/S}(w).$
\end{proposition}
\begin{proof}
        
We do the fiberwise pushout over $S$ for the isogeny $h_{\mc T}:\mc T\to\mc T$, to a $\mc G_0$-invariant isogeny
$h:\mc G\to\mc G'.$
Let $\mc G_1'$ (resp. $\mc G_1'^\perp$) be the pushout of $\mc G'$ along $p_{1,\mc T}$ (resp. $p_{1,\mc T}^\perp$), and denote by $p_1':\mc G'\to \mc G_1'$ (resp. $p_1'^\perp:\mc G'\to \mc G_1'^\perp$) the induced projection. By condition (b), there is an isogeny $h_1:\mc G_1\to \mc G_1'$
such that
\begin{equation}\label{eq_isogeny_cd}
h_1\circ p_1 = p_1'\circ h.    
\end{equation}

Consider the quotient
\[
\mc G'/h(\mc B)=(\mc G_1'/\mc B_1)\times_{\mc G_0}(\mc G_1'^\perp/\mc B_1'^\perp),\]
and the induced isogeny $\ovl h:\mc G/\mc B\rightarrow \mc G'/h(\mc B)$. Let
$\ovl p_1':\mc G'/h(\mc B)\to \mc G_1'/\mc B_1$
be the first projection. Then we have the following commuative diagram
\[
\begin{tikzpicture}[
  scale=1.1,
  every node/.style={font=\small},
  >=stealth
]
\node (GB) at (0,0) {$\mc G/\mc B$};
\node (G1B) at (3.5,0) {$\mc G_1/p_1(\mc B)$};
\node (GpB) at (1.4,1.6) {$\mc G'/h(\mc B)$};
\node (G1pB) at (4.9,1.6) {$\mc G_1'/\mc B_1$};

\node (G) at (0,2.5) {$\mc G$};
\node (G1) at (3.5,2.5) {$\mc G_1$};
\node (Gp) at (1.4,4.1) {$\mc G'$};
\node (G1p) at (4.9,4.1) {$\mc G_1'$};

\draw[->] (G) -- node[below] {$p_1$} (G1);
\draw[->] (G) -- node[left] {$h$} (Gp);
\draw[->] (G1) -- node[right] {$h_1$} (G1p);
\draw[->] (Gp) -- node[above] {$p_1'$} (G1p);

\draw[->] (GB) -- node[below] {$\ovl p_1$} (G1B);
\draw[->] (GB) -- node[left] {$\ovl h$} (GpB);
\draw[->] (G1B) -- node[right] {$\ovl h_1$} (G1pB);
\draw[->] (GpB) -- node[above] {$\ovl p_1'$} (G1pB);

\draw[->] (G) -- node[left] {} (GB);
\draw[->] (G1) -- node[right] {} (G1B);
\draw[->] (Gp) -- node[left] {} (GpB);
\draw[->] (G1p) -- node[right] {} (G1pB);

\end{tikzpicture}
\]
Then by \eqref{eq_isogeny_cd},
\[
\ovl p'_1\bigl(\ovl h(\mc X/\mc B)\bigr)=\ovl p'_1\bigl(h(\mc X)/h(\mc B)\bigr)=h_1(\mc X_1)/\mc B_1.
\]
Hence for a general point $z\in h_1(\mc X_1)_w$ with $\ovl z$ being its image in $h_1(\mc X_1)_s/(\mc B_1)_s$, we have
\[
\begin{aligned}\allowdisplaybreaks
\dim\bigl(\ovl h(\mc X/\mc B)_w\bigr)&\ge\dim\bigl((h_1(\mc X_1)/\mc B_1)_w\bigr)+\dim\Bigl({\ovl p'_1}^{-1}(\ovl z)\cap \ovl h(\mc X/\mc B)_w\Bigr)\\
&\ge\dim(\mc X_1)_w - \dim(\mc B_1)_s+\dim\Bigl({\ovl p'_1}^{-1}(\ovl z)\cap \ovl h(\mc X/\mc B)_w\Bigr)\\
&\kern 7em+\begin{cases}
    1  \qquad \text{if }\mc B_1\not=0,\\
    0   \qquad \text{otherwise}
\end{cases},
\end{aligned}
\]
where the last inequality follows from condition (iii) for $\mc X_1$.

Now notice that
\[
{\ovl p'_1}^{-1}(\ovl z)\cap \ovl h(\mc X/\mc B)_w=\ovl h((\mc X/\mc B)\cap \ovl h^{-1}{\ovl p'_1}^{-1}(\ovl z))=\ovl h((\mc X/\mc B)\cap \ovl p_1^{-1}\ovl h_1^{-1}(\ovl z))
\]
where the second equality is due to \eqref{eq_isogeny_cd}.
Therefore 
$$
\begin{aligned}\allowdisplaybreaks
    {\ovl p'_1}^{-1}(\ovl z)\cap \ovl h(\mc X/\mc B)_w&=\ovl h((p_1^{-1}h_1^{-1}(z)\cap \mc X)/\mc B_s)\\
    &\kern-2em=\bigcup_{y\in \mc X_1\cap h_1^{-1}(z)} \ovl h\big(\left(\{y\}\times(\mathcal X_1^\perp)_w\right)/\mc B_s\big)
\end{aligned}
$$
where the first equality holds since $$(p_1^{-1}h_1^{-1}(z+(\mc B_1)_s)\cap \mc X)/\mc B_s=(p_1^{-1}h_1^{-1}(z)\cap\mc X)/\mc B_s.$$
As $\mc X_1\cap h_1^{-1}(z)$ is a finite set, we only need to estimate $\dim ((\{y\}\times(\mathcal X_1^\perp)_w)/\mc B_s)$.

We take arbitrary $y'\in \{y\}\times (\mc G_1^\perp)_w$,
Then 
\begin{align*}\allowdisplaybreaks\dim ((\{y\}\times(\mathcal X_1^\perp)_w)/\mc B_s)&=\dim ((\{y\}\times(\mathcal X_1^\perp)_w-y')/\mc B_s)
\end{align*}
Observe that \begin{align*}
    \{y\}\times(\mathcal X_1^\perp)_w-y'=\{e_1\}\times ((\mc X_1^\perp)_s-p_1^{\perp}(y'))_{e_0}
\end{align*}
where $e_1$ (resp. $e_0$) stands for the identity in $(\mc G_1)_s$ (resp. $(\mc G_0)_s$). Therefore \begin{align*}
\dim ((\{y\}\times(\mathcal X_1^\perp)_w)/\mc B_s)&=\dim (((\mc X_1^\perp)_s-p_1^{\perp}(y'))_{e_0}/(\{e_1'\}\times (\mc T_1^\perp)_s\cap\mc B_s))\\
&=\dim ((\mc X_1^\perp/\mc B_1^\perp)_{w})
\end{align*}
where $e_1'$ stands for the identity in $(\mc T_1)_s$. The last equality is due to the facts that 
\begin{align*}
    \mc B_1^\perp&=\text{neutral component of }h_{\mc T}^{-1}(\epsilon_S\times_S \mc B_1'^\perp)\\&=\text{neutral component of }\mc B\cap h_{\mc T}^{-1}(\epsilon_S\times_S\mc T_1^\perp).
\end{align*}
and $(e_S\times_S T_1^\perp)\cap \mc B\subset \mc B\cap h_{\mc T}^{-1}(\epsilon_S\times_S\mc T_1^\perp)$ by hypothesis (b), which thus contains $\mc B_1^\perp$ as a subgroup of finite index by dimension reason.
\end{proof}

\subsection{Quotient by tori}
Now we consider the situation as in \eqref{eq_semiabelian_subscheme}. 
\begin{proposition}\label{prop_torus_quotient}
    Let $h_{\tor}$ be an isogeny as in Proposition \ref{prop_isogeny_trick} applying to $T\subset (\mathbb G_m^t)^m$ such that after reordering factors of $(\mathbb G_m^t)^m$, $h_{\tor}(T)=T_1\times \cdots\times T_m$ and $T_i\not=0$ if and only if $i\leq m_{\tor}$ for some $m_{\tor}$.
    Then $$\dim (\mc X_s^m/\mc H_s)\geq \dim \mc X_s^m-\dim \mc H_s+m_{\tor}$$
    for general $s\in S.$
\end{proposition}
\begin{proof}
Notice that 
$$\dim (\mc X_s^m/\mc H_s)\geq\dim (\mc X^m_s/T)-\dim (\mc A_{\mc H})_s$$
and $\dim (\mc X^m_s/T)\geq \dim (\mc X^m_w/T)+\dim \pi(X)^m_s$ for general $w\in \pi(\mc X)^m_s$.
It suffices to prove that \begin{equation}\label{eq_torus_quotient_fiber_dim}
    \dim (\mc X^m_w/T)\geq\dim (\mc X_w^m)-\dim T+m_{\tor}
\end{equation}
for general $s\in S$ and $w\in \pi(\mc X)^m.$

We proceed by induction on $m$. In the case that $m=1$ is due to the condition (3) in Theorem \ref{theo_nondeg_gen}.
Now assume that $m\geq 2$. We may further assume that $m_{\tor}>0$ since there is nothing to prove if $m_\tor=0.$
We briefly explain that we are in the situation as \S \ref{subsection_induction_step} by the substitution
\[\begin{cases}\allowdisplaybreaks
    \mc T_1=\mathbb G_{m,S}^t, \\\mc G_1=\mc G\times_S \mc A^{[m-1]},\\
    \mc T_1^\perp=(\mathbb G_m^t)^{m-1}\times S,\\
    \mc G_1^\perp=\mc A\times_S\mc G^{[m-1]},\\
    \mc X_1=\mc X\times_S \pi(\mc X)^{[m-1]},\\ \mc X_1^\perp=\pi(\mc X)\times_S \mc X^{[m-1]}.
\end{cases}
\]
Note that here $\mc G_0=\mc A^{[m]}$. Applying Proposition \ref{prop_isogeny_induction} for $\mc B=T\times S$, we obtain that 
for general $w\in \pi(\mc X)^{[m]},$
\begin{equation}\label{eq_torus_quotient_fiber_dim_1}\dim (\mc X_w^m/T)\geq \dim (\mc X_1)_w-\dim (T_1)+\dim ((\mc X_1^\perp)_w/ T_1^\perp)+1.\end{equation}
Here $T_1^\perp$ is the neutral component of $h_{\tor}^{-1}(\{0\}\times T_2\times\cdots T_m)$.

Since $(\mc X_1)_w\simeq \mc X_{p_1(w)}$ and $(\mc X_1^\perp)_w\simeq\mc X^{[m-1]}_{(p_2,\dots,p_m)(w)}$, applying the induction hypothesis to $h_{\tor}^{(1)}$ (see the notation in Proposition \ref{prop_isogeny_trick}(3)) and $T_1^\perp$, we obtain that 
$$\dim ((\mc X_1^\perp)_w/ T_1^\perp)\geq \dim (\mc X^{[m-1]}_{(p_2,\dots,p_m)(w)})-\sum_{2\leq i\leq m}\dim T_i+m_\tor-1.$$
Combining with \eqref{eq_torus_quotient_fiber_dim_1}, we obtain \eqref{eq_torus_quotient_fiber_dim}, which concludes the proof.
\end{proof}
\subsection{Quotient by abelian subscheme}
The following is a generalization of \cite[Lemma 10.2 and (10.4)]{ZG1}.
\begin{proposition}\label{prop_abel_quotient} Assume that $\mc B\subset \mc G^{[m]}$ is an abelian subscheme. Then either
 $$\dim(\mc X_s^m/\mc B_s)\geq \dim \mc X^{[m]}-\dim \mc B+m$$
for general $s\in S$, or $\mr{Mod}.\dim_{\mc G^{[m]}/\mc B}(S)=\mr{Mod}.\dim_{\mc G}(S).$ 
\end{proposition}
\begin{proof}
Applying Proposition \ref{prop_isogeny_trick} to the generic fiber $\mc B_\eta\subset \mc G_\eta^m$, we obtain an isogeny $h:\mc G^m_\eta\rightarrow \mc G_\eta^m$  We may assume that $S$ is normal after replacing $S$ with its normalization.
Then by \cite[Chapter I, Proposition 2.7]{FalitngsChai}, $h$ induces an isogeny $\mc G^{[m]}\rightarrow \mc G^{[m]}$, which we still denote as $h$, such that the properties in Proposition \ref{prop_isogeny_trick} are preserved. We recall here:
    \begin{enumerate}
        \item $h(\mc B)=\mc B_1\times \cdots\times \mc B_m$ where $\mc B_i=p_{i}(h(\mc B));$
        \item $\mc B_i\not=0$ if and only if $i\leq m'$ for some $0\leq m'\leq m;$
        \item For each $i=0,\dots, m-1$, $h$ preserves the subgroup $\mr{ker}(p_{1},\dots,p_i)$. We denote the induced isogeny on $\mr{ker}(p_{1},\dots,p_i)\simeq \mc G^{[m-i]}$ as $h^{(i)}.$ Moreover $h^{(m')}$ is the multiplication by $n$ map for some $n>0;$ 
        \item For each $i=1,\dots,m$ there exists an isogeny $h_i:\mc G\rightarrow \mc G$ such that 
        $$h_i\circ p_{i}= p_{i}\circ h^{(i-1)}.$$
    \end{enumerate}
We first consider the case that $m'<m$.
After identifying $\mr{ker}(p_1,\dots,p_{m'})$ with $\mc G^{[m-m']}$, by condition (3), we have $\mc G^{[m]}/h(\mc B)=\mc G^{[m-m']}\times_S \mc G'$ for some semiabelian scheme $\mc G'$ over $S$.
Therefore $$\mr{Mod}.\dim_{\mc G^{[m]}/\mc B}(S)=\mr{Mod}.\dim_{\mc G^{[m]}/h(\mc B)}(S)=\mr{Mod}.\dim_{\mc G}(S).$$

From now on, we assume that $m=m'$. We proceed by induction on $m$. The case that $m=1$ is trivial . Now we assume that $m\geq 2$. Similarly as in previous subsection, we consider the following substitution according to the notations in \S \ref{subsection_induction_step}:
\[\begin{cases}
    \mc G_0=S,\\
    \mc T_1=\mc G_1=\mc G,\\
    \mc T_1^\perp=\mc G_1^\perp=\mc G^{[m-1]},\\
    \mc X_1=\mc X,\\
    \mc X_1^\perp=\mc X^{[m-1]}.
\end{cases}\]
Then Proposition \ref{prop_isogeny_induction} gives that 
\begin{equation}\label{eq_quotient_abel_fiber_dim}
    \dim (\mc X_s^m/\mc B_s)\geq \dim \mc X_s-\dim (\mc B_1)_s+\dim (\mc X^{m-1}_s/(\mc B_1^\perp)_s)+1
\end{equation}
where $\mc B_1^\perp$ is the neutral component of $h^{-1}(\epsilon_S\times_S \mc B_2\times_S\cdots\times_S\mc B_m).$ By induction hypothesis, $$\dim (\mc X_s^{m-1}/(\mc B_1^\perp)_s)\geq \dim (\mc X^{[m-1]})-\dim (\mc B_1^\perp)+m-1$$
which concludes the proof by combining with \eqref{eq_quotient_abel_fiber_dim}.
\end{proof}

\subsection{Proof of Theorem \ref{theorem_nondeg_group_quotient}}
\begin{proof}
    Applying the Proposition \ref{prop_isogeny_trick} to $T$ and $(\mathbb G_m^t)^m$, we obtain an isogeny $h_{\mr{tor}}:(\mathbb G_m^t)^m\rightarrow (\mathbb G_m^t)^m$ such that \begin{enumerate}
        \item $h_{\tor}(T)=T_1\times \cdots \times T_m$ where $T_i$ is the projection of $h_{\tor}(T)$ on the $i$-th $\mathbb G_m^t.$
        \item For each $i$ such that $T_i=0$, the restriction of $h_\tor$ on the $i$-th $\mathbb G_m^t$ is the multiplication by $n$ map.
    \end{enumerate}
    Let $\mc G'$ be the pushout of $\mc G^{[m]}$ along $h_{\tor}$, which yields an isogeny $h:\mc G^{[m]}\rightarrow \mc G'.$ 
    
    Without loss of generality, we may assume that $T_i\not=0$ if and only if $i\leq m_\tor$ for some $0\leq m_{\tor}\leq m.$ We abbreviate by $\mc G^{[m-m_\tor]}$ the fiber product of the last $m-m_{\tor}$ $\mc G$'s over $S$, and by $p:\mc G^{[m]}\rightarrow \mc G^{[m_\tor]}$ the quotient map.
    By the construction and condition (2) above, the restriction of $h$ on $\mc G^{[m-m_\tor]}$ is the multiplication by $n$ map.
    
    Let $\mathcal B$ be the neutral component of $h(\mc H)\cap \mc G^{[m-m_\tor]}$, which is a semiabelian subscheme of $\mc G^{[m-m_\tor]}$. The torus part of $\mc B$ is contained in $(\mathbb G_m^t)^{m-m_\tor}\cap h_{\tor}(T)=0$, i.e., $\mc B$ is an abelian subscheme.
We have the following diagram \begin{equation}\label{eq_cd_quotient_isogeny}
\begin{tikzcd}
0 \ar[r] & \mc G^{[m-m_{\mr{tor}}]} \ar[r] \ar[d, "\lbrack n\rbrack"] & \mc G^{[m]} \ar[r,"p"] \ar[d, "h"] & \mc G^{[m_{\mr{tor}}]} \ar[r] \ar[d,"h'"] & 0 \\
0 \ar[r] &
\mc G^{[m-m_{\mr{tor}}]} \ar[r] \ar[d] &
\mc G' \ar[r, "p'"] \ar[d] &
\mc G'' \ar[r] \ar[d] &
0 \\
0 \ar[r] & \mc G^{[m-m_{\mr{tor}}]}/\mathcal B \ar[r] & \mc G'/h(\mc H) \ar[r,"\ovl p'"] & \mc G''/h'(p(\mc H)) \ar[r] & 0
\end{tikzcd}
\end{equation}
where $\mc G''=\mc G'/\mc G^{[m-m_\tor]},$ and $h':\mc G^{[m_\tor}]\rightarrow \mc G''$ is an $\mathcal A^{[m_\tor]}$-invariant isogeny. Remark that the bottom left square is not necessarily a pull-back, but since $\mc B$ is of finite index in $\mr{ker}(p')$, this does affect our dimension arguments.

    Therefore \begin{equation}\label{eq_torus_fiber_dim}
        \begin{aligned}
        \dim (\mathcal X_s^{m}/\mc H_s)&=\dim (h(\mathcal X^{m}_s)/h(\mc H_s))\\
        &\geq\dim \ovl p'(h(\mathcal X_s^m)/h(\mc H_s))+\dim (\ovl p'^{-1}(\ovl {h'(y)})\cap (h(\mathcal X_s^m)/h(\mc H_s)))
    \end{aligned}
    \end{equation}
    for general $s\in S$, $y\in p(\mathcal X_s^m)=\mathcal X_s^{m_\tor}$, and $\ovl {h'(y)}$ being its image in $\mc G''/h'(p(\mc H)).$

    We first see that 
    \begin{equation}\label{eq_abelian_quot_dim}
        \begin{aligned}\dim \ovl p'(h(\mathcal X_s^m)/h(\mc H_s))&=\dim (h'(p(\mathcal X_s^m))/h'(p(\mc H_s)))\\&=\dim (h'(\mathcal X_s^{m_{\tor}})/h'(p(\mc H_s)))\\&\geq \dim (\mathcal X_s^{m_{\tor}})-\dim p(\mc H_s)+m_{\tor}
    \end{aligned}
    \end{equation}
    where the inequality is due to Proposition \ref{prop_torus_quotient}.
    
    On the other hand, for general $s\in S$ and $y\in \mc X^{[m_\tor]},$
    \begin{align*}
        &\dim (\ovl p'^{-1}(\ovl {h'(y)})\cap \big(h(\mathcal X_s^m)/h(\mc H_s)\big))=\dim (\big(p^{-1}(y)\cap \mathcal X_s^m\big)/\mc H_s)\\
        &\kern 3em=\dim \big(\mathcal X_s^{m-m_{\tor}}+(z_1,\dots,z_{m-m_\tor})\big)/(\mc H_s\cap \mc G_s^{m-m_\tor})\\
        &\kern 3em=\dim \left(\big((\mathcal X_s-z_1)\times \cdots \times(\mathcal X_s-z_{m-m_\tor})\big)/\mathcal B_s\right)
    \end{align*}
    for $(z_1,\dots,z_{m-m_{\tor}})\in \mathcal X_s^{m-m_{\tor}}\simeq p^{-1}(y)\cap \mathcal X_s^m.$
    By Proposition \ref{prop_abel_quotient}, either

\noindent(i) for general $s\in S$,
    \begin{align*}
        \dim \left(\big((\mathcal X_s-z_1)\times \cdots \times(\mathcal X_s-z_{m-m_\tor})\big)/\mathcal B_s\right)\geq \dim \mathcal X_s^{m-m_\tor}-\dim \mathcal B_s+m-m_{\tor}
    \end{align*}
    (ii) or $$\mr{Mod}.\dim_{\mc G^{[m-m_{\tor}]}/\mc B}(S)=\mr{Mod}.\dim_{\mc G^{[m-m_{\tor}]}}(S)=\dim \iota_S(S).$$ 
    
    Note that case (ii) gives that $$\dim \iota_S(S)=\mr{Mod}.\dim_{\mc G'/h(\mc H)}(S)=\mr{Mod}.\dim_{\mc G^{[m]}/\mc H}(S)$$
    where the first equation is given by the inclusion $\mc G^{[m-m_{\tor}]}/\mc B\subset \mc G'/h(\mc H)$ in \eqref{eq_cd_quotient_isogeny}.

    If $m\geq \dim S$, combining with \eqref{eq_torus_fiber_dim} and \eqref{eq_abelian_quot_dim}, we thus conclude the proof in both cases. 
\end{proof}

\section{Adelic line bundles}\label{section_adelic}

We review Yuan-Zhang's theory of adelic line bundles on quasi-projective varieties \cite{YuanZhang}.

By an \emph{arithmetic variety} (resp. \emph{projective arithmetic variety}) we mean an integral scheme which is quasi-projective (resp. projective) and flat over $\spec(\ZZ)$.

\subsection{Adelic divisors}

Let $\ZZZ$ be a projective arithmetic variety. An \emph{arithmetic divisor} $\ovl\DDD=(\DDD,g)$ on $\ZZZ$ consists of a Cartier divisor $\DDD$ on $\ZZZ$ and a continuous Green function
$$g:(\ZZZ\setminus|\DDD|)(\C)\to\R$$
with logarithmic singularity along $|\DDD|(\C)$, i.e., for any meromorphic function $f$ on an analytic open subset $U\subset\ZZZ(\C)$ such that $\DDD|_{U}=\mr{div}(f)$, the function $g+\log|f|$ extends to a continuous function on $U$. Denote by $\widehat\Div(\ZZZ)$ the group of arithmetic divisors.

An arithmetic divisor $\ovl\DDD=(\DDD,g)$ is \emph{effective} (resp. \emph{strictly effective}), denoted by $\ovl\DDD\ge0$ (resp. $\ovl\DDD>0$), if $\DDD$ is an effective divisor and $g\ge0$ (resp. $g>0$) on $(\ZZZ\setminus|\DDD|)(\C)$.

A \emph{$\Q$-arithmetic divisor} on $\ZZZ$ is an element in $\widehat\Div(\ZZZ)\otimes\mathbb Q$. A $\Q$-arithmetic divisor $\ovl\DDD$ is effective if $m\ovl\DDD$ is an effective arithmetic divisor for some positive integer $m$.

Let $\UUU$ be an arithmetic variety. A \emph{projective model} of $\UUU$ is a projective arithmetic variety $\ZZZ$ with an open immersion $\UUU\to\ZZZ$. A \emph{model divisor} on $\UUU$ is 
a pair $(\DDD,\ovl\DDD')$ of a divisor $\DDD$ on $\UUU$ and a $\Q$-arithmetic divisor $\ovl\DDD'$ on some projective model $\ZZZ$ of $\UUU$ such that $\DDD'|_\UUU=\DDD$. If both $\DDD$ and $\ovl\DDD'$ are effective, we say $(\DDD,\ovl\DDD')$ is effective.

A \emph{boundary divisor} of $\UUU$ is a strictly effective arithmetic divisor $\ovl\EEE$ on some projective model $\ZZZ$ such that $|\EEE|=\ZZZ\setminus\UUU$.

An \emph{adelic divisor} $\ovl\DDD=(\DDD,(\ZZZ_i,\ovl\DDD_i)_{i\ge1})$ on $\UUU$ consists of:
\begin{enumerate}
    \item[(1)]a Cartier divisor $\DDD$ on $\UUU$, called the \emph{underlying divisor} of $\ovl\DDD$,
    \item[(2)]a projective model $\ZZZ_i$ of $\UUU$ for each $i\ge1$,
    \item[(3)]a $\Q$-arithmetic divisor $\ovl\DDD_i$ on each $\ZZZ_i$ such that $\DDD_i|_{\UUU}=\DDD$.
\end{enumerate}
It is required to satisfy the following \emph{Cauchy condition}. Let $\ovl\EEE$ be a boundary divisor of $\UUU$. Then there is a sequence of rational numbers $(\epsilon_i)_{i\ge1}$ converging to $0$ such that for all $i\ge j\ge 1$,
$$-\epsilon_j\ovl\EEE\le\ovl\DDD_{i}-\ovl\DDD_{j}\le\epsilon_j\ovl\EEE.$$
This condition is independent of $\ovl\EEE$ (cf. \cite[Lemma 2.4.1]{YuanZhang}). Two adelic divisors $\ovl\DDD=(\DDD,(\ZZZ_i,\ovl\DDD_i)_{i\ge1})$ and $\ovl\DDD'=(\DDD',(\ZZZ_i',\ovl\DDD_i')_{i\ge1})$ are identified if $\ovl\DDD_i-\ovl\DDD_i'$ converges to $0$, i.e., there is a sequence of rational numbers $(\epsilon_i)_{i\ge1}$ coverging to $0$ such that 
$$-\epsilon_i\ovl\EEE\le\ovl\DDD_{i}-\ovl\DDD_{i}'\le\epsilon_i\ovl\EEE.$$
Denote by $\widehat\Div(\UUU/\ZZ)$ the group of adelic divisors.

We say an adelic divisor $\ovl\DDD$ on $\UUU$ is \emph{effective} if $\ovl\DDD=(\DDD,(\ZZZ_i,\ovl\DDD_i)_{i\ge1})$ for a sequence of effective $\Q$-arithmetic divisors $\ovl\DDD_i$.

Let $U$ be a quasi-projective variety over $\Q$. An \emph{arithmetic model} of $U$ is an arithmetic variety $\UUU$ with an isomorphism $\UUU_\Q\cong U$. An adelic divisor $\ovl D$ on $U$ is an element in the direct limit $\varinjlim\widehat\Div(\UUU/\Z)$ over all arithmetic models. It is effective if it is represented by some effective adelic divisors on some arithmetic model.

\subsection{Adelic line bundles}


Let $\ZZZ$ be a projective arithmetic variety. A \emph{hermitian line bundle} $\ovl\LLL=(\LLL,\lVert\cdot\rVert)$ on $\ZZZ$ consists of a line bundle $\LLL$ on $\ZZZ$ with a continuous hermitian metric $\lVert\cdot\rVert$ on $\LLL(\mathbb{C})$. Any rational section $s$ of $\LLL$ defines an arithmetic divisor
$$\widehat{\mr{div}}(s)=(\mr{div}(s),-\log\|s\|).$$
Conversely, an arithmetic divisor $\overline\DDD=(\DDD,g)$ defines a hermitian line bundle
$$\mathcal{O}(\overline\DDD)=(\mathcal{O}(\DDD),\lVert\cdot\rVert)$$
with the hermitian metric determined by $-\log\|s\|=g,$
where $s$ is the rational section of $\mathcal{O}(\DDD)$ corresponding to $\DDD$.

A $\Q$-line bundle $q\LLL$ (resp. a $\Q$-hermitian line bundle $q\ovl\LLL$) on $\ZZZ$ consists of a rational number $q$ and a line bundle $\LLL$ (resp. a hermitian line bundle $\ovl\LLL$). A rational section of a $\Q$-line bundle $q\LLL$ is of form $\frac{1}{m}s_m$ for some integer $m$ killing the denominator of $q$ and rational section $s_m$ of $(mq)\LLL$. Two rational sections $\frac1ms_m$ and $\frac1{m'}s_{m'}$ are identified if $nm's_m=nms_{m'}$ for some positive integer $n$.

Let $q\ovl\LLL$ be a $\Q$-hermitian line bundle on $\ZZZ$. A rational section $\frac{1}{m}s_m$ of $q\LLL$ defines a $\Q$-arithmetic divisor
$$\widehat{\mr{div}}\left(\frac1ms_m\right)=\frac1m\widehat{\mr{div}}(s_m).$$

Let $\UUU$ be an arithmetic variety. Let $\ovl\LLL$ be a $\Q$-hermitian line bundle on some projective model, and $s$ be a rational section of $\LLL$. Then $(\mr{div}(s)|_{\UUU},\widehat{\mr{div}}(s))$ is a model divisor. We still denote this by $\widehat{\mr{div}}(s)$ if there is no ambiguity.

Let $\ovl \LLL_1$ and $\ovl\LLL_2$ be two $\Q$-hermitian line bundles on two projective models $\mathcal Z_1$ and $\mathcal Z_2$, respectively. Assume that we have an isomorphism $\ell:\LLL_1|_\UUU\to\LLL_2|_\UUU$ as $\Q$-line bundles. Then $\ell$ induces a rational function $f$ on $\mathcal U$, via $m\LLL_2|_\UUU-m\LLL_1|_\UUU\simeq \mathcal O_{\mathcal U}$ for some positive integer $m$, which can also be viewed as rational functions on both $\mathcal Z_1$ and $\mathcal Z_2.$ We denote by $\widehat {\mr{div}}(\ell)$ the model divisor given by $\frac{1}{m}f$ and $\ovl \LLL_1-\ovl \LLL_2.$

\begin{definition}\label{def_adelic_line_bundle} Let $\mc U$ be an arithmetic variety.
    An \emph{adelic line bundle} $\ovl\LLL=(\LLL,(\ZZZ_i,\ovl\LLL_i,\ell_i)_{i\ge1})$ on $\UUU$ consists of:
\begin{enumerate}
    \item[(1)]a line bundle $\LLL$ on $\UUU$, called the underlying line bundle of $\ovl\LLL$,
    \item[(2)]a projective model $\ZZZ_i$ of $\UUU$ for each $i\ge1$,
    \item[(3)]a $\Q$-hermitian line bundle $\ovl\LLL_i$ on each $\ZZZ_i$, with an isomorphism $\ell_i:\LLL_i|_{\mathcal U}\simeq \mathcal L$
\end{enumerate}
such that it satisfies the Cauchy condition:
$$-\epsilon_j\ovl\EEE\le\widehat{\mr{div}}(\ell_i\ell_j^{-1})\le\epsilon_j\ovl\EEE$$
for some boundary divisor $\ovl\EEE$ of $\UUU$ and sequence of rational number $(\epsilon_i)_{i\ge1}$ converging to $0$.

Let $\ovl \LLL=(\LLL,(\ZZZ_i,\ovl\LLL_i,\ell_i)_{i\ge1})$ and $\ovl \LLL'=(\LLL',(\ZZZ_i',\ovl\LLL_i',\ell_i')_{i\ge1})$ be two adelic line bundles on $\mathcal U$. We say $\ovl\LLL$ is isomorphic to $\ovl\LLL'$ if there exists an isomorphism $\iota:\LLL\rightarrow \LLL'$ such that $\widehat{\mr{div}}(\ell_i'^{-1}\iota\ell_i)$ converges to $0$.
\end{definition}

Let $U$ be a quasi-projective variety over $\Q$. An adelic line bundle on $U$ is an adelic line bundle $\ovl{\mathcal L}$ on some arithmetic model. To distinguish with the case over an arithmetic variety, we denote it by $\ovl L$ for $L=\mathcal L_{\Q}$.  An isomorphism of two adelic line bundles is an isomorphism of their pull-back to some common model. Denote by $\Pic(U)$ the isomorphism class group of adelic line bundles on $U$.

Here we recall some definition of positivity for adelic line bundles:
\begin{definition}
An adelic line bundle $\ovl L$ is said to be
\begin{enumerate}
    \item[(1)]\emph{strongly nef} (resp. \emph{semipositive}) if it is a limit of nef (resp. \emph{relatively semipositive}) $\Q$-hermitian line bundles;
    \item[(2)]\emph{nef} if for any strongly nef adelic line bundle $\ovl M$ on $U$, $\ovl L+\ovl M$ is nef;
    \item[(3)]\emph{integrable} if there are strongly nef adelic line bundles $\ovl M_1,\ovl M_2$ on $U$ such that $\ovl L=\ovl M_1-\ovl M_2$.
\end{enumerate}
Denote by $\Pic(U)_{\mr{int}}$ the isomorphism class group of integrable adelic line bundles.
\end{definition}

We also recall that for an adelic line bundle $\ovl L=(\mc L, (\mc Z_i,\ovl{\mc L}_i,s_i)_{i\geq 1})$, defined over an arithmetic model $\mc U$ of $U$ as in Definition \ref{def_adelic_line_bundle}, 
the data $(\mc L_\Q,((\mc Z_i)_\Q, (\mc L_i)_\Q,\ell_i)$, is called the \emph{geometric part} of $\ovl{L}$, and denote it as $\widetilde{L}.$
Note that the sequence also satisfies the Caunchy condition that $$-\epsilon_j\mc E_\Q\leq \mr{div}(\ell_i\ell_j^{-1})_\Q\leq \epsilon_j\mc E_\Q.$$
\subsection{Analytic data}

Let $U$ be a quasi-projective variety over $\Q$. Let $\ovl D$ be an adelic divisor on $U$. The Green function of $\ovl D$ is a continuous function
$$g_{\ovl D}:(U\setminus|D|)(\C)\to\R$$
with logarithmic singularity along $|D|(\C)$ defined by taking the pointwise limit.

Let $\ovl L$ be an adelic line bundle on $U$. It defines a hermitian metric $\lVert\cdot\rVert$ on $L_\C$ also by taking limit. If $s$ is a rational section of $L$, then $-\log\|s\|$ is the Green function of the adelic divisor $\widehat{\mr{div}}(s)$. Conversely, for any adelic divisor $\ovl D$ on $U$, the hermitian metric on $\mathcal{O}(\ovl D)$ satisfies
$$-\log\|s_D\|=g_{\ovl D},$$
where $s_D$ is the rational section of $\mathcal{O}(D)$ corresponding to $D$. 

The first Chern current of $\ovl L$ is given by $$c_1(\ovl L):=dd^c(-\log \lVert s\rVert)+\delta_{\mr{div}(s)},$$
where $s$ is a rational section of $L$. Note that $c_1(\ovl L)$ is independent of the choice of $s.$ Let $d=\dim U$. For $\ovl L_1,\dots,\ovl L_d\in\Pic(U)_{\mr{int}}$, 
$$c_1(\ovl L_1)\cdots c_1(\ovl L_d)$$
exists as a signed-measure on $U(\C)$. In particular, if each $\ovl L_i$ is nef, then $c_1(\ovl L_1)\cdots c_1(\ovl L_d)$ is non-negative, i.e., a measure.

\subsection{Intersection theory and height}

Let $U$ be a quasi-projective variety over $\Q$ of dimension $d$. There is an intersection pairing defined in \cite[\S 4.1.1]{YuanZhang}
$$\Pic(U)_{\mr{int}}^{d+1}\longrightarrow\R,\quad(\ovl L_0,\ovl L_1,\dots,\ovl L_d)\longmapsto(\ovl L_0\cdot\ovl L_1\cdots\ovl L_d),$$
which is symmetric, multi-linear and semipositive in the sense that if $\ovl L_0,\ovl L_1,\dots,\ovl L_d$ are nef, then
$$(\ovl L_0\cdot\ovl L_1\cdots\ovl L_d)\ge0.$$
In particular, if $d=0$, we denote the intersection number (of a single adelic line bundle) by $\Deg(\cdot)$.

We do not recall the full definition. Instead, we present an example to be used. Let $e$ be a real number and $\ovl{\mathcal{O}}(e)$ be the pull-back to $X$ of some adelic line bundle on $\spec(\Q)$ of degree $e$. Then its intersection number with integrable adelic line bundles $\ovl{L}_1,\dots,\ovl{L}_d$ on $X$ is
$$(\ovl{\mathcal O}(e)\cdot\ovl L_1\cdots\ovl L_d)=e\int_{U(\C)}c_1(\ovl L_{1})\cdots c_1(\ovl L_{d}).$$

Note that in \cite[\S 4.1.1]{YuanZhang}, we also have the intersection pairing for the geometric part:
$$(\widetilde L_1,\dots,\widetilde L_d)\longmapsto (\widetilde L_1\dots\widetilde L_d)$$
satisfying the formula
\begin{equation}\label{eq_chern_formula}
    \widetilde L_1\cdots\widetilde L_d=\int_{U(\C)}c_1(\ovl L_1)\dots c_1(\ovl L_d),
\end{equation}
as shown in \cite[Theorem 1.2]{Guo2025ChernFormula}.

Let $\ovl L$ be an adelic line bundle on $U$. The associated \emph{height function}
$$h_{\ovl L}(\cdot):U(\ovl\Q)\longrightarrow\R$$
is defined as follows. For any number field $K$ and $x\in U(K)$, define
$$h_{\ovl L}(x)=\frac{\Deg(x^*\ovl L)}{[K:\Q]}.$$
This is stable under field extension and gives the height function on $U(\ovl\Q)$. If $\ovl L$ is nef, then $h_{\ovl L}(\cdot)$ is non-negative, the \emph{essential minimum} is defined by
$$\mr{ess}(U,\ovl L):=\sup_{\text{open }V\subset U}\inf_{x\in V(\ovl \Q)}h_{\ovl L}(x).$$
If further $\int_{U(\C)}c_1(\ovl L)^d>0$, then it satisfies the following fundamental inequality (cf. \cite[Theorem 5.3.3]{YuanZhang})
\begin{equation}\label{eq_fundamental}
    \mr{ess}(U,\ovl L)\ge\frac{(\ovl L^{d+1})}{(d+1)\displaystyle \int_{U(\C)}c_1(\ovl L)^d}.
\end{equation}
Here $(\ovl L^{d+1})$ is the self-intersection number.

If $\ovl L$ is relatively semipositive, the \emph{absolute minimum} is defined as
$$\mr{abs}(U,\ovl L)=\sup\{\lambda\in\mathbb R:\ovl L-\ovl{\mc O}(\lambda)\mathrm{\ is\ nef}\}.$$
If further $\displaystyle\int_{U(\C)}c_1(\ovl L)^d>0$, the following inequality follows from the non-negativity of height function associated to nef line bundle:
\begin{equation}\label{eq_fundamental2}\mr{abs}(U,\ovl L)\le\frac{(\ovl L^{d+1})}{(d+1)\displaystyle \int_{U(\C)}c_1(\ovl L)^d}.\end{equation}

An \emph{effective section} $s$ of $\ovl L$ is a rational section of $L$ such that $\widehat{\mr{div}}(s)$ is an effective adelic divisor. If $s$ is an effective section of $\ovl L$, then the height function $h_{\ovl L}(\cdot)$ is non-negative on $(U\setminus|\mr{div}(s)|)(\ovl\Q)$.

Denote the logarithm of the number of effective sections by $\widehat h^0(\ovl L)$, which is finite due to \cite[Lemma 5.1.5]{YuanZhang}. The \emph{volume} of $\ovl L$ is defined as
$$\vol(\ovl L)=\lim_{n\to\infty}\frac{}{}\frac{\widehat h^0(n\ovl L)}{n^{d+1}/(d+1)!}.$$
This limit always exists and satisfies the following arithmetic Siu's inequality \cite[Theorem 5.2.2.(2)]{YuanZhang}: for nef adelic line bundles $\ovl L_1$ and $\ovl L_2$ on $U$,
\begin{equation}\label{eq_Siu_inequality}
    \vol(\ovl L_1-\ovl L_2)\ge(\ovl L_1^{d+1})-(d+1)(\ovl L_1^d\cdot\ovl L_2).
\end{equation}

Recall that in \cite[\S 5]{YuanZhang}, we say \begin{itemize}
    \item $\ovl L$ is \emph{big} if $\vol(\ovl L)>0$,
    \item $\ovl L$ is \emph{pseudo-effective} if for any big adelic line bundle $\ovl M$, $\ovl L+\ovl M$ is big.
\end{itemize}
The following is an easy consequence of \eqref{eq_Siu_inequality}.
\begin{lemma}\label{lemma_height_ineq_nef_and_big}
    Let $\ovl L$ be an adelic line bundle on $U$ such that its restriction on any closed subvariety is nef and big. Then for any integrable adelic line bundle $\ovl M$, there exists a constant $\ell>0$ depending on $\ovl L$ and $\ovl M$, such that $$h_{\ovl L}(x)\geq \ell h_{\ovl M}(x)$$
    for any $x\in U(\ovl{\Q}).$
\end{lemma}
\begin{proof}
    We proceed by induction on $\dim U$. Let $\ovl M_1$ and $\ovl M_2$ be nef adelic line bundles on $U$ such that $\ovl M=\ovl M_1-\ovl M_2$. Then by \eqref{eq_Siu_inequality}
    \begin{align*}\vol(N\ovl L-\ovl M)&=\vol(N\ovl L+\ovl M_2-\ovl M_1)\\&=(N+\ovl L_1+\ovl M_2)^{d+1}-(d+1)((N\ovl L_1+\ovl M_2)^{d}\cdot \ovl M_1)
    \end{align*}
    which is greater than $0$ when $N\gg 0.$ Let $s$ be an effective section of a multiple of $N\ovl L-\ovl M$. Then $h_{N\ovl L-\ovl M}(x)>0$ on $(U\setminus\lvert\mathrm{div}(s)\rvert)(\ovl{\Q})$. By applying the induction hypothesis to each irreducible component of $\mathrm{div}(s)$, there exists a constant $\ell'$ such that $h_{\ovl L}(x)\geq\ell' h_{\ovl M}(x)$ on $\lvert\mr{div}(s)\rvert(\ovl {\Q})$. We then conclude the proof by putting $\ell:=\displaystyle\min\left\{\frac{1}{N},\ell'\right\}$.
\end{proof}

\section{Canonical adelic line bundles on semiabelian schemes}\label{section_canonical}

In this subsection, We define the canonical adelic line bundle on semiabelian schemes. It is a sum of two parts: one from the abelian scheme and the other from a toric compactification.

We first recall some facts from \cite[\S 6]{YuanZhang}. We quote the part we need here for the reader's convenience.
Let $S$ be a quasi-projective variety over $\Q$. We consider the following triplet $(X,f,L)$ over $S$ consists of:
\begin{enumerate}
    \item[(1)]an integral scheme $X$ projective and flat over $S$,
    \item[(2)]an $S$-morphism $f:X\to X$,
    \item[(3)]a line bundle $L$ on $X$ such that $f^*L=qL$ for some integer $q>1$.
\end{enumerate}
Then the line bundle $L$ has a unique \emph{invariant adelic extension} $\ovl L$ \cite[Theorem 6.1.2(1)]{YuanZhang}, i.e., $\ovl L$ is an adelic line bundle on $X$ with the underlying line bundle $L$ such that $f^*\ovl L=q\ovl L$.

Moreover, if $L$ is relatively ample over $S$, then $(X,f,L)$ is called a \emph{polarized dynamical system}. In this case, the extension $\ovl L$ is further nef \cite[Theorem 6.1.1(2)]{YuanZhang}.
\subsection{Abelian part}\label{subsec_abelian_canonical}

We first review the canonical adelic line bundle on abelian schemes constructed in \cite[\S 6]{YuanZhang}. 

Let $\AAA/S$ be an abelian scheme. By a \emph{symmetric polarization} of $\AAA$ we mean a line bundle $L$ on $\AAA$ satisfying the following conditions:
\begin{enumerate}
    \item[(1)]$L$ is relatively ample over $S$;
    \item[(2)]$L$ is symmetric, i.e., $[-1]^*L$ is isomorphic to $L$;
    \item[(3)]$L$ is rigidified, i.e., for the zero section $e:S\to\AAA$, there is an isomorphism $e^*L=\mathcal O_S$, called a \emph{rigidification}. The rigidification is part of the data.
\end{enumerate}
Symmetricity implies that for any integer $n$ there is a unique isomorphism $[n]^*L=n^2L$ compatible with the rigidification.

If $S$ is a quasi-projective variety over $\mathbb{Q}$, then $(\AAA,[2],L)$ is a polarized dynamical system and hence $L$ has a unique invariant adelic extension $\ovl L\in\Pic(\AAA)$ satisfying $[2]^*\ovl L=4\ovl L$. The uniqueness implies $[n]^*\ovl L=n^2\ovl L$ for any integer $n$. The \emph{canonical height function} $\widehat h(\cdot)$ on $\AAA(\ovl\Q)$ is defined to be the height function $h_{\ovl L}(\cdot)$ associated with $\ovl L$.

Assume that $\mc A\rightarrow S$ is equipped with a principal polarization, up to a finite cover, we have the modular map
\[
\begin{tikzcd}
  \mathcal{A} \ar[d] \ar[r,"\iota_{\mathcal A}"] & \mathcal{A}_g \ar[d] \\
  S \ar[r,"\iota_S"] & \mathbb A_g
  \ar[ul, phantom, "\ulcorner", pos=0.45] 
\end{tikzcd}.
\]
There is a line bundle $\mc L_{\mr{pp}}$ on $\mathcal A_g$ constructed in \cite[\S 8.7]{ChristinaLange2004CAV} such that for each $s\in \mathbb A_g(\C)$, $(\mc L_{\mr{pp}})_s$ is the principal polarization on $(\mc A_g)_s.$
We set \begin{equation}\label{eq_symmstrification_pp}
    L:=\iota_{\mc A}^* (\mc L_{\mr{pp}}+[-1]^* \mc L_{\mr{pp}}).
\end{equation}
 which is a symmetric polarization of $\mc A$. 
Then we have \begin{equation}\label{eq_first_Chern_form_pullback}
    c_1(\ovl L)=\iota_A^*\omega_\ab
\end{equation}
by \cite[Proof of Lemma 7]{kühne2024equi}.

\subsection{Compactification}\label{subsec_toric_compactification}

We review the compactification of semiabelian schemes. In this article, we only consider compactifications of type $(\mathbb{P}^1)^t$. See \cite[\S 3]{CL} for general toric compactification.

The complement of the natural inclusion $\mathbb{G}_{m}\subset\mathbb{P}^1$ consists of two irreducible divisors, which we denote by $D_{[0]}$ and $D_{[\infty]}$. For general positive integer $t$, the complement of the $t$-fold product $\mathbb{G}_{m}^t\subset(\mathbb{P}^1)^t$ consists of $2t$ divisors, namely the pull-back $D_{[0]}^{(i)}$ and $D_{[\infty]}^{(i)}$ of $D_{[0]}$ and $D_{[\infty]}$ along the $i$-th projection for $i=1,2\dots t$. Moreover, the map $[n]:\mathbb{G}_{m}^t\to\mathbb{G}_{m}^t$ extends to $[n]:(\mathbb{P}^1)^t\to(\mathbb{P}^1)^t$ and $$[n]^*D_{[0]}^{(i)}=nD_{[0]}^{(i)},\quad[n]^*D_{[\infty]}^{(i)}=nD_{[\infty]}^{(i)}.$$

Let $\mathcal G$ be a semiabelian scheme over $S$ as in \S \ref{sec_nondeg_fiberproduct}.

The pushout
$$\ovl\GGG=(\GGG\times(\mathbb{P}^1)^t)/\mathbb{G}_{m}^t$$
of $\mc G$ along $\mathbb{G}_{m}^t\to(\mathbb{P}^1)^t$ contains $\GGG$ as an open subset and is called the \emph{compactification} of $\GGG$. For $i=1,2\dots t$, the divisors $D_{[0]}^{(i)}$ and $D_{[\infty]}^{(i)}$ on $\GGG\times(\mathbb{P}^1)^t$ descend to $\ovl\GGG$, for which we use the same notation $D_{[0]}^{(i)}$ and $D_{[\infty]}^{(i)}$. Again, the multiplication by $n$ map $[n]:\mc G\rightarrow \mc G$ extends to $[n]:\ovl{\mc G}\rightarrow \ovl {\mc G}$, and
$$[n]^*D_{[0]}^{(i)}=nD_{[0]}^{(i)},\quad[n]^*D_{[\infty]}^{(i)}=nD_{[\infty]}^{(i)}.$$
Denote by
\begin{equation}\label{eq_boundary_line_bundle}M=\sum_{i=1}^t\mathcal O(D_{[0]}^{(i)}+D_{[\infty]}^{(i)})\end{equation}
the total toric line bundle.

\subsection{Toric canonical line bundle}\label{subsection_tor_canonical}
We keep the same hypothesis as in last subsection. There exist unique invariant adelic extensions $\ovl D_{[0]}^{(i)}$ (resp. $\ovl D_{[\infty]}^{(i)}$) with respect to the triplet $(\ovl\GGG,[2],D_{[0]}^{(i)})$ (resp. $(\ovl\GGG,[2],D_{[\infty]}^{(i)})$). Here the invariance means that for any integer $n$,
$$[n]^*\ovl D_{[0]}^{(i)}=n\ovl D_{[0]}^{(i)}\quad(\mr{resp.\ }[n]^*\ovl D_{[\infty]}^{(i)}=n\ovl D_{[\infty]}^{(i)}).$$
Denote by
$$\ovl M=\sum_{i=1}^t\mathcal{O}(\ovl D_{[0]}^{(i)}+\ovl D_{[\infty]}^{(i)})$$
the invariant adelic extension of $M$. Note that $(\ovl\GGG,[2],M)$ is not a polarized dynamical system due to the lack of relatively ampleness.
But we will show that $\ovl M$ is integrable.

\begin{proposition}
\label{integrable}
For any symmetric polarization $L$ of $\AAA$ and its invariant adelic extension $\ovl L$, there is a nef adelic line bundle $\ovl H$ on $S$ and a positive integer $m$ such that 
$$\mathcal{O}(\ovl D_{[0]}^{(i)})+m\ovl L+\pi_{\mathcal G/S}^*\ovl H \quad(\mr{resp.\ }\mathcal{O}(\ovl D_{[\infty]}^{(i)})+n\ovl L+\pi_{\mathcal G/S}^*\ovl H)$$
is nef.
In particular, $\ovl D_{[0]}^{(i)}$ and $\ovl D_{[\infty]}^{(i)}$ are integrable.
\end{proposition}

\begin{proof}
We only prove the statement for $\ovl D_{[0]}^{(i)}$. The proof for $\ovl D_{[\infty]}^{(i)}$ is verbatim.

We only need to consider the toric rank $1$ case. Indeed, let $p_i:\overline\GGG\to\overline\GGG_i$ (resp. $p_i:\GGG\to\GGG_i$) be the push-forward along the $i$-th projection map $(\mathbb{P}^1)^t\to\mathbb{P}^1$ (resp. $\mathbb{G}_{m}^t\to\mathbb{G}_{m}$). Then $\GGG_i$ is a semiabelian scheme and $\ovl\GGG_i$ is its compactification. Let $D_{[0]}$ be the divisor on $\ovl\GGG_i$. Then $p_i^*D_{[0]}=D_{[0]}^{(i)}$. Once the statement is true for $\ovl D_{[0]}$, then so is it for $\ovl D_{[0]}^{(i)}=p_i^*\ovl D_{[0]}$. Thus, we may and do assume $t=1$ and omit the superscript $(i)$.

Let $\AAA^\vee$ be the dual abelian scheme of $\AAA$ and $\PPP$ the Poincaré torsor over $\AAA\times_S\AAA^\vee$.
Then we have the following properties:
\begin{enumerate}
    \item $\PPP$ is a biextension of the pair $(\mathcal A,\mathcal A^\vee)$ by $\mathbb G_m$, that is, $\PPP$ is a semiabelian scheme over $\AAA^\vee$ with the abelian part $\mathcal A\times_S\AAA^\vee$, and $\PPP$ is also a semiabelian scheme over $\mathcal A$ with the abelian part $\mathcal A^\vee\times_S\mathcal A$.
    \item Applying the Weil--Barsotti formula \cite[\S III.18]{Oort1966group} fiberwisely, there is a section $s:S\to\AAA^\vee$ such that $\GGG=s^*\PPP.$
    \item Let $\ovl\PPP=(\PPP\times\mathbb P^1)/\mathbb{G}_m$ be the $\mathbb P^1$-toric compactification of $\PPP$ such that $\ovl\GGG=s^*\ovl \PPP$. Then the boundary divisors of $\ovl\GGG$ are exactly the pull-back of the boundary divisors of $\ovl\PPP$. Moreover, $\ovl\PPP$ is also the $\mathbb P^1$-toric compactification of $\PPP$ as a semiabelian scheme over $\AAA$.
    \item The the fiberwise multiplication by $n$ map $[n]_{\PPP}$ (resp. $[n^\vee]_{\PPP}$) of $\PPP$ over $\AAA^\vee$ (resp. $\AAA$), extends uniquely to $\ovl\PPP$. We denote the extension as $[n]_{\ovl \PPP}$ (resp. $[n^\vee]_{\ovl \PPP}$).
\end{enumerate}
To summarize, we have the following commutative diagrams:
\[\begin{tikzcd}
    \ovl\GGG &
	\ovl\PPP && \\
	& & {\AAA\times_S\AAA^\vee} & \AAA \\
	&S & {\AAA^\vee} & S
    \arrow[from=1-1, to=1-2]
    \arrow["s", from=3-2, to=3-3]
    \arrow["{\pi_{\ovl \GGG/S}}"', curve={height=18pt}, from=1-1, to=3-2]
	\arrow["\pi"', from=1-2, to=2-3]
	\arrow["{\pi_{\ovl \PPP/\AAA}}", curve={height=-18pt}, from=1-2, to=2-4]
	\arrow["\pi_{\ovl \PPP/\AAA^\vee}"', curve={height=18pt}, from=1-2, to=3-3]
	\arrow["{p'}", from=2-3, to=2-4]
	\arrow["p"', from=2-3, to=3-3]
	\arrow[from=2-4, to=3-4]
	\arrow[from=3-3, to=3-4]
\end{tikzcd}\]
and 
\begin{equation}\label{cd_multiplication}
    \begin{tikzcd}
    \ovl{\PPP} & \ovl{\PPP} \\
    {\AAA\times\AAA^\vee} & {\AAA\times\AAA^\vee} \\
    \AAA & \AAA
    \arrow["{[n]_{\ovl{\PPP}}}", from=1-1, to=1-2]
    \arrow[from=1-1, to=2-1]
    \arrow[from=1-2, to=2-2]
    \arrow["{[n]_{\AAA}\times_S \mr{id}_{\AAA^\vee}}"', from=2-1, to=2-2]
    \arrow["{p'}"', from=2-1, to=3-1]
    \arrow["{p'}", from=2-2, to=3-2]
    \arrow["{[n]_{\AAA}}"', from=3-1, to=3-2]
    \arrow["{\pi_{\ovl{\PPP}/\AAA}}"', curve={height=24pt}, from=1-1, to=3-1]
    \arrow["{\pi_{\ovl{\PPP}/\AAA}}", curve={height=-24pt}, from=1-2, to=3-2]
\end{tikzcd}
\quad
\begin{tikzcd}
    \ovl{\PPP} & \ovl{\PPP} \\
    {\AAA\times\AAA^\vee} & {\AAA\times\AAA^\vee} \\
    \AAA^\vee & \AAA^\vee
    \arrow["{[n^\vee]_{\ovl{\PPP}}}", from=1-1, to=1-2]
    \arrow[from=1-1, to=2-1]
    \arrow[from=1-2, to=2-2]
    \arrow["{\mr{id}_\AAA\times[n]_{\AAA^\vee}}"', from=2-1, to=2-2]
    \arrow["{p}"', from=2-1, to=3-1]
    \arrow["{p}", from=2-2, to=3-2]
    \arrow["{[n]_{\AAA^\vee}}"', from=3-1, to=3-2]
    \arrow["{\pi_{\ovl{\PPP}/\AAA^\vee}}"', curve={height=24pt}, from=1-1, to=3-1]
    \arrow["{\pi_{\ovl{\PPP}/\AAA^\vee}}", curve={height=-24pt}, from=1-2, to=3-2]
\end{tikzcd}.
\end{equation}
Notice that if the statement of the proposition is true for the boundary divisors of $\ovl \PPP$ over $\mathcal A^\vee$, then it is also true for the boundary divisors of $\ovl \GGG$ over $S$ via pull-backs. Hence we only need to consider the case of $\mathcal P/\mathcal A^\vee$. By abuse of notation, we still use $D_{[0]}$ to denote the boundary divisor on $\ovl \PPP$.

From the property \textnormal{(3)} and \textnormal{(4)}, we can see that
$$[n]_{\ovl\PPP}^*D_{[0]}=[n^\vee]_{\ovl\PPP}^*D_{[0]}=nD_{[0]}.$$

Let $L'$ be a symmetric polarization of $\AAA^\vee$ and $\ovl L'$ its invariant adelic extension. Since $p'^*L+p^*L'$ is relatively ample over $S$ and $D_{[0]}$ is relatively ample over $\AAA\times_S\AAA^\vee$, we can take an integer $m$ large enough such that the line bundle
$$N=\mathcal{O}(D_{[0]})+m(\pi_{\ovl \PPP/\AAA}^*L+\pi_{\ovl \PPP/\AAA^\vee}^*L')$$
on $\ovl\PPP$ is relatively ample over $S$ (see \cite[Proposition 13.65]{GortzWedhorn}).

We claim that $(\ovl\PPP,[2]_{\ovl\PPP}\circ[2^\vee]_{\ovl\PPP},N)$ is a polarized dynamical system over $S$. By the commutative diagram \eqref{cd_multiplication}, we have the isomorphism \begin{equation*}\label{eq_inv_dsh}
    \begin{aligned}
([2]_{\ovl{\PPP}}\circ[2^\vee]_{\ovl{\PPP}})^* N
&= [2^\vee]_{\ovl{\PPP}}^*\,[2]_{\ovl{\PPP}}^*\Bigl(\mathcal{O}(D_{[0]}) + m\pi_{\ovl{\PPP}/\AAA}^*L + m\pi_{\ovl{\PPP}/\AAA^\vee}^*L'\Bigr)\\
&= \mathcal{O}(4D_{[0]}) + m\,\pi_{\ovl{\PPP}/\AAA}^*[2]_{\AAA}^*L + m\,\pi_{\ovl{\PPP}/\AAA^\vee}^*[2^\vee]_{\AAA^\vee}^*L'=4N.
\end{aligned}
\end{equation*}Therefore, there exists a unique nef invariant adelic extension $\ovl  N$ of $N$. Since $\ovl N-m(\pi_{\ovl \PPP/\AAA^\vee}^*\ovl  L+\pi_{\ovl \PPP/\AAA}^*\ovl L')$ is an invariant extension of $D_{[0]}$, by the uniqueness we obtain $\ovl D_{[0]}=\ovl N-m(\pi_{\ovl \PPP/\AAA^\vee}^*\ovl  L+\pi_{\ovl \PPP/\AAA}^*\ovl L')$, which is integrable.
\end{proof}

\begin{remark}
    The nef line bundle $\ovl H$ in the proposition can be taken to be the pull-back of $m\ovl L'$ via the section $s:S\rightarrow \mc A^\vee$.
\end{remark}

\subsection{Relation with the Betti map}\label{subsec_betti_adelic_relation} 
Let $\widetilde S$ be the universal covering of $S(\C)$.
Recall the toric Betti map $b_{\tor}:\GGG_{\widetilde S}\to(\C^\times)^t$ and its composition $b_{\tor,i}$ with the $i$-th projection $(\C^\times)^t\rightarrow \C^\times$. Here $\GGG_{\widetilde S}$ is the pull-back of $\GGG(\C)$ along the covering map $\widetilde S\to S(\C)$.

\begin{proposition}\label{prop_dsh_toric_betti}
Let $g_{[0]}^{(i)}$ (resp. $g_{[\infty]}^{(i)}$) be the Green function of $\ovl D_{[0]}^{(i)}$ (resp. $\ovl D_{[\infty]}^{(i)}$) on $\GGG(\C)$ and $\widetilde{g}_{[0]}^{(i)}$ (resp. $\widetilde{g}_{[\infty]}^{(i)}$) its pull-back to $\GGG_{\widetilde S}$. Then
$$\widetilde{g}_{[0]}^{(i)}(\cdot)=\max\{-\log|b_{\tor,i}(\cdot))|,0\}\quad(\mr{resp.\ }\widetilde{g}_{[\infty]}^{(i)}(\cdot)=\max\{\log|b_{\tor,i}(\cdot))|,0\}).$$
In particular, the toric Betti currents $\widetilde \omega_{\tor,i}$ descend to currents $\omega_{\tor,i}$ on $\mc G$, which are differences of semipositive $(1,1)$-currents.
\end{proposition}

\begin{proof}
We may assume $t=1$ and only deal with $\widetilde g_{[0]}$. Let
$$f(x)=\widetilde{g}_{[0]}(x)-\max\{-\log|b_{\tor}(x)|,0\}$$
be the difference. 

Since $b_{\tor}$ is a group homomorphism and $[n]^*\ovl D_{[0]}=n\ovl D_{[0]}$, we have $f(nx)=nf(x)$. In fact, it suffices to prove $f=0$ on $(\mathbb{G}_{m})_{\widetilde{S}}\subset\GGG_{\widetilde{S}}$. Indeed, this implies $f(x)=0$ for those $x\in\GGG_{\widetilde{S}}$ satisfying $nx\in(\mathbb{G}_{m})_{\widetilde{S}}$ for some positive integer $n$. Such points $x$ are dense in the analytic topology in $\GGG_{\widetilde{S}}$. The continuity of $f$ concludes the proof. The restriction $b_{\tor}|_{(\mathbb G_m)_{\widetilde S}}$ is just the projection $(\mathbb G_m)_{\widetilde S}(\C)=\widetilde S(\C)\times\C^\times\to\C^\times.$ 

On the other hand, the Zariski closure of $\mathbb{G}_{m,S}$ in $\ovl \GGG$ is exactly $\mathbb P^1_S.$ Moreover, by the uniqueness of the invariant adelic line bundle, the restriction of $\ovl D_{[0]}$ on it is nothing but the pull-back of $(D_{[0],\ZZ},\max\{-\log\lvert\cdot\rvert,0\})$ via $\mathbb P_S^1\rightarrow \mathbb P^1_\Z.$ This concludes the proof for the first assertion.

For the second assertion, $\omega_{\tor}$ descends to the current $dd^c (g_{[0]}+g_{[\infty]})$ by the definition. We thus conclude the proof by the integrability of $g_{[0]}+g_{[\infty]}.$
\end{proof}

\section{Intersection of canonical line bundles}\label{section_intersection}
Let $\mc G$ be a principally polarized semiabelian scheme over a quasi-projective $\Q$-variety $S$. Let $L$ be as in \eqref{eq_symmstrification_pp}, and $\ovl L$ be its invariant adelic extension.
For simplicity of language, from now on we still denote by $L$ (resp. $\ovl L$) its pull-back to $\ovl\GGG$ if there is no ambiguity.

The canonical height function on $\GGG$ is defined as
\begin{equation}\label{eq_canonical_height}
\widehat h(\cdot)=h_{\ovl L}(\cdot)+h_{\ovl M}(\cdot).\end{equation}
\subsection{Top Betti current}\label{subsec_top_betti}

For any subset $I\subset\{1,2,\dots,t\}$, denote
\begin{align*}
    \ovl M_I:=\sum_{i\in I}\mathcal{O}(\ovl D_{[0]}^{(i)}+\ovl D_{[\infty]}^{(i)}),\quad\widehat h_I(\cdot)=h_{\ovl L}(\cdot)+h_{\ovl M_{I}}(\cdot)
\end{align*}
and $b_I:\GGG_{\widetilde S}\to(\C^\times)^{\lvert I\rvert}$ the composition of $b_{\tor}$ with the projection map $(\C^\times)^t\to(\C^\times)^{\lvert I\rvert}$, where $(\C^\times)^{\lvert I\rvert}$ is the product of all $i$-th components of $(\C^\times)^t$ for $i\in I$.
Define
$$\GGG_{I,\mathbb S^1}:=\{x\in\GGG(\C):g_{[0]}^{(i)}(x)=g_{[\infty]}^{(i)}(x)=0,i\in I\}\subset \GGG(\C).$$
By Proposition \ref{prop_dsh_toric_betti}, its inverse image in $\GGG_{\widetilde{S}}$ is $b_I^{-1}((\mathbb{S}^1)^{\lvert I\rvert})$, where $\mathbb S^1$ is the unit circle in $\C^\times$.
Let $\omega_{\mathbb{S}^1}$ be twice the unit Haar measure on $\mathbb{S}^1$, viewed as a smooth $1$-form, and
$$\omega_I=\sum_{i\in I}p_i^*\omega_{\mathbb S^1}.$$
Then $b_I^*\omega_I$ on $b_I^{-1}((\mathbb{S}^1)^{\lvert I\rvert})$ descends to $\GGG_{I,\mathbb S^1}$, which we also denote by $\omega_I$.

Let $\XXX\subset\GGG$ be an irreducible closed subvariety. For any subset $I\subset\{1,2,\dots,t\}$, and a smooth point $x\in\XXX(\C)$, $\XXX$ is called \emph{$I$-non-degenerate} at $x$ if 
$$\mr{rank}_\R(d\widetilde{(b_I,b_{\ab})}|_{
\widetilde{\XXX(\C)}})_{\widetilde x}=2\dim\XXX.$$
The set $U$ of all $x$ at which $\mc X$ is $I$-non-degenerate is called the \emph{$I$-non-degeneracy locus}. 

From now on, Let $2r$ be the abelian Betti rank of $\mc X$, i.e., the maximal rank of the tangent map of $\widetilde b_{\ab}$'s restriction on $\widetilde {\XXX(\C)}^{\mr{sm}}$. We further assume that $$\lvert I\rvert=\dim\XXX-r.$$

We need the following Fubini type result on iterated integrals. We say a function is \emph{dsh} if it can be written as a difference of two plurisubharmonic functions.

\begin{lemma}\label{lemm_fubini}
Let $M$ be a complex manifold of dimension $d$. Let $N$ be a real manifold of dimension $2r$. Let $(p,q):M\to N\times\C^{d-r}$ be a real-analytic local diffeomorphism such that for any $x\in N$, $p^{-1}(x)$ is complex analytic and $q|_{p^{-1}(x)}:p^{-1}(x)\to\C^{d-r}$ is holomorphic. Let $\alpha$ be a smooth $2r$-form on $N$ such that $p^*\alpha$ is semi-positive on $M$. Let $g$ be a continuous dsh function on $M$. Denote by $\beta=(dd^cg)^{d-r}$ the current on $M$. Then for any continuous function $f$ with compact support on $M$,
$$(x\in N)\mapsto\int_{p^{-1}(x)}f\beta|_{p^{-1}(x)}$$
is a continuous function with compact support, and
$$\int_Mf\beta\wedge p^*\alpha=\int_N\Big(\int_{p^{-1}(x)}f\beta|_{p^{-1}(x)}\Big)\alpha.$$
\end{lemma}

\begin{proof}
Take a uniform approximation $(g_m)_{m\ge1}$ of $g$ by smooth dsh functions on the support of $f$. Denote $\beta_m=(dd^cg_m)^{d-r}$. By definition
$$\int_Mf\beta\wedge p^*\alpha=\lim_{m\to\infty}\int_Mf\beta_m\wedge p^*\alpha,$$
and
$$\int_{p^{-1}(x)}f\beta|_{p^{-1}(x)}=\lim_{m\to\infty}\int_{p^{-1}(x)}f\beta_m|_{p^{-1}(x)}.$$
By the Fubini theorem for smooth forms, $\displaystyle\int_{p^{-1}(x)}f\beta_m|_{p^{-1}(x)}$ is continuous in $x$ and
$$\int_Mf\beta_m\wedge p^*\alpha=\int_N\Big(\int_{p^{-1}(x)}f\beta_m|_{p^{-1}(x)}\Big)\alpha.$$
It remains to show that $\displaystyle\int_{p^{-1}(x)}f\beta_m|_{p^{-1}(x)}$ converges to $\displaystyle\int_{p^{-1}(x)}f\beta|_{p^{-1}(x)}$ uniformly in $x$. By partition of unity, we may assume $(p,q)$ is injective on the support of $f$. Then $\displaystyle\int_{p^{-1}(x)}f\beta_m|_{p^{-1}(x)}$ can be viewed as an integration in $\C^{d-r}$. Notice that $$\left\lvert\displaystyle\int_{p^{-1}(x)}f(\beta_m-\beta)|_{p^{-1}(x)}\right\rvert\leq \sum_{i=1}^{d-r} \left\lvert\displaystyle\int_{p^{-1}(x)}f(dd^c (g_m-g)(dd^c g_m)^{d-r-i}(dd^c g)^{i-1})|_{p^{-1}(x)}\right\rvert,$$ we thus conclude the proof by the Chern--Levine--Nirenberg inequality \cite{ChernLevineNirenberg}.
\end{proof}

\begin{corollary}\label{coro_null_mass}
We keep the same hypothesis as in Lemma \ref{lemm_fubini}. Let $Z\subset M$ be a closed subset. Assume for any $x\in N$, $Z\cap p^{-1}(x)$ is null with respect to $\beta$ on $p^{-1}(x)$. Then $Z$ is null with respect to $\beta\wedge p^*\alpha$
\end{corollary}

\begin{proof}
We may assume $Z$ is compact. By an approximation of $Z$ we mean a family $(f_n)_{n\ge1}$ of continuous functions with compact support on $M$ such that $f_n=1$ on $Z$, and $(f_n)_{n\ge 1}$ decrease pointwise to $0$ outside $Z$. We can always construct such an approximation.

Applying the dominated convergence theorem, the signed measure $(\beta\wedge p^*\alpha)(Z)$ of $Z$ equals the limit of $\displaystyle\int_{M}f_n\beta\wedge p^*\alpha$ as $n\to\infty$. For any $x\in N$, $(f_n|_{p^{-1}(x)})_{n\ge1}$ is an approximation of $Z\cap p^{-1}(x)$. So $\displaystyle\int_{p^{-1}(x)}f_n\beta|_{p^{-1}(x)}\to0$ as $n\rightarrow\infty$, since $Z\cap p^{-1}(x)$ is null with respect to $\beta$. Using Lemma \ref{lemm_fubini}, we obtain that
$$\int_{M}f_n\beta\wedge p^*\alpha=\int_N\Big(\int_{p^{-1}(x)}f_n\beta|_{p^{-1}(x)}\Big)\alpha(x).$$
Recall that $\beta=(dd^c g)^{d-r}$, we can decompose $\beta$ into the positive part and the negative part.
By applying the monotone convergence theorem to both parts respectively,
$$\lim_{n\rightarrow\infty}\int_N\Big(\int_{p^{-1}(x)}f_n\beta|_{p^{-1}(x)}\Big)\alpha(x)=\int_N\lim_{n\rightarrow\infty}\Big(\int_{p^{-1}(x)}f_n\beta|_{p^{-1}(x)}\Big)\alpha(x)=0,$$
which concludes the proof.
\end{proof}

\begin{theorem}\label{measure}
Let $U\subset\XXX^{\mr{sm}}(\C)$ be the $I$-non-degeneracy locus. Denote by $j:U\cap\GGG_{I,\mathbb S^1}\to\XXX(\C)$ the inclusion map. Then
$$c_1(\ovl L)^rc_1(\ovl M_I)^{\lvert I\rvert}=j_*(c_1(\ovl L)^r\omega_I^{\lvert I\rvert}).$$
\end{theorem}

\begin{proof}
Since $\XXX(\C)\setminus\XXX^{\mr{sm}}(\C)$ is pluripolar, it is null with respect to $c_1(\ovl L)^rc_1(\ovl M_I)^{\lvert I\rvert}$. Let $V\subset\XXX^{\mr{sm}}(\C)$ be the open subset where the maximal abelian Betti rank $2r$ is attained. Then the smooth $(r,r)$-form $c_1(\ovl L)^r$ vanishes outside $V$. Thus, $\XXX^{\mr{sm}}(\C)\setminus V$ is also null.

Note that for any $x\in U$, $$\mr{rank}_\R(d(\widetilde{b_{\ab},b_I})|_{\widetilde{\XXX}}))_{\widetilde x}=2\dim\XXX$$ implies that $$\mr{rank}_\R(d(\widetilde b_{\ab}|_{\widetilde \XXX}))_{\widetilde x}=2r.$$ Indeed, 
\begin{align*}2\dim\XXX&=\mr{rank}_\R(d(\widetilde{b_{\ab},b_I})|_{\widetilde{\XXX}}))_{\widetilde x}\\&\le\mr{rank}_\R(d(\widetilde b_{\ab}|_{\widetilde \XXX}))_{\widetilde x}+\mr{rank}_\R(d(\widetilde b_{I}|_{\widetilde \XXX}))_{\widetilde x}\\&\le 2r+2(\lvert I\rvert)=2\dim\XXX.
\end{align*}
So the equalities hold. We get $U\subset V$.

For any point $y\in V$, we may take an open neighborhood $M\subset V$ whose closure $\ovl M$ is compact and simply connected. 
Since $b_{\ab}$ has constant rank $2r$ near $y$, after shrinking $M$, $b_{\ab}$ factors through a submersion $p:M\to N$ and an immersion $N\to(\R/\ZZ)^{2g}$ of real manifolds. Recall that $c_1(\ovl L)$ is a pull-back along $b_{\ab}$ by \eqref{eq_currents_pull_back_via_betti} and \eqref{eq_first_Chern_form_pullback}. So there is a $(2r)$-form $\alpha$ on $N$ such that $p^*\alpha=c_1(\ovl L)^r$.

\begin{equation}\label{eq_cd_real_manifolds}\begin{tikzcd}
	M & {(\R/\ZZ)^{2g}} \\
	N
	\arrow["{b_{\ab}}", from=1-1, to=1-2]
	\arrow["p"', from=1-1, to=2-1]
	\arrow[from=2-1, to=1-2]
\end{tikzcd}\end{equation}


Notice that the restriction of $b_{I}$ on $p^{-1}(p(y))=b_{\ab}^{-1}b_{\ab}(y)\cap M$ is holomorphic by the formula \eqref{eq_Betti_map}.
For any continuous function $f$ with compact support on $M$, by Lemma \ref{lemm_fubini},
$$\int_Mfc_1(\ovl L)^rc_1(\ovl M_I)^{\lvert I\rvert}=\int_N\Big(\int_{p^{-1}(x)}fc_1(\ovl M_I|_{p^{-1}(x)})^{\lvert I\rvert}\Big)\alpha.$$

For $x\in N$, $\dim (b_{I}((V\setminus U)\cap p^{-1}(x))<\lvert I\rvert$. Hence $(V\setminus U)\cap p^{-1}(x)$ is null with respect to $c_1(\ovl M_I)^{\lvert I\rvert}$ again by the holomorphicity \eqref{eq_Betti_map}. So $(V\setminus U)\cap M$ is null with respect to $c_1(\ovl L)^rc_1(\ovl M_I)^{\lvert I\rvert}$ due to Corollary \ref{coro_null_mass}. By taking a countable collection of $M$'s that covers $V$, we see that $V\setminus U$ is so.

Assume $M\subset U$ in the sequel. 
Since $p|_{U\cap\GGG_{I,\mathbb S^1}}$ is submersive onto $N$, by the Fubini theorem for smooth forms,
$$\int_{M\cap\GGG_{I,\mathbb S^1}}(j^*f)c_1(\ovl L)^r\omega_I^{\lvert I\rvert}=\int_N\Big(\int_{p^{-1}(x)\cap\GGG_{I,\mathbb S^1}}(j^*f)\omega_I^{\lvert I\rvert}\Big)\alpha.$$

It remains to prove
$$c_1(\ovl M_I|_{p^{-1}(x)})^{\lvert I\rvert}=j_*(\omega_{I}|_{p^{-1}(x)\cap\GGG_{I,\mathbb S^1}}^{\lvert I\rvert}).$$
Recall that $b_I^*$ commutes with $dd^c$ due to holomorphicity, and the Green function of $\ovl M_I$ is the pull-back of $\sum_{i\in I}|\log|z_i||$. So we only need to prove
$$(\sum_{i=1}^ndd^c|\log|z_i||)^n=j_*(\sum_{i=1}^np_i^*\omega_{\mathbb S^1})^n.$$
on $(\C^\times)^{n}$. The case that $n=1$ can be verified by direct calculation. We then conclude the proof by induction on $n$.
\end{proof}

\subsection{Non-degeneracy in terms of intersection theory}

\begin{theorem}\label{theorem_I_non_deg_integral}
Let $\XXX\subset\GGG$ be a subvariety. Let $2r$ be its abelian Betti rank. Let $I\subset\{1,2,\dots,t\}$ be a $(\dim\XXX-r)$-element subset. Then $\XXX(\C)$ is $I$-non-degenerate if and only if
$$\int_{\XXX(\C)}c_1(\ovl L)^{r}c_1(\ovl M_I)^{\dim\XXX-r}>0.$$
\end{theorem}

\begin{remark}\label{remark_I_non_deg_equiv_non_deg}
For completeness, we remark that $\XXX$ is non-degenerate if and only if
$$\int_{\XXX(\C)}c_1(\ovl L)^{r}c_1(\ovl M)^{\dim\XXX-r}>0.$$
This can be derived from the following two facts: \begin{enumerate}
    \item $\XXX$ is non-degenerate if and only if it is $I$-non-degenerate for some $I$ with $\lvert I\rvert=\dim \mc X-r$;
    \item $$\int_{\XXX(\C)}c_1(\ovl L)^{r}c_1(\ovl M)^{\dim\XXX-r}>0$$
if and only if
$$\int_{\XXX(\C)}c_1(\ovl L)^{r}c_1(\ovl M_I)^{\dim\XXX-r}>0$$
for some $I$.
\end{enumerate}
 To see the second fact, note that the proof of Theorem \ref{measure} actually implies that each term
$$c_1(\ovl L)^rc_1(\ovl M_{\{i_1\}})c_1(\ovl M_{\{i_2\}})\cdots c_1(\ovl M_{\{i_{\dim\XXX-r}\}})$$
in the expansion of $c_1(\ovl L)^{r}c_1(\ovl M)^{\dim\XXX-r}$ is a non-negative measure on $\XXX(\C)$. So the strict positivity of total measure is equivalent to that of some term.
\end{remark}

\begin{proof}
By Theorem \ref{measure}, if
$$\int_{\XXX(\C)}c_1(\ovl L)^{r}c_1(\ovl M_I)^{\dim\XXX-r}>0,$$
then $\XXX$ is $I$-non-degenerate.

Conversely, assume that $\XXX$ is $I$-non-degenerate. Let $x$ be an $I$-non-degenerate point. There exists a simply connected open neighborhood $U$ of $x$ contained in the $I$-non-degeneracy locus. By fixing a lifting, we get the Betti map $(b_I,b_{\ab}):U\to(\C^\times)^{\lvert I\rvert}\times (\R^{2g}/\Z^{2g})$. Similarly as in \eqref{eq_cd_real_manifolds}, the image $(b_I,b_\ab)(U)$ contains a real  manifold $M$ around $(b_I,b_\ab)(x)$, with a submersion $M\rightarrow N\subset\R^{2g}/\Z^{2g}$ of real manifolds. We may assume that $M$ is open in $ (\C^\times)^{\lvert I\rvert}\times N$.

We first consider the case that $x\in \GGG_{I,\mathbb S^1}$. Since the map $$\lambda:(\C^\times)^{\lvert I\rvert}\rightarrow (\R_{>0})^{\lvert I\rvert}, (z_1,\dots, z_{{\lvert I\rvert}})\mapsto (\lvert z_1\rvert,\dots, \lvert z_{{\lvert I\rvert}}\rvert)$$ is open, $\lambda(M)$ contains an open real submanifold of $\R_{>0}^{\lvert I\rvert}$ around $\lambda((b_I,b_\ab)(x))=1$. Therefore $(\lambda\circ(b_I,b_\ab))^{-1}(1)\subset U\cap \mc G_{I,\mathbb S^1}$ contains an real manifold near $x$, and $c_1(\ovl L)^r\omega_I^{\lvert I\rvert}$ is a non-vanishing smooth top form on this manifold. So the integral is strictly positive.

In general, we may take a point $s\in\mathbb{G}_m^{\lvert I\rvert}(\C)$ such that $b_I(U)+s$ intersects $(\mathbb S^1)^{\lvert I\rvert}$. Denote by $t_s:\GGG\to\GGG$ the translation by $s$. Then $t_s(\XXX)$ has an $I$-non-degenerate point contained in $\GGG_{I,\mathbb S^1}$. Thus,
$$(\widetilde L|_{t_s(\XXX)}^{r}\cdot\widetilde M_I|_{t_s(\XXX)}^{\dim\XXX-r})=\int_{t_s(\XXX)(\C)}c_1(\ovl L)^{r}c_1(\ovl M_I)^{\dim\XXX-r}>0.$$
First note that $c_1(\ovl L)$ is translation invariant. The Green function of $t_s^*\ovl M_I-\ovl M_I$ is
$$\sum_{i\in I}(|-\log|s_ib_{\tor}(\cdot)||-|-\log|b_{\tor}(\cdot)||),$$
where $s_i$ is the $i$-th component of $s$. The triangular inequality
$$\Big\lvert\big\lvert-\log\lvert s_iz\rvert\big\rvert -\big\lvert-\log\lvert z\rvert\big\rvert\Big\rvert\le\big\lvert-\log\lvert s_i\rvert\big\rvert.$$
implies this Green function is bounded. By \cite[Theorem 3.6.4]{YuanZhang}, $t_s^*\widetilde M_I=\widetilde M_I$. Therefore,
$$\int_{\XXX(\C)}c_1(\ovl L)^{r}c_1(\ovl M_I)^{\dim\XXX-r}=(\widetilde L|_{\XXX}^{r}\cdot\widetilde M_I|_{\XXX}^{\dim\XXX-r})=(\widetilde L|_{t_s(\XXX)}^{r}\cdot\widetilde M_I|_{t_s(\XXX)}^{\dim\XXX-r})>0.$$
\end{proof}

\subsection{Non-degeneracy and bigness}
In this subsection, we show that non-degeneracy has direct application to the height inequality via arithmetic bigness. This will be used in \S\ref{subsec_height_ineq}.
\begin{proposition}\label{big}
Assume that $\XXX\subset\GGG$ is $I$-non-degenerate. Let $\ovl H$ be a nef adelic line bundle on $S$. Then there is a positive integers $n$ and a positive constant $e>0$ such that for any positive integer $m$ large enough,
$$(m^2n\ovl L+m\ovl M_I-\pi_{\mathcal G/S}^*\ovl H+\ovl{\mathcal O}(e))|_\XXX$$
is big.
\end{proposition}

\begin{proof}
By Proposition \ref{integrable} there is a positive integer $n$ and a nef adelic line bundle $\ovl H'$ on $S$ such that the adelic line bundle
$$\ovl Q_1=n\ovl L+\ovl M_I+\pi_{\mathcal G/S}^*\ovl H'$$
on $\GGG$ is nef. Thus so is $\ovl Q_m=[m]^*\ovl Q_1=m^2n\ovl L+m\ovl M_I+\ovl H'$

Since $\ovl Q_m+\ovl{\mathcal{O}}(e)$ and $\ovl H+\ovl H'$ are both nef, after applying the arithmetic Siu's inequality \eqref{eq_Siu_inequality},
\begin{align*}
&\kern-1em \vol((m^2n\ovl L+m\ovl M_I-\ovl H+\ovl{\mathcal{O}}(e))|_\XXX)\\
=&\ \vol((\ovl Q_m+\ovl{\mathcal{O}}(e)-\ovl H-\ovl H')|_\XXX)\\
\ge&\ ((\ovl Q_m+\ovl{\mathcal{O}}(e))|_\XXX)^{\dim\XXX+1}-(\dim\XXX+1)((\ovl Q_m+\ovl{\mathcal O}(e))|_\XXX^{\dim\XXX}\cdot(\ovl H+\ovl H')|_\XXX)\\
=&\ (\ovl Q_m|_\XXX)^{\dim\XXX+1}+e(\dim\XXX+1)(\widetilde Q_m|_{\XXX}^{\dim\XXX})\\
&\ -(\dim\XXX+1)\Big(\big(\ovl Q_m|_\XXX^{\dim\XXX}\cdot(\ovl H+\ovl H')|_\XXX\big)+e\dim\XXX\big(\widetilde Q_m|_{\XXX}^{\dim\XXX-1}\cdot(\widetilde H+\widetilde H')|_\XXX\big)\Big)\\
=&\ e(\dim\XXX+1)\Big(\big(\widetilde Q_m|_{\XXX}\big)^{\dim\XXX}-\dim\XXX\big(\widetilde Q_m|_{\XXX}^{\dim\XXX-1}\cdot(\widetilde H+\widetilde H')|_\XXX\big)\Big)\\
&\ +(\ovl Q_m|_\XXX)^{\dim\XXX+1}-(\dim\XXX+1)(\ovl Q_m|_\XXX^{\dim\XXX}\cdot(\ovl H+\ovl H')|_\XXX)
\end{align*}
This lower bound is a linear function in $e$. Once we have
$$(\widetilde Q_m|_{\XXX})^{\dim\XXX}-\dim\XXX(\widetilde Q_m|_{\XXX}^{\dim\XXX-1}\cdot(\widetilde H+\widetilde H')|_\XXX)>0,$$
then $(m^2n\ovl L+m\ovl M_I-\ovl H+\ovl{\mathcal{O}}(e))|_\XXX$ is big for $e$ large enough. It remains to prove the above inequality holds for some large $m$.

We analyze the asymptotic behavior as $m\to\infty$. Since $2r$ is the Betti rank of $\pi(\XXX)$, we have
$$c_1(\ovl L|_\XXX)^{r+1}=0.$$
So
\begin{align*}
&\ (\widetilde Q_m|_{\XXX}^{\dim\XXX})-\dim\XXX(\widetilde Q_m|_{\XXX}^{\dim\XXX-1}\cdot(\widetilde H+\widetilde H')|_\XXX)\\
=&\ \int_{\XXX(\C)}c_1(\ovl Q_m)^{\dim\XXX}-\dim\XXX\int_{\XXX(\C)}c_1(\ovl Q_m)^{\dim\XXX-1}c_1(\ovl H+\ovl H')\\
=&\ m^{\dim\XXX+r}\binom{\dim\XXX}{r}\int_{\XXX(\C)}c_1(\ovl L)^r\cdot c_1(\ovl M_I)^{\dim\XXX-r}+O(m^{\dim\XXX+r-1}).  
\end{align*}
We conclude the proof.
\end{proof}

\begin{corollary}\label{weak_bogomolov}
We keep the same hypothesis as in Proposition \ref{big}. Then
$$\{x\in \XXX(\overline{\Q}):\widehat{h}_I(x)\le\epsilon h_{\ovl H}(\pi_{\mc G/S}(x))-\delta\}$$
is not Zariski dense in $\XXX$ for some positive constants $\epsilon$, $\delta>0$.
\end{corollary}

\begin{proof}
By the bigness, some positive multiple of $(m^2n\ovl L+m\ovl M_I-\pi_{\mathcal G/S}^*\ovl H+\ovl{\mathcal O}(e))|_\XXX$ admits an effective section $s$. Then
$$h_{m^2n\ovl L+m\ovl M_I-\pi_{\mathcal G/S}^*\ovl H'+\ovl{\mathcal O}(e)}(\cdot)\ge0$$
on $(\XXX\setminus|\mr{div}(s)|)(\ovl\Q)$. Since
\begin{align*}
    h_{m^2n\ovl L+m\ovl M_I-\pi_{\mathcal G/S}^*\ovl H+\ovl{\mathcal O}(e)}(x)&=h_{m^2n\ovl L+m\ovl M_I}(x)-h_{\ovl H}(\pi_{\mathcal G/S}(x))+e\\&\le m^2n\widehat{h}_I(x)-h_{\ovl H}(\pi_{\mathcal G/S}(x))+e,
\end{align*}
the statement holds for $\epsilon=m^{-2}n^{-1},\delta=m^{-2}n^{-1}e$.
\end{proof}

\subsection{Equidistribution}
The arithmetic equidistribution theorem which goes back to \cite{SUZ}. It is then generalized by Yuan \cite{Yuan2008Bigness} and Yuan--Zhang \cite{YuanZhang}. Equidistribution on semiabelian varieties is proved by K\"uhne \cite{Kuhne}. Balla\"y and Sombra \cite{BS} gives a further generalization. We give a semiabelian scheme version here by adapting \cite[Theorem 1.2]{BS}.

Recall a sequence $x_i\in X(\ovl\Q)\ (i=1,2\dots)$ in a variety $X$ over $\Q$ is called \emph{generic} if $x_i$ converges to the generic point of $X$, or equivalently, any closed subscheme $Z\subsetneq X$ contains at most finitely many points in the sequence. Let $\mu$ be a measure on $X(\C)$, we say the sequence \emph{equidistributes to $\mu$} if
$$\frac{1}{\#(\mr{Gal}(\ovl\Q/\Q)\cdot x_i)}\sum_{x\in\mr{Gal}(\ovl\Q/\Q)\cdot x_i}\delta_{x}$$
coverges weakly to $\mu$, i.e., for any continuous function $f$ with compact support on $X(\C)$,
$$\lim_{i\to\infty}\frac{1}{\#(\mr{Gal}(\ovl\Q/\Q)\cdot x_i)}\sum_{x\in\mr{Gal}(\ovl\Q/\Q)\cdot x_i}f(x)=\int_{X(\C)}f\mu.$$

\begin{theorem}\label{theorem_equidistribution}
Let $\XXX\subset\GGG$ be a closed subvariety. Let $x_i\in\XXX\ (n=1,2\dots)$ be a generic sequence of closed points. If $\XXX$ is $I$-non-degenerate and
$$\lim_{i\to\infty}\widehat{h}_I(x_i)=0,$$
then the sequence $(x_i)$ equidistributes to
$$\frac{c_1(\ovl L|_{\XXX})^rc_1(\ovl M_I|_\XXX)^{\dim\XXX-r}}{\displaystyle\int_{\XXX(\C)}c_1(\ovl L)^rc_1(\ovl M_I)^{\dim\XXX-r}}.$$
\end{theorem}

\begin{proof}
Take $n$ and $e$ as in Proposition \ref{big}. Note that $nm^2\ovl L+m\ovl M_I+
\ovl H$ is nef for all $m$. Denote $\ovl N=(n\ovl L+\ovl M_I)|_\XXX$. Take the $\Q$-adelic line bundle
$$\ovl Q_m=\Big(\frac{n}{2}\ovl L+\frac{1}{2m}\ovl M_I+\frac{1}{2m^2}\ovl H-\frac{1}{2m^2}\ovl{\mathcal{O}}(e)\Big)\Big|_\XXX$$
on $\XXX$. We show that $\ovl Q_m$ is a \emph{semipositive approximation} of $\ovl N$ for $m$ large enough in the sense of \cite[Definition 8.10]{BS} and the sequence satisfies the conditions of \cite[Theorem 8.11]{BS}. Once this is done, the sequence $(x_i)$ equidistributes to the weak limit
$$\lim_{m\to\infty}\frac{c_1(\ovl Q_m)^{\dim\XXX}}{\displaystyle\int_{\XXX(\C)}c_1(\ovl Q_m)^{\dim\XXX}}.$$
Since
\begin{align*}
c_1(\ovl Q_m|_\XXX)^{\dim\XXX}&=m^{-\dim\XXX+r}\binom{\dim\XXX}{r}\frac{n^r}{2^{\dim\XXX}}c_1(\ovl L|_\XXX)^rc_1(\ovl M_I|_\XXX)^{\dim\XXX-r}\\
&+O(m^{-\dim\XXX+r-1}),
\end{align*}
we get the equilibrum measure
$$\lim_{m\to\infty}\frac{c_1(\ovl Q_m)^{\dim\XXX}}{\displaystyle\int_{\XXX(\C)}c_1(\ovl Q_m)^{\dim\XXX}}=\frac{c_1(\ovl L|_\XXX)^rc_1(\ovl M_I|_\XXX)^{\dim\XXX-r}}{\displaystyle\int_{\XXX(\C)}c_1(\ovl L)^rc_1(\ovl M_I)^{\dim\XXX-r}}.$$

We first check that $(\ovl Q_m)_{m\ge1}$ is a semipositive approximation of $\ovl N$. Note that the birational map is taken to be the identity map of $\XXX$. So we only need:
\begin{enumerate}
    \item[(i)]$\ovl Q_m$ is semipositive;
    \item[(ii)]$\widetilde Q_m$ is big;
    \item[(iii)]$\ovl N-\ovl Q_m$ is pseudo-effective. 
\end{enumerate}
For (i), since $\ovl Q_m+\frac1{2m^2}\ovl{\mathcal{O}}(e)$ is nef, $\ovl Q_m$ is semipositive. By Proposition \ref{big}, both
\begin{equation}\label{eq_decompose_Qm}
    \ovl Q_m+\frac1{m^2}\ovl{\mathcal{O}}(e)=\frac{1}{2m^2}(m^2n\ovl L+m\ovl M_I-\pi_{\mathcal G/S}^*\ovl H+\ovl{\mathcal O}(e))|_\XXX+\frac{1}{m^2}\pi_{\mathcal G/S}^*\ovl H|_\XXX
\end{equation}
and
$$\ovl N-\ovl Q_m=\frac{1}{2m^2}(m^2n\ovl L+m\ovl M_I-\pi_{\mathcal G/S}^*\ovl H+\ovl{\mathcal O}(e))|_\XXX$$
are big. This implies (ii) and (iii) for $m$ large enough.

Take
$$r_n=\sup\{\lambda\in\R:\widetilde Q_m-\lambda\widetilde N\mr{\ is\ big}\}.$$
It remains to verify the following conditions of \cite[Theorem 8.11]{BS}:
\begin{enumerate}
    \item[(1)]$$\lim_{n\to\infty}\frac{1}{r_n}\left(\mr{ess}(\ovl N)-\frac{\ovl Q_m^{\dim\XXX+1}}{(\dim\XXX+1)\widetilde Q_m^{\dim\XXX}}\right)=0;$$
    \item[(2)]$$\sup_n\frac{\mr{ess}(\ovl N)-\mr{abs}(\ovl Q_m)}{r_n}<\infty.$$
\end{enumerate}
Since
$$\widetilde Q_m-\frac1{4m}\widetilde N=\frac{2m-1}{4m}\widetilde L+\frac{1}{4m}\widetilde M_I+\frac{1}{2m^2}\widetilde H'$$
is big, $r_n\ge\frac1{4m}$. By the assumption that $\widehat{h}_I(x_i)\to0$, $\mr{ess}(\widetilde N)=0$. Since $(x_i)$ is generic, by Corollary \ref{weak_bogomolov}, for $i$ large enough, $h(\XXX_{\pi_{\mathcal G/S}(x_i)})\le\epsilon_1^{-1}\widehat h(x_i)+\epsilon_1^{-1}\delta$. Therefore,
\begin{align*}
    -\frac e{2m^2}\le \mr{abs}(\ovl Q_m)&\le\frac{\ovl Q_m^{\dim\XXX+1}}{(\dim\XXX+1)\widetilde Q_m^{\dim\XXX}}\le \mr{ess}(\ovl Q_m)\\&\le\liminf_{i\to\infty}h_{\ovl Q_m}(x_i)\le\liminf_{i\to\infty}h_{\frac1{2m^2}\ovl H}(\pi_{\mc G/S}(x_i))\le\frac{\epsilon^{-1}\delta}{2m^2},
\end{align*}
where the second and the third inequality are due to \eqref{eq_fundamental} and \eqref{eq_fundamental2} respectively, and the fifth inequality is due \eqref{eq_decompose_Qm} and $\lim \widehat h_{I}(x)=0$.
Therefore (1) and (2) follow.
\end{proof}

\section{Uniform Bogmolov type results}\label{section_UBC}
Let $G$ be a semiabelian variety over an algebraically closed field $K$ of characteristic $0$. By a \emph{coset} in $G$, we mean a translation of a group subvariety. Recall that the \emph{Ueno locus} of a closed subvariety $X\subset G$ is the union of all positive dimensional cosets contained in $X$, which is a Zariski closed subset(cf. \cite{Noguchi} or \cite{abramovich}). Denote by $X^\circ$ its complement in $X$. 

If $G$ is principally polarized, let $L$ be as in \eqref{eq_symmstrification_pp}, and $M$ be as in \eqref{eq_boundary_line_bundle}. We define the \emph{degree} of $X$ as 
$$\deg X:=\deg_{L+M}(\ovl X)$$
where $\ovl X$ is the Zariski closure of $X$ in the compactification $\ovl G$ constructed in \S\ref{subsec_toric_compactification}.

The canonical height function in this case is defined in \eqref{eq_canonical_height} by taking $S=\mr{Spec} K$. 
Recall that we say $X\subset G$ generates $A$ if $A$ itself is the smallest abelian subvariety containing $\pi_{G/A}(X)-\pi_{G/A}(X)$. We first prove the following uniform Bogomolov type result, the version with extra term will be given in \S\ref{subsec_proof_subc}.
\begin{theorem}
\label{bogomolov}
There exist positive constants $c_1(g+t,d)$, $c_2(g+t,d)>0$ satisfying the property as follows. Let $G$ be a principally polarized semiabelian variety over $\ovl\Q$ of dimension $g+t$ with a level-$4$ structure. Let $X\subset G$ be a closed subvariety of degree $d$, such that $X$ generates the abelian quotient $A$. Then there is a reduced closed subscheme $Z\subset X$ of degree at most $c_1(g+t,d)$ such that
$$\#\{x\in X^\circ(\ovl\Q):\widehat{h}(x)\le c_2(g+t,d)h([X])\}\subset Z(\ovl\Q)$$
where the definition of $h([X])$ will be given in \eqref{eq_modular_height}.
\end{theorem}

Note that in this section, when we write that a semiabelian variety $G$ is of dimension ``$g+t$", this implicitly means that we will proceed in the case that $G$ has the toric rank $t$, and a $g$-dimensional abelian quotient. After running over all such $g$ and $t$ with fixed $g+t$, and taking the maximum, we obtain our constants depending only on $g+t$ (and the degree). We will not repeat this argument.

In sequel, by abuse of notation, $L$ and $M$ will be also used to denote the line bundles on a family of semiabelian varieties. In other words, the case here can be understood as the restriction of $L$ and $M$ on some special fiber.
\subsection{Hilbert scheme and modular height}\label{subsec_hilbert_modular}
We introduce a variety parametrizing subvarieties of semiabelian varieties and define the modular height of them. The reader may skip this part and only admit Proposition \ref{hilbert}.

For simplicity of notations, let $S_0=(\mc A_g^\vee)^{[t]}$, $\mc P=\mc P_{g,t}^\times$ and $\ovl{\mc P}$ be the compactification of $\mc P$ as in \S\ref{subsec_toric_compactification}. The \emph{Hilbert scheme} $\hilb(\ovl{\mc P}/S_0)$ is a scheme, projective and locally of finite type over $\Q$ which parametrizes the flat families of closed subschemes of $\ovl{\mc P}$ over $S$. More precisely, there is a closed subscheme $\YYY\subset\ovl{\mc P}\times_{S_0}\hilb(\ovl{\mc P}/S_0)$ flat over $\hilb(\ovl{\mc P}/S_0)$ satisfying that for any $S$-scheme $T$ and closed subscheme $\mc Y'\subset\ovl{\mc P}\times_{S_0}T$ flat over $T$, there exists a unique $S$-morphism $T\to\hilb(\ovl{\mc P}/S_0)$ such that $\mc Y'=\YYY\times_{\hilb(\ovl{\mc P}/S_0)}T$.

\[\begin{tikzcd}
	\mc Y' & {\YYY} \\
	{\ovl{\mc P}\times_{S_0}T} & {\ovl{\mc P}\times_{S_0}\hilb(\ovl{\mc P}/S_0)} \\
	T & {\hilb(\ovl{\mc P}/S_0)}
	\arrow[from=1-1, to=1-2]
	\arrow[hook', from=1-1, to=2-1]
	\arrow[hook', from=1-2, to=2-2]
	\arrow[from=2-1, to=2-2]
	\arrow[from=2-1, to=3-1]
	\arrow[from=2-2, to=3-2]
	\arrow[from=3-1, to=3-2]
\end{tikzcd}\]

As in \eqref{eq_symmstrification_pp}, let $\mc L_{\mr{pp}}$ be the principal polarization on $\mc A_g$ and $L:=\mc L_{\mr{pp}}+[-1]_{\mc A_g}^*\mc L_{\mr{pp}}$ on $\mc A_{g}$, which is a symmetric polarization. Let $M$ be the boundary line bundle $\O_{\ovl{\mc P}}(\ovl{\mc P}\setminus \mc P).$ 
Recall that our degrees of subvarieties are defined with respect to the restriction of $M+L$ on some special fiber over $S_0.$

Following \cite[\S 3]{GGK}, the \emph{restricted Hilbert scheme} $\hilb_d^\circ(\ovl{\mc P}/S_0)$ of degree $d$ is the reduced subscheme of $\hilb(\ovl{\mc P}/S_0)$ with the open underlying set
$$\{s\in \hilb(\ovl{\mc P}/S_0):\YYY_s\subset\ \ovl{\mc P}_s\mr{\ is\ geometrically\ integral\ and\ of\ degree\ }d\}.$$
It is of finite type over $\Q$.

Let $\mathfrak X=\YYY\cap(\mc P\times_{S_0}\hilb_d^\circ(\ovl{\mc P}/S_0))$, which is still flat over $\hilb_d^\circ(\ovl{\mc P}/S_0))$. Therefore the image of the projection under $\hilb_d^\circ(\ovl{\mc P}/S_0))$ is open, which is denoted as $\hilb_{g,t.d}^\circ.$ It satisfies the following universal property: 
\begin{proposition}\label{hilbert}
    For any algebraically closed field $K$ of characteristic $0$, a principally polarized semiabelian variety $G$ over $K$, and a closed subvariety $X\subset G$ of degree $d$, 
    after fixing a level-$4$ structure on $G$,
    there exists a unique $K$-point $i_X\in\hilb_{g,t,d}^\circ$ such that $X=i_X^*\mathfrak X$.
    We denote the point $i_X$ as $[X].$
\end{proposition}

Now let $S$ be an irreducible component of $\hilb_{g,t,d}^\circ$, which is quasi-projective over $\Q$. According to the construction of Hilbert scheme in \cite[Th\'{e}or\`{e}me 2.1(b) and Lemme 2.4]{Gro} and \cite[\S 3]{GGK}, $S$ parametrizes all geometrically integral subscheme of degree $d$ and a fixed Hilbert polynomial. By definition, $\mathfrak X_S$ has geometrically integral closed fibers, and is flat over $S$, hence it is a closed subvariety of $\mc P_S$. 

We fix a projective model $\ovl S$ of $S$, and an ample hermitian line bundle $\ovl H$ on $\ovl S$. We refer the reader to \cite{Zhang95Pos} for the definition of ampleness.
For any $X$ as in Proposition \ref{hilbert}, define its \emph{modular height} as\begin{equation}\label{eq_modular_height}
    h([X]):=\max \{h_{\ovl H}(i_X)\}
\end{equation}
where the maximum run over all irreducible component $S$ that containing $i_X.$ 

Classically, we have another notion $h(X)$ for the height of subvarieties, defined as the \emph{projective height} of $\ovl X\hookrightarrow \ovl G\hookrightarrow \mathbb P^N$ induced by the global sections of $M+4(t+1)L$  as in \cite[pp. 195]{remond}, where $L$ and $M$ are considered as their restrictions on some special fiber, and $N=3^t(H^0(A,4L)+1)^{t+1}-1=3^t(\deg_{4L}(A)/g!+1)^{t+1}-1$.
More precisely, if we set $\ovl{O(1)}$ to be the adelic line bundle on $\mathbb P^N_{\Q}$ induced by the tautological line bundle on $\mathbb P^N_{\Z}$ equipped with the Fubini-Study metric, then $$h(X):=(\ovl{O(1)}|_{\ovl X})^{\dim X+1}.$$

Now we consider the irreducible component $S\subset \hilb_{g,t,d}^\circ$ such that $[X]\in S$, we have the $S$-embedding $\ovl{\mc G}=\ovl {\mc P}_S\hookrightarrow \mathbb P^N_{\Q}\times S.$ By abuse of notation, we also denote by $\ovl{O(1)}$ its pull-back to $\mathbb P^N_{\Q}\times S$. Let $\ovl H'$ be the Deligne pairing $\underbrace{\langle \ovl{O(1)}|_{\mc X},\dots,\ovl O(1)|_{\mc X}\rangle}_{\dim X+1\text{ times}}$ defined as in \cite[\S4]{YuanZhang}, which is an adelic line bundle on $S$. Then $h(X)=h_{\ovl H'}([X])$ by \cite[Theorem 4.1.3]{YuanZhang}. Since $\ovl H$ is big and nef on every closed subvariety of $S$, by Lemma \ref{lemma_height_ineq_nef_and_big}, we have $$h([X])\geq h_{\ovl H}([X])\geq \ell h_{\ovl H'}([X])=\ell h(X)$$
for a constant $\ell>0.$

\subsection{Relative Bogomolov conjecture for self-product}
We first reduce Theorem \ref{bogomolov} to the following relative Bogomolov type result. Note that the height function in the sequel is the family version \eqref{eq_canonical_height}, which agrees with the height function used in Theorem \ref{bogomolov} on some special fiber.

\begin{theorem}
\label{relative}
Let $S$ be a subvariety of $\hilb_{g,t,d}^\circ,$  $\mc G:=S\times_{(\mc A_g^\vee)^{[t]}} \mc P_{g,t}^\times$, and $\mc X:=\mathfrak X_S.$ Let $\ovl\eta$ be the geometric generic point of $S$. Assume that 
for every $s\in S$, $\XXX_s$ generates the abelian quotient $\mc A_s$.
Then either $\XXX_s^\circ$ is empty for every $s\in S(\ovl\Q)$, or there is a positive integer $n$ and a positive number $\epsilon>0$ such that
$$\{x\in \mc X^{[n]}(\ovl\Q):\widehat{h}(x)\le\epsilon h([\mc X_s]), x\in \mc X^n_s\}$$
is not Zariski dense in $\mc X^{[n]}$.
\end{theorem}

\begin{proof}[Proof of Theorem \ref{relative} $\Rightarrow$ Theorem \ref{bogomolov}]
To get Theorem \ref{bogomolov}, consider the subset 
$$\hilb^{\circ, \mr{gen}}_{g,t,d}:=\{[X]\in \hilb^\circ_{g,t,d}\,:\, X\text{ generates }A\}$$
where $A$ is the abelian quotient of the ambient semiabelian variety $G$ of $X$. This is a closed subset of $\hilb^\circ_{g,t,d}$ due to \cite[Lemma 3.5]{GGK}. Note that this closedness may fail if we replace ``generating $A$" with ``generating $G$".

After letting $S$ run over irreducible components of $\hilb^{\circ, \mr{gen}}_{g,t,d}$ equipped with reduced scheme structures, it suffices to show that for any $s\in S(\ovl\Q)$, there is a reduced subscheme $Z_s\subset\XXX_s$ of degree at most $c_1(\XXX/S)$ such that
$$\#\{x\in \XXX_s^\circ(\ovl\Q)\,:\,\widehat{h}(x)\le c_2(\XXX/S)h([\XXX_s])\}\subset Z_s(\ovl\Q)$$
for some positive constants $c_1(\XXX/S)$, $c_2(\XXX/S)>0$. 

If every $\XXX_s^\circ$ is empty, there is nothing to prove. Assume 
$$\{x\in \mc X^{[n]}(\ovl {\Q})\,:\,\widehat{h}(x)\le\epsilon h([\XXX_s], x\in \mc X^n_s\}\subset\ZZZ$$
for some $\epsilon>0$, and a reduced closed subscheme $\ZZZ\subsetneq\mc X^{[n]}$. We prove the statement by induction on $\dim S$.

Consider the Zariski closure $\ovl{\mc Z}$ (resp. $\ovl{\mc X^{[n]}}=\ovl{\mc X}^{[n]} $) of $\mc Z$ (resp. $\mc X^{[n]}$) in $\ovl{\mc G^{[n]}}=\ovl{\mc G}^{[n]}$. 
Since $\ovl {\mc X^{[n]}}$ is irreducible, we may enlarge $\mc Z$ so that $\ovl{\mc Z}\subsetneq \ovl {\mc X^{[n]}}$ is equidimensional.

There is an open subvariety $U\subset S$ such that $\ovl \ZZZ_U$ is flat and equidimensional over $U$. In particular, for $s\in U$, $\ovl Z_s\subsetneq \ovl{\mc X_s^{[n]}}$ (the ``$\subsetneq$" is due to dimension reason) are of constant degree. Then \cite[Lemma 4.3]{GGK} implies the statement over $U$\footnote{For a sufficiently large $\ell$, $\ell(L+M)$ is relatively very ample over $S$. Here we substitute $L$ and $\Sigma$ in the original statement of \cite[Lemma 4.3]{GGK}, with $\ell(L+M)$ and $\{x\in \XXX_s^\circ(\ovl\Q)\,:\,\widehat{h}(x)\le c_2(\XXX/S)h([\XXX_s])\}$ respectively. For a closed subvariety $X'\subsetneq X$, $\deg_{L+M}(X')\leq \deg_{\ell(L+M)}(X')$, hence we obtained the statement as desired.}. Decompose the complement $S\setminus U$ into closed subvarieties $S_1,\dots,S_m$. 
The induction hypothesis implies the statement over them. We thus conclude the proof by setting
$$c_1(\XXX/S):=\max\{c_1(\XXX_U/U),c_1(\XXX_{S_1}/S_1),\dots,c_1(\XXX_{S_m}/S_m)\}$$
and
$$c_2(\XXX/S):=\min\{c_2(\XXX_U/U),c_2(\XXX_{S_1}/S_1),\dots,c_2(\XXX_{S_m}/S_m)\}.$$
\end{proof}

\subsection{Applying Bigness}\label{subsec_height_ineq}

It remains to prove Theorem \ref{relative}. We do some reductions. 
In the case that $\mc X_s^\circ$ is not empty for some $s\in S(\ovl {\Q})$. we divide the conclusion into two parts: 
\begin{enumerate}
    \item[(BC1)] for every sufficiently large $n$, there are positive numbers $\epsilon_1>0$ and $\delta>0$ such that
    $$E_1=\{x\in \mc X^{[n]}(\Q)\,:\,\widehat{h}(x)\le\epsilon_1 h([\mc X_s])-\delta, x\in \mc X_s^n\}$$
    is not Zariski dense;
    \item[(BC2)]there is a positive integer $n$ and a positive number $\epsilon_2>0$ such that
    $$E_2=\{x\in \mc X^{[n]}(\Q):\widehat{h}(x)\le\epsilon_2\}$$
    is not Zariski dense. 
\end{enumerate}
To see how these two parts implies the original conclusion, take $n$ as in (2) and $\displaystyle\epsilon=\frac{\epsilon_1\epsilon_2}{\epsilon_2+\delta}$. We show that the set
$$E=\{x\in \mc X^{[n]}(\Q):\widehat{h}(x)\le\epsilon h([\mc X_s]), x\in \mc X_s^n\}$$
is contained in $E_1\cup E_2$, which is not Zariski dense in the subvariety $\XXX$. For any point $x\notin E_1\cup E_2$, we have
\begin{align*}
    (\epsilon_2+\delta)\widehat h(x)&=\epsilon_2\widehat h(x)+\delta\widehat h(x)\\&>\epsilon_2(\epsilon_1h([\mc X_s])-\delta)+\delta\epsilon_2\\&=\epsilon_1\epsilon_2h([\mc X_s]),
\end{align*}
where $x\in \mc X_s^n$. Hence $x\notin E$. In summary, we reduce Theorem \ref{relative} to proving the previous two parts under the non-degeneracy condition. We finish the first in this subsection, leaving the second to the next subsection. 

First, we may assume any geometric generic fiber of $\XXX/S$ has finite stabilizer, otherwise every $\XXX_s^\circ$ is empty. By Theorem \ref{theo_nondeg_gen}, $\mc X^{[n]}$ is non-degenerate due to the following lemma, whose proof is the same as \cite[Lemma 3.3]{GGK}.

\begin{lemma}
Let $S$ be a (irreducible) subvariety of $\hilb_{g,t,d}^\circ$, and $\mc X:=\mathfrak X_S$. Then for $m$ large enough, $\mc X^{[m]}\to\iota^{[m]}(\mc X^{[m]})$ is generically finite.    
\end{lemma}


Therefore $\mc X^{[n]}$ is $I$-non-degenerate for some $I\subset \{1,\dots, nt\}$ due to Remark \ref{remark_I_non_deg_equiv_non_deg}. Applying Corollary \ref{weak_bogomolov} to $\mc X^{[n]}$, $\mc G^{[n]}$ and the adelic line bundle $\ovl H$ defining the modular height of subvarieties, we obtain that there exist $\epsilon_1>0$ and $\delta>0$ such that 
$$\{x\in \XXX^{[n]}(\Q):\widehat{h}_I(x)\le\epsilon h([\mc X_s])-\delta, x\in \mc X_s^n\}$$
is not Zariski dense in $\XXX^{[n]}.$
Since $\widehat h_I(x)\le\widehat h(x)$, we obtain (BC1).

\subsection{Applying equidistribution}
We finish the proof of Theorem \ref{relative}. We first need a lemma as follows.

\begin{lemma}\label{lemma_non_deg_indices}
Let $(\mc G,\mc X)$ and $(\mc G',\mc X')$ be two two pairs over $S$ as in the begining of this section. Assume $x\in\mc X^{\mathrm{sm}}(\mathbb C)$ is an $I$-non-degenerate point for some index set of toric part of $\mc G$. Let $I'$ be the index set for the whole toric part of $\mc G'$. Then any smooth point $x'\in(x\times_S\mc X')^\mathrm{sm}(\mathbb C)$ on the fiber over $x$ is an $(I\sqcup I')$-non-degenerate point.

In particular, if $\mc X$ is $I$-non-degenerate, then $\mc X\times_S\mc X'$ is $(I\sqcup I')$-non-degenerate.
\end{lemma}

\begin{proof}
Denote by $Y=x\times_S\mc X'$ the fiber over $x$. There is a short exact sequence of tangent spaces:
$$0\longrightarrow T_{x'}Y\longrightarrow T_{x'}(\mc X\times_S\mc X')\longrightarrow T_{x}\mc X\longrightarrow0.$$
We need to show the $(I\sqcup I')$-th Betti map in injective on $T_{x'}(\mc X\times_S\mc X')$. By the hypothesis on $\mc X$, the $I$-th Betti map is injective on $T_x\mc X$. Thus, if some tangent vector is killed by the differential of the $(I\sqcup I')$-th Betti map, it must lie in $T_{x'}Y$. However, $Y$ is contained in a single semiabelian variety. So the $I'$-th Betti map is injective on $T_{x'}Y$. 
\end{proof}

\begin{proof}[Proof of Theorem \ref{relative}]
We may and do assume that $\mc X$ is non-degenerate after replace $\mc X$ by $\mc X^{[m]}$ for some $m\gg 0.$

We denote by $\ovl L^{\boxtimes n}$ (resp. $\ovl M^{\boxtimes n}$) the fiberwise box product $\displaystyle\sum_{i=1}^{n}p_i^*\ovl L$ (resp. $\displaystyle\sum_{i=1}^{n}p_i^*\ovl M$) where $p_i:\mc G^{[n]}\rightarrow \mc G$ is the $i$-th projection.

Now we only need to prove (BC2), which is equivalent to $$\mr{ess}((\ovl L+\ovl M)^{\boxtimes n}|_{\mc X^{[n]}})>0$$ for some positive integer $n$.

Assume the contrary. Consider the fiberwise \emph{Faltings--Zhang map}
\begin{align*}
    \alpha_n:\mc G^{[n]}&\longrightarrow \mc G^{[n-1]},\\
    (x_1,\dots, x_n)&\longmapsto (x_1-x_2,\dots, x_{n-1}-x_n).
\end{align*}
By the proof of \cite[Lemma 3.1]{Zhang98BC}, we can take an $n$ large enough such that $\mc X^{[n]}\rightarrow \alpha_n(\mc X^{[n]})$ is generically finite.

Take a generic sequence $(x_i)$ in $\mc X^{[n+1]}$ such that $\widehat h(x_i)\to 0\ (n\to\infty)$. Hence $\widehat h_I(x_i)\to 0\ (n\to\infty)$ for any set $I$ of indices.

Let $x\in\XXX^\mr{sm}(\C)\cap\GGG_{I,\mathbb S^1}$ be an $I$-non-degenerate point where $\XXX(\C)$ is smooth over $S(\C)$. Then by Lemma \ref{lemma_non_deg_indices}, $(x,x,\dots,x)\in(\XXX^{[n+1]})^\mr{sm}(\C)\cap(\mc G^{[n+1]})_{I_1,{\mathbb S^1}}$ is $I_1$-non-degenerate for some $I_1$. Moreover, $\mc X\times_S\alpha_n(\mc X^{[n]})$ is $I_2$-non-degenerate for some $I_2$.

We denote that $\alpha=\mr{id}_{\mc G}\times_S \alpha_n.$
Applying Theorem \ref{theorem_equidistribution} to $(x_i)$ in $\mc X^{[n+1]}$ and $(x_1,\alpha(x_i))$ in $\mc X\times_S\alpha_n(\mc X^{[n]})$ respectively. 
Since $\alpha_*\delta_{x_i}=\delta_{\alpha(x_i)}$, we get
$$\alpha_*\mu_{I_1}=\mu_{I_2},$$
where $\mu_{I_1}$ (resp. $\mu_{I_2}$) is the equilibrum measure on $\mc X^{[n+1]}(\C)$ (resp. $(\mc X\times_S\alpha_n(\mc X^{[n]}))(\C)$).

Let $U_1$ (resp. $U_2$) be the $I_1$-non-degeneracy (resp. $I_2$-non-degeneracy) locus of $\mc X^{[n+1]}$ (resp. $\mc X\times_S\alpha_n(\mc X^{[n]})$).
Put $\Sigma_1:=(\mc G^{[n+1]})_{I_1,{\mathbb S^1}}$ and $\Sigma_2:=(\mc G^{[n]})_{I_2,{\mathbb S^1}}.$
Then by Theorem \ref{measure}, \begin{equation}\label{eq_measure_to_form}
    \mu_{I_1}|_{U_1}=\omega_1|_{U_1\cap\Sigma_1}\text{ (resp. }\mu_{I_2}|_{U_2}=\omega_2|_{U_2\cap \Sigma_2})
\end{equation} for some smooth non-vanishing form $\omega_1$ (resp. $\omega_2$) on $\Sigma_1$ (resp. $\Sigma_2$).

Let $V$ be a Zariski open subset such that $\alpha|_V:V\to\alpha(V)$ is finite and étale. Since $(\mc X^{[n]}\setminus V)(\C)$ (resp. $(\alpha(\mc X^{[n]})\setminus \alpha(V))(\C)$) is polaripolar, thus null under $\mu_1$ ({resp.} $\mu_2$). Hence $\mu_1$ is null outside $W=V\cap U_1\cap\alpha^{-1}(U_2)$. By comparing the support of $\alpha_*(\mu_1|_W)$ and $\mu_2|_{\alpha(W)}$, we get $\alpha(W\cap\Sigma_1)\subset\Sigma_2$. 

Moreover, by \eqref{eq_measure_to_form}, we can interpret the relation of measure pushforward as smooth forms pulling-back, that is, $$\alpha^*\omega_2=\omega_1$$ on $W\cap\Sigma_1$. 

Since $U_1\cap\Sigma_1$ is a real analytic manifold due to the definition of being $I_1$-non-degenerate, and the complement of $W$ in $(\mc X^{[n+1]})^{\mr{sm}}(\C)$ is real analytic, we get that $W\cap\Sigma_1$ is dense in $U_1\cap\Sigma_1$. Thus, $\alpha(U_1\cap\Sigma_1)\subset\Sigma_2$ and $\alpha^*\omega_2=\omega_1$ at $(x,x,\dots,x)\in U_1\cap \Sigma_1$. However, $\alpha^*\omega_2$ vanishes at $(x,x,\dots,x)$ (which can be obtained by calculating the differential of $\alpha$), while $\omega_1$ does not. We get a contradiction.
\end{proof}

\subsection{The ``extra term"}\label{subsec_proof_subc}
We prove a \emph{uniform Bogomolov type result} in the form with the ``extra term" as in \cite{Yuan2026bigness}. Following the idea of \cite{yu2026quantitativitynumberrationalpoints}, it will serve as an alternative to Mumford's inequality in next section.
\begin{theorem}
\label{mumford}
There exists a positive constants $c_3=c_3(g+t,d)>0$ satisfying the property as follows. Let $G$ be a principally polarized semiabelian variety over $\ovl\Q$ of dimension $g+t$. 

Let $X\subset G$ be a closed subvariety of degree $d$ generating $A$. Assume $X$ has finite stabilizer. Then there exists a reduced subscheme $Z\subsetneq X$ of degree $$\deg Z\le c_1(g+t,d),$$ and a point $g_0\in G(\ovl\QQ)$ such that for any $g_1\in G(\ovl\QQ)$, we have
$$\{x\in X^\circ(\ovl\Q):\widehat{h}(x-g_1)\le c_3(g+t,d)\widehat h(g_1-g_0)\}\subset Z(\ovl\Q).$$
\end{theorem}
\begin{proof}[Proof of Theorem \ref{mumford}]
Let $c_0$ and $g_0$ be as in forthcoming Proposition \ref{prop_height_inequality_translation}.
Let $c_2$ be as in Theorem \ref{bogomolov}. Set $c_3=c_2/c_0.$ 
Then we conclude the proof by applying Theorem \ref{bogomolov} to $X+g_1\subset G$.
\end{proof}

From now on, we introduce some notations that will be kept until the end of this subsection.
Set $B:=(\mc A_g^\vee)^{[t]}$. Let $\chi$ be a Hilbert polynomial, and $S$ be the $B$-subscheme of $\hilb_{g,t,d}^\circ$ that represents the functor
\begin{align*}
    (T\in (\textbf{Sch}/B)^{\mr{op}})&\longmapsto\\
    &\kern-6em\left\{Z \mathop{\subset}_{\text{closed}}\ovl {P}_T\,:\, \begin{aligned}Z\rightarrow T\text{ is flat}; \forall t\in T, Z_t\text{ is geometrically integral}, \\\text{of degree } d, \text{ with the Hilbert polynomial being }\chi\end{aligned}\right\}.
\end{align*}
Set that \begin{align*}&\kern7em\begin{cases}
    \mc P:=\mc P^\times_{g,t},\\\mc G:=S\times_{B}P,\\\XXX:=\mathfrak X_S\subset \mc G, \\ \ovl{\mc P}:=\text{compactification of }\mc P\text{ as in \S\ref{subsec_toric_compactification}}.
\end{cases}\end{align*}

For any $B$-schemes $T$ and $C$, we denote $C(T):=\mr{Hom}_{B}(T,C).$
We then define the map \begin{align*}
    \kern-4em f(T): \mc G(T)=\mc P(T)\times S(T)&\longrightarrow S(T),\\
                   (g_1,Z)&\longmapsto (\ovl t_{g_1}(Z))
\end{align*}
where $\ovl t_{g_1}:\ovl {\mc P}_T\rightarrow \ovl {\mc P}_T$ is the extension of the fiberwise translation $t_{g_1}:\mc P_T\rightarrow \mc P_T, x\mapsto x-g_1$ by the section $g_1: T\rightarrow \mc P.$ This is well defined since the Hilbert polynomial remains unchanged as $L$ and $M$ are numerically invariant under translation.
By Yoneda's Lemma, this corresponds to a $B$-morphism
$$f:\mc G=\mc P\times_B S\longrightarrow S$$
which is a fiberwise group action over $B$. For $s\in S$ and $g_1\in \mc P_{\pi_{S/B}(s)}$, we denote that $g_1\cdot s:=f(g_1,s)$. Note that $g_1\cdot s=[\mc X_s-g_1]$.

We say a subset $S'\subset S$ is \emph{stable under translation} if for any $s\in S'$, an extension field $K$ of the residue field $\kappa(s)$ of $s$ in $S$, and $g_1\in \mc P_{\pi_{S/B}(s)}(K)$,  we have $$g_1\cdot s_K:\mr{Spec}K\rightarrow S$$ factors through some point $s'\in S'.$ We simply denote this as $g_1\cdot s\in S'$ for $g_1\in \mc P_{\pi_{S/B}(s)}.$

In particular, every irreducible component of $S'$ is stable under translation. Indeed, if $S'$ is an irreducible component of $S$, then $f(\mc G_{S'})$ is irreducible. Since $S'=f(\epsilon_{\mc G_{S'}})\subset f(\mc G_{S'})$, we have $f(\mc G_{S'})=S'.$

\begin{lemma}\label{lemma_maximal_stable_under_translation}
    Put $S^\circ:=\{s\in S\,:\, \mathfrak X_s \text{ has finite stabilier}\}$. Then $S^\circ$ is open in $S$ and stable under translation.
\end{lemma}
\begin{proof}
It suffices to prove that $S^\circ$ is open in $S$.
Let $m:\GGG\times_S\GGG\to\GGG$ be the multiplication map. Then $\GGG\times_S\GGG$ is a semiabelian scheme over $\GGG$, where the structure morphism is the second projection, and $m^{-1}\XXX\subset\GGG\times_S\GGG$ is faithfully flat over the base $\GGG$ with geometrically integral fibers. 

Consider the closed subset $Z=m^{-1}\XXX\cap p_1^{-1}\XXX\subset\GGG\times_S\GGG$. For any point $g_1\in\GGG(\ovl\Q)$, it is in the stabilizer of $\XXX_{\pi_{\mc G/S}(g_1)}$ if and only if $\dim(p_2^{-1}(g_1)\cap Z)$ equals the relative dimension $d$ of $\XXX$ over $S$. By the semi-continuity of fiber dimension, such points form a Zariski closed subset $Y\subset\GGG$, as the relative dimension of $Z$ over the second $\mc G$ is a most $d$. 
Then $$S^\circ=\{s\in S\,:\, \dim Y_s=0\}.$$
Notice that $Y$ maps to $S$ surjectively since the stabilizer always contains the identity. Then again by the semi-continuity of fiber dimension, $S^\circ$ is open in $S.$
\end{proof}
\begin{lemma}\label{lemm_SUT_disjoint_union}
    Assume that $S_1\sqcup S_2\subset S$ is stable under translation. If $S_1$ is stable under translation, then so is $S_2.$
\end{lemma}
\begin{proof}
    Assume that there exists $s_2\in S_2$ and $g_2\in \mc P_{\pi_{S/B}(s_2)}$ such that $g_2\cdot s_2\in S_1$. Then
    $s_2=(-g_2)\cdot (g_2\cdot s_1)$ which contradicts to the stability of $S_1$.
\end{proof}
\begin{proposition}\label{prop_height_inequality_translation}
There exists a positive constant $c_0=c_0(g+t,d)>0$ satisfying the property as follows. Let $G$ be a principally polarized semiabelian variety over $\ovl\Q$ of dimension $g+t$. Let $X\subset G$ be a closed subvariety of degree $d$. Assume $X$ has finite stabilizer. Then there exists a point $g_0\in G(\ovl\Q)$ such that for any $g_1\in G(\ovl\Q)$,
$$h([X+g_1])\ge c_0\widehat{h}(g_1-g_0).$$
\end{proposition}
\begin{proof}
After running $\chi$ over all possible Hilbert polynomials (which form a finite set),
it suffices to prove that for any irreducible subset $S_1\subset S^\circ$ stable under translation, there is a positive constant $c_0(S_1)>0$ such that the following holds:\begin{equation}\label{eq_translation_minimum}\begin{aligned}
    &\text{for any }s\in S_1(\ovl\Q)\text{, there is a point }g_s\in\GGG_s(\ovl\Q)\text{ satisfying }\\
&\kern5em h([\XXX_s+g_1])\ge c_0(S_1)\widehat{h}(g_1-g_s)\\
&\text{for any }g_1\in\GGG_s(\ovl\Q).
\end{aligned} 
\end{equation}
Indeed, since $S^\circ$ itself is stable under invariant by Lemma \ref{lemma_maximal_stable_under_translation}, we only need to set $$c_0=\min\{c_0(S_1)\,:\, S_1 \text{ is a irreducible component of }S^\circ\}$$.
We prove by induction on $\dim S_1$. 
Fix a point $s_0\in S_1(\ovl\Q)$. Since $S_1\subset S^\circ$, the orbit $$\mc P_{\pi_{S/B}(s_0)}\cdot s_0=f(f^{-1}(s_0)),$$ which represents all translation of $\mc X_{s_0}$, is of dimension $g+t$.

Take a closed subvariety $W\subset S_1$ such that $\dim W=\dim S_1-(g+t)$ and $W$ intersects $f(f^{-1}(s_0))$ at finitely many points (not empty). Let $f_0:=f|_{\GGG_{W}}:\GGG_{W}\to S_1.$
Since $$\dim \mc G_W= \dim (\mc P\times_B W)=\dim W+g+t=\dim S_1,$$
and $f_0$ has finite fiber over $s_0$, by the semi-continuity of fiber dimension, $f_0$ is dominant and generically quasi-finite. 

Note that for $s\in S_1$ and $g_2\in \mc P_{\pi_{S/B}(s)},$
\begin{align*}
    f_0^{-1}(g_2\cdot s)&=\{(g_1,w)\,:\, g_1\cdot w=g_2\cdot s, w\in W_{\pi_{S/B}(s)},g_1\in \mc P_{\pi_{S/B}(s)}\}\\
    &=\{(g_2+g'_1,w')\,:\, (g_1',w')\in f_0^{-1}(s)\}.
\end{align*}
Therefore if we set
$$\begin{cases}
    S_0:=S_1\setminus f_0(\mc G_W),\\
    S_{\mr{fin}}:=\{s\in S_1\, :\, \dim f_0^{-1}(s)=0\},\\
    S_{\infty}:=\{s\in S_1\,:\, \dim f_0^{-1}(s)\geq 1\},
\end{cases}$$
then $S_0, S_{\mr{fin}}$ and $S_\infty$ is stable under translation. Moreover, $S_{\infty}\subsetneq S_1$ is Zariski closed.

Let $S'$ be the largest open subset of $S_{\mr{fin}}$ contained in $S_{\mr{fin}}\setminus S_0=f(\mc G_W)\setminus S_\infty=f(\mc G_{W}\setminus S_\infty)$. It is non-empty by Chevalley's theorem. 

We claim that $S_1\setminus S'$ is stable under translation. Hence by Lemma \ref{lemm_SUT_disjoint_union}, so is $S'$. Since $S_1\setminus S'=(S_{\mr{fin}}\setminus S')\sqcup S_{\infty}$, it suffices to show that $S_{\mr{fin}}\setminus S'$ is stable under translation.
Consider the set
\begin{align*}
    R&=\pi^{-1}_{\mc G/S}(S_{\mr{fin}}(\ovl{\Q})\setminus f(\mc G_W(\ovl {\Q})))\subset\mc G_{S_{\mr{fin}}\setminus S'}(\ovl{\mathbb Q}).
\end{align*}
which is dense in $\mc G_{S_{\mr{fin}}\setminus S'}$ by the maximality of $S'$ and the fact that $\pi_{\mc G/S}$ is open.

We show that $$f(R)\subset S_{\mr{fin}}(\ovl{\Q})\setminus f(\mc G_W(\ovl {\Q})) \subset S_{\mr{fin}}(\ovl {\Q})\setminus S'(\ovl {\Q}).$$ Assume the contrary, then there exists $s\in S_{\mr{fin}}(\ovl {\Q})\setminus f(\mc G_W(\ovl {\Q}))$ and $g_1\in \mc P_{\pi_{S/B}(s)}(\ovl {\Q})$ such that $g_1\cdot s=g_2\cdot w$ for $(g_2,w)\in \mc G_W(\ovl {\Q}).$ Therefore $s=(-g_1+g_2)\cdot w\in f(\mc G_W(\ovl{\Q}))$ which is a contradiction. By the density of $R$, we obtain that $S_{\mr{fin}}\setminus S'$ is stable under translation.

Applying the induction hypothesis to irreducible components of $S_1\setminus S'$, we obtain a constant $c_0(S_1\setminus S')>0$ such that \eqref{eq_translation_minimum} holds for $S_1\setminus S'$. 

Recall the adelic line bundle $\ovl H$ on $S$ defining the modular height. 
Put $W'=W\cap S'$ and $f':=f_0|_{\mc G_{W'}}:\mc G_{W'}\rightarrow S'$ which is surjective and quasi-finite.

Therefore $f'^*\ovl H$ is big and nef on every closed subvariety of $\GGG_{W'}$. By Lemma \ref{lemma_height_ineq_nef_and_big}, there exists $c_0(S')>0$ such that
\begin{equation}\label{eq_height_ineq_modular}h_{f'^*\ovl H}(\cdot)\ge c_0(S')\widehat{h}(\cdot)\end{equation}
on $\GGG_{W'}(\ovl\Q)$.

For any $s\in S'(\ovl\Q)$, that $f'$ is surjective implies $s=g_0\cdot w$ for some $w\in W'_{\pi_{S/B}(s)}$ and $g_0\in \mc P_{\pi_{S/B}(s)}$. For any $g_1\in\GGG_{s}(\ovl\Q)=\GGG_w(\ovl \Q)=\mc P_{\pi_{S/B}(s)}(\ovl{\Q})$, we have $$(-g_1)\cdot s=(g_0-g_1)\cdot w={f'(g_0-g_1,w)}.$$ The above inequality \eqref{eq_height_ineq_modular} reads
$$h([\XXX_s+g_1])\geq h_{\ovl H}((-g_1)\cdot s)=h_{f'^*\ovl H}(g_0-g_1,w)\ge c_0(S')\widehat{h}(g_1-g_0).$$
We then conclude the proof by setting $c_0(S)=\min\{c_0(S'),c_0(S_1\setminus S')\}.$
\end{proof}

\section{Uniform Mordell--Lang conjecture}\label{section_UML}

Recall our main theorem is the following.

\begin{theorem}\label{ML}
There exists a constant $c(g+t,d)$ satisfying the property as follows. Let $K$ be a field of characteristic $0$. Let $G$ be a polarized semiabelian variety over $K$ of dimension $g+t$. Let $\Gamma\subset G(K)$ be a subgroup of finite rank. Let $X\subset G$ be a closed subvariety of degree $d$. Then there are at most $c(g+t,d)^{\rk(\Gamma)+1}$ cosets $H_1,H_2\dots$ contained in $X$ such that
$$X(K)\cap\Gamma=\bigcup H_i(K)\cap\Gamma.$$
\end{theorem}

\subsection{Uniform Mordell--Lang conjecture for points}

\begin{proposition}\label{ML1}
There exists a constant $c_7(g+t,d,n)$ satisfying the property as follows. Let $G$ be a principally polarized semiabelian variety of dimension $g+t$ over $\ovl\Q$. Let $X\subset G$ be a closed subvariety of degree $d$ and dimension $n$. Let $\Gamma\subset G(\ovl\Q)$ be a subgroup of finite rank. Then there exist reduced closed subschemes $Z_i\subsetneqq X$ $(i=1,2,\dots,c_7(g+t,d,n)^{\rk(\Gamma)+1})$ of degree at most $c_1(g+t,d)$ such that
$$\#(X^\circ(\ovl\Q)\cap\Gamma)\subset\bigcup_{i=1}^{c_7(g+t,d,n)^{\rk(\Gamma)+1}}Z_i(\ovl\Q).$$
Here $c_1(g+t,d)$ is the constant in Theorem \ref{bogomolov}
\end{proposition}

We recall the height functions on $\Gamma\otimes_{\Z}\R$ before proving the above proposition. Both ${h}_{\ovl M}(\cdot)$ and $\sqrt{h_{\ovl L}(\cdot)}$ extend to semi-norms on $\Gamma\otimes_{\Z}\R$, that is, they are non-negative, homogeneous and satisfy triangular inequality. For the triangular inequality of ${h}_{\ovl M}(\cdot)$, see \cite[Corollaire 3.1]{remond}. We now show that $\|\cdot\|:={h}_{\ovl M}(\cdot)+\sqrt{h_{\ovl L}(\cdot)}$ is a norm on $\Gamma\otimes_{\Z}\R$. It only remains to prove the strict positivity, i.e. $\|x\|=0$ implies $x=0$ for $x\in\Gamma\otimes_\Z\R$. We may assume $G,L$ and $\Gamma$ are defined on a number field by replacing $\Gamma$ with a finitely generated subgroup, which does not change $\Gamma\otimes_\Z\R$. Fix a Haar measure on $\Gamma\otimes_\Z\R$. If the semi-norm $\|\cdot\|$ is not strictly positive on $\Gamma\otimes_\Z\R$, then the domain
$$\{x\in\Gamma\otimes_\Z\R:\|x\|\le1\}$$
has infinite volume. This domian is contained in
$$\{x\in\Gamma\otimes_\Z\R:\hat{h}(x)\le2\}.$$
So the lattar also has infinite volume, and hence
$$\{x\in\Gamma:\hat{h}(x)\le2\}$$
is an infinite set, contradicting the Northcott theorem.

Since $\|\cdot\|$ is a norm, for any $r_1,r_2>0$, the polyball
$$B(x,r_1,r_2)=\{y\in\Gamma\otimes_\ZZ\R:{h}_{\ovl M}(y-x)\le r_1,{h}_{\ovl L}(y-x)\le r_2^2\}$$
has finite volume, and its volume is proportional to $$r_1^{\rk(\Gamma\cap\mathbb{G}_m^t(\C))}r_2^{\rk(\Gamma)-\rk(\Gamma\cap\mathbb{G}_m^t(\C))}.$$

We will use the following simple fact for several times. It relates the gap principle and point counting.

\begin{lemma}\label{lemma_point_counting}
Let $P\subset\Gamma$ be a subset. Let $h_0,h_1>0$ be positive numbers. Assume that $\hat{h}(x)\le h_0$ for any $x_0\in P$, and $\hat{h}(x_1-x_2)\ge h_1$ for any distinct $x_1,x_2\in P$. Then
$$\# P\le\max\left\{\frac{4h_0}{h_1}+1,\sqrt\frac{8h_0}{h_1}+1\right\}^{\rk(\Gamma)+1}.$$
\end{lemma}

\begin{proof}
By the two assumptions, the union
$$\bigcup_{x\in P}B\left(x,\frac{1}{4}h_1,\sqrt{\frac{1}{8}h_1}\right)$$
is disjoint and contained in
$$B\left(0,h_0+\frac{1}{4}h_1,\sqrt{h_0}+\sqrt{\frac{1}{8}h_1}\right).$$
Computation of the volumes gives the conclusion.
\end{proof}

\begin{proof}[Proof of Proposition \ref{ML1}]
Fix a level-$4$ structure on $G$. Take $g_0\in G(\ovl\Q)$ as in Theorem \ref{mumford}. We may assume $g_0=0$. Indeed, take $X'=X+g_0$ and the subgroup $\Gamma'\subset G(\ovl\Q)$ generated by $\Gamma$ and $g_0$. Then $\rk(\Gamma)\le\rk(\Gamma')+1$. Once we have
$$\#((X')^\circ(\ovl\Q)\cap\Gamma')$$
is covered by $c_7(g+t,d,n)^{\rk(\Gamma')+1}$ subschemes $Z_i$, this would implies
$$\#(X^\circ(\ovl\Q)\cap\Gamma)\ \le\#(X^\circ(\ovl\Q)\cap\Gamma')=\#((X')^\circ(\ovl\Q)\cap\Gamma')$$
is covered by
$$c_7(g+t,d,n)^{\rk(\Gamma')+1}\le\max\{c_7(g+t,d,n)^2,1\}^{\rk(\Gamma)+1}$$
subschemes.

We list the results to be used in counting points.
\begin{enumerate}
    \item[(1)](Theorem \ref{bogomolov}) For any $x\in X(\ovl\Q)$, there exists a $Z_x\subset X$ of degree at most $c_1(g+t,d)$ such that
    $$\{y\in X^\circ(\ovl\Q):\widehat h(y-x)\le c_2(g+t,d)h([X])\}\subset Z_x(\ovl\Q)$$
    \item[(2)](Theoerm \ref{mumford}) For any $x\in X(\ovl\Q)$, there exists a $Z_x'\subset X$ of degree at most $c_1(g+t,d)$ such that
    $$\{y\in X^\circ(\ovl\Q):\widehat h(y-x)\le c_3(g+t,d)\widehat h(x)\}\subset Z_x'(\ovl\Q)$$
    \item[(3)](\cite[Theorem 4.1]{remond}) There exist positive constants $c_4(g+t,d,n),c_5(g+t,d,n)$ and $c_6(g+t,d,n)$ such that for $x_1,\dots,x_{n+1}\in X(\ovl\Q)$ satisfying
    $$h_{\ovl M}\left(\frac{x_i}{\widehat{h}(x_i)}-\frac{x_{i+1}}{\widehat{h}(x_{i+1})}\right)\le\frac1{c_4},\quad h_{\ovl L}\left(\frac{x_i}{\widehat{h}(x_i)}-\frac{x_{i+1}}{\widehat{h}(x_{i+1})}\right)\le\frac1{c_4^2},$$
    $$\widehat h(x_{i+1})\ge c_5^2\widehat h(x_i)\quad\mr{and}\quad\widehat{h}(x_1)\ge c_6h([X])\footnote{The original statement of \cite[Theorem 4.1]{remond} didn't provide uniform constants. However, in its proof, one can see that the constants are explicitly chosen using the notations appeared in \cite{remond2005ineq}, so that we get the inequality $\widehat h(x_1)\geq c_6'(g+t,d,n) \max \{h(X),(t+1)h_{\mr{add}}(A)\}$, where $h_{\mr{add}}(A)$ is a height function associated with the group law on the abelian quotient $A$. By \cite[(8.7)]{DGH}, $h_{\mr{add}}(A)\leq \ell \max\{1,h_{\mr{Fal}}(A)\}$, for some $\ell>0$, hence can be bounded by $h([X])$ up to a positive constant due to Lemma \ref{lemma_height_ineq_nef_and_big}. Combining with \eqref{eq_height_ineq_modular}, we obtain the current form.},$$
    we have $x_i\notin X^\circ(\ovl\Q)$ for some $i$.
\end{enumerate}

In the following, we say $$\{x\in X^\circ(\ovl\Q):\widehat{h}(x)\le c_6(g+t,d,n)h([X])\}$$
is the set of \emph{small points}. Its complement in $X^\circ(\ovl {\Q})\cap\Gamma$ is called the set of \emph{large points}.

We take a maximal subset $P_1$ of the small points set
such that any distinct $x_1,x_2\in P_1$ satisfy
$$\widehat{h}(x_1-x_2)>c_2(g+t,d)h([X]).$$
By (1), small points are covered by
$$\bigcup_{x\in P_1}Z_x(\ovl\Q).$$
By Lemma \ref{lemma_point_counting},
$$\#P_1\le\max\left\{\frac{4c_6(g+t,d,n)}{c_2(g+t,d)}+1,\sqrt\frac{8c_6(g+t,d,n)}{c_2(g+t,d)}+1\right\}^{\rk(\Gamma)+1}.$$
Denote the base number by $c_7'(g+t,d,n)$

For any $r>0$ and ratio $\kappa>1$, consider the annular region
$$A(r,\kappa)=\{x\in\Gamma\otimes_\ZZ\R:r\le\widehat{h}(x)\le\kappa r\}.$$
Choose a maximal subset $Q(r,\kappa)$ of $A(r,\kappa)$ such that any distinct $x_1,x_2\in Q(r,\kappa)$ satisfy
$$\widehat{h}(x_1-x_2)>c_3(g+t,d)r.$$
By (2), $X^\circ(\ovl\Q)\cap A(r,\kappa)$ is covered by
$$\bigcup_{x\in Q(r,\kappa)}Z_x'.$$
By Lemma \ref{lemma_point_counting},
$$\#Q(r,\kappa)\le\max\Big\{\frac{4\kappa}{c_3(g+t,d)}+1,\sqrt\frac{8\kappa}{c_2(g+t,d)}+1\Big\}^{\rk(\Gamma)+1}.$$
We take $\kappa=c_5^4(g+t,d,n)$. Then the base number of the right hand side depends only of $g+t,d$ and $n$. Denote it by $c_7''(g+t,d,n).$

Take a maximal subset $P_2$ of large points such that any distinct $x_1,x_2\in P_2$ satisfy either
$${h}_{\ovl M}\left(\frac{x_1}{\widehat h(x_1)}-\frac{x_2}{\widehat h(x_2)}\right)>\frac1{c_4(g+t,d,n)}$$
or
$${h}_{\ovl L}\left(\frac{x_1}{\widehat h(x_1)}-\frac{x_2}{\widehat h(x_2)}\right)>\frac1{c_4^2(g+t,d,n)}.$$
Then large points is covered by the cones
$$\left\{y\in\Gamma\otimes_\ZZ\R:{h}_{\ovl M}\left(\frac{y}{\widehat h(y)}-\frac{x}{\widehat h(x)}\right)\le\frac1{c_4(g+t,d,n)},{h}_{\ovl L}\left(\frac{y}{\widehat h(y)}-\frac{x}{\widehat h(x)}\right)\le\frac1{c_4^2(g+t,d,n)}\right\}$$
for $x\in P_2$. By (3), every cone is covered by $n$ annuli. Now it remains to bound $\# P_2$. Since
$$\bigcup_{x\in P_2}B\left(\frac x{\widehat h(x)},\frac{1}{2c_4(g+t,d,n)},\frac{1}{2c_4(g+t,d,n)}\right)$$
is a disjoint union contained in $\displaystyle B\left(0,1+\frac{1}{2c_4(g+t,d,n)},1+\frac{1}{2c_4(g+t,d,n)}\right)$, we get
$$\#P_2\le(2c_4(g+t,d,n)+1)^{\rk(\Gamma)+1}.$$

In conclusion, the number of reduced subschemes $Z$ of degree at most $c_1$ that we need to cover $X^\circ(\ovl\Q)\cap\Gamma$ is
$$(c_7')^{\rk(\Gamma)+1}+n(c_7'')^{\rk(\Gamma)+1}(2c_4+1)^{\rk(\Gamma)+1}\le\Big(c_7'+nc_7''(2c_4+1)\Big)^{\rk(\Gamma)+1}.$$
\end{proof}

\begin{lemma}\label{pp_reduction_lemma}
There exists a constant $a(g+t,d)$ satisfying the property as follows. For any polarized semiabelian variety $(G,L)$ of dimension $g+t$ and subvariety $X$ of degree $d$, if $X$ generates the abelian quotient $A$, then there exists a principally polarized semiabelian variety $(G',L')$ and an isogeny $f:G'\to G$ of degree at most $a(g+t,d)$ such that $f^*L$ is proportional to $L'$, i.e., $f^*L\cong nL'$ for some $n\ge1$. Moreover, the restriction of $f$ on the toric part is the identity.
\end{lemma}

\begin{proof}
Denote $Y={\pi(\ovl X)}\subset A$.
By \cite[Lemma 2.2 \& Lemma 2.5]{GGK}, there is an isogeny $f'_{\ab}:A\to A'$ such that $\deg(f'_{\ab})\le a'(g+t,\deg_L(Y))$ for some constant $a'$ and a principal polarization $L'$ on $A'$ such that $(f'_{\ab})^*L'\cong L$. Since $L$ and $M$ are nef, we have $$\deg_L(Y)=(L|_{Y}^{\dim Y})\le(\pi^*L|_{\ovl X}^{\dim Y})\cdot(M|_{\ovl X}^{\dim X-\dim Y})\le(L+M)|_{\ovl X}^{\dim X}=d$$
where the first inequality is due to the projection formula, and the second is due to the binomial expansion. We may assume that $a$ is increasing in $\deg_L(Y)$, then $$\deg(f_{\ab}')\leq a'(g+t,d).$$ 
Take a quasi-inverse $f_{\ab}:A'\to A$ of $f_{\ab}'$, i.e., $f_{\ab}$ is an isogeny such that $f_{\ab}'\circ f_{\ab}=[\deg(f'_{\ab})]:A'\to A'$. Then $\deg(f_{\ab})\le\deg(f'_{\ab})^{g+t}\le a'(g+t,d)^{g+t}$ and
$$f_{\ab}^*L=f^*_{\ab}(f'_{\ab})^*L'=\deg(f'_{\ab})^2L'.$$
We then set $G'$ to be the pull-back of $G$ via $f_{\ab}$, and $a(g+t,d):=a'(g+t,d)^{g+t}$. Then $G',L'$ and $a$ are as desired.
\end{proof}

\begin{corollary}\label{MLforpoints}
There exists a constant $c'(g+t,d)$ satisfying the property as follows. Let $G$ be a polarized semiabelian variety of dimension $g+t$ over $\ovl\Q$. Let $X\subset G$ be a reduced closed subscheme of degree $d$. Let $\Gamma\subset G(\ovl\Q)$ be a subgroup of finite rank. Then
$$\#(X^\circ(\ovl\Q)\cap\Gamma)\le c'(g+t,d)^{\rk(\Gamma)+1}.$$
\end{corollary}

\begin{proof}
We proceed by induction on $\dim X.$

If $\dim X=0$, then
$$\#(X^\circ(\ovl\Q)\cap\Gamma)\le\# X(\ovl\Q)=d.$$
Assume that proposition statement holds in the case that $\dim X<n$ with constant $c^*(g+t,d,n).$

We now prove it in the case that \begin{equation*}
   \kern-4em\text{(HP0)}\qquad \dim X=n.
\end{equation*}

{\textbf{Step 1.}} First, we reduce to the case that
\begin{equation*}
     \text{(HP1)}\quad
        X\text{ is irreducible, and generates }A. 
\end{equation*}
Assume that the statement of the proposition holds under (HP0-1) for some constant $c''(g+t,d,n).$ We may and do further assume that $c''(g+t,d,n)\geq c^*(g+t,d,n)$ after replacing $c''$ with $\max\{c'',c^*\}.$ Moreover we assume that $c''(g+t,d,n)$ is increasing in $g+t$ and $d$ by replacing $c''(g+t,d,n)$ with
$$\max_{0\le i\le g+t,1\le j\le d}c''(i,j,n).$$

Let $X_1,X_2,\dots, X_m(m\leq d)$ be irreducible components of $X.$ 
For each $i=1,\dots,m$, put $A_i$ to be the abelian subvariety of $A$ generated by $X_i$, and take any $x_i\in X_i(\ovl{\Q})$ such that $\pi(X_i-x_i)\subset A_i.$ We set $G_i=\pi^{-1}(A_i)$ and $\Gamma_i:=\langle x_i,\Gamma\rangle\cap G_i(\ovl \Q).$ Hence $\rk \Gamma_i\leq \rk \Gamma+1.$

Then
\begin{align*}
\#(X^\circ(\ovl\Q)\cap\Gamma)&\ \le\sum_{i=1}^m\#(X_i^\circ(\ovl\Q)\cap\Gamma)\\
&\ \le\sum\# ((X_i-x_i)^\circ (\ovl \Q)\cap\Gamma_i)\\
&\ \le\sum_{i=1}^mc''(\dim A_i+t,\deg X_i,n)^{\rk(\Gamma_i)+1}\\
&\ \le\Big(dc''(g+t,d,n)\Big)^{\rk(\Gamma)+2},\\
&\ \leq c'^{\rk(\Gamma)+1}
\end{align*}
where
$$c':=(dc''(g+t,d,n))^2$$
depends only on $g+t,d$ and $n$.

\textbf{Step 2.} Under the assumption (HP1), now we further reduce to the case that
\begin{equation*}
    \text{(HP2)}\quad G\text{ is principally polarized.} 
\end{equation*}
Assume that the proposition holds under hypothesis (HP0-2) for some constant $c'''(g+t,d,n)$. We want to show that the proposition holds under only (HP0-1). Again, we may assume that $c'''(g+t,d,n)\geq c^*(g+t,d,n).$

By Lemma \ref{pp_reduction_lemma}, there is an isogeny $f:G'\to G$ of degree at most $a(g+t,d)$ such that $f^*L$ is proportional to a principal polarization $L'$ on $G'$. Take an irreducible component $X'$ of $f^{-1}(X)$ which maps onto $X$, which exists because $f:f^{-1}(X)\to X$ is étale and finite.

Note that $f$ is the identity if restricted to the toric part. Therefore we we have the extension $\ovl f:\ovl G'\rightarrow \ovl G$ which is finite of the same degree as $f$. If we denote that $M':=\O_{\ovl G'}(\ovl G'\setminus G')$, then we have $M'=\ovl f^*M.$ We have
\begin{align*}\deg_{L'+M'}(\ovl {X'})&\le \deg_{\ovl f^*L+M'}(\ovl {X'})=\deg_{\ovl f^*(L+M)}(\ovl {X'})\\&=\deg(f)\deg_{L+M}(\ovl X)\le a(g+t,d)d\end{align*}
where the first inequality is due the nefness of $L'$ and $M'$.

We may assume that $c'''(g+t,d,n)$ is increasing in $d$. Then
$$\#((X')^\circ(\ovl\Q)\cap f^{-1}(\Gamma))\le c'''(g+t,a(g+t,d)d,n)^{\mr{rk}(f^{-1}(\Gamma))+1}.$$
Since $f$ is an isogeny, $f^{-1}(\Gamma)\to\Gamma$ is surjective with finite kernel, and hence $\mr{rk}(f^{-1}(\Gamma))=\mr{rk}(\Gamma)$.

For any $x\in X^{\circ}(\ovl \Q)\cap\Gamma$, there is an $x'\in X'(\ovl\Q)\cap f^{-1}(\Gamma)$ such that $f(x')=x$. Moreover, we have $x'\in (X')^\circ(\ovl\Q)$. Otherwise $x'\in H'(\ovl\Q)$ for some positive dimensional coset $H'\subset X'$. Then $x\in f(H')(\ovl\Q)$, where $f(H')\subset X$ is a positive dimension coset, contradicting $x\in X^\circ(\ovl\Q)$. Therefore,
$$\#(X^\circ(\ovl\Q)\cap\Gamma)\le\#((X')^\circ(\ovl\Q)\cap f^{-1}(\Gamma))\le c'''(g+t,a(g+t,d)d,n)^{\mr{rk}(\Gamma)+1}.$$
We thus set $c''(g+t,d,n)=c'''(g+t,a(g+t,d)d,n).$

\textbf{Step 3.} Now  it suffices to prove the proposition under the assumption (HP0-2).
Applying Proposition \ref{ML1} to $X$, we get reduced closed subschemes $Z_i\subset X$ $(i=1,2,\dots,c_7(g+t,d,n)^{\rk(\Gamma)+1})$ of degree at most $c_1(g+t,d)$. By the induction hypthesis,
$$\#(Z_i^\circ(\ovl\Q)\cap\Gamma)\le c^*(g+t,c_1(g+t,d),n)^{\rk(\Gamma)+1}.$$
Since $X^\circ\cap Z_i\subset Z_i^\circ$,
$$\#(X^\circ(\ovl\Q)\cap\Gamma)\le(c_7(g+t,d,n)c^*(g+t,c_1(g+t,d),n))^{\rk(\Gamma)+1}.$$
\end{proof}

\subsection{Ueno locus}\label{subsection_relative_ueno}

We define a relative version of Ueno locus and show that it is Zariski closed, which allows specialization and induction on dimension later.

Let $S$ be a quasi-projective variety over an algebraically closed field $K$ of characteristic $0$. Let $\GGG$ be a polarized semiabelian scheme over $S$. Let $\XXX\subset\GGG$ be a subvariety. Define the \emph{relative Ueno locus} of $\XXX$ over $S$ as the union of the Ueno locus of $\XXX_s$ for all $s\in S$.

\begin{proposition}\label{prop_relative_ueno} 
There are finitely many pairs $(S_1,\GGG_1),(S_2,\GGG_2)\dots(S_N,\GGG_N)$ of a closed subvariety $S_i\subset S$ and a semiabelian subscheme $\GGG_i\subset\GGG_{S_i}=\GGG\times_SS_i$ such that for any $s\in S(K)$, the Ueno locus of $\XXX_s$ is a union of cosets of $\GGG_{i,s}$ for $S_i$ containing $s$.

In particular, the relative Ueno locus is a Zariski closed set.
\end{proposition}
We first need the following auxiliary lemma. Recall that our notations are taken from \S\ref{section_Shimura}.
\begin{lemma}\label{lemma_fiber_weakly_special_subvariety}
Let $M\subset\mc P_{g,t}^\times$ be a special subvariety with connected Mixed Shimura datum $(Q,\mc Y^+)$. Let $N\trianglelefteq Q$ be a normal subgroup. Denote $S=\pi_{\sab}(M)\subset(\mc A_g^\vee)^{[t]}$ and $\mc G=\pi_{\sab}^{-1}(S)$. Then there is a semiabelian subscheme $\mc H\subset\mc G$ such that for any $\widetilde y\in\mc Y^+$ and $s\in S(\C)$, each irreducible component of $\uni(N(\mathbb R)^+U_N(\mathbb C)\widetilde y)\cap \pi_{\sab}^{-1}(s)$ is a coset of $\mc H_s$.
\end{lemma}

\begin{proof}
It can be verified that $\mc G$ is also a special subvariety with the connected mixed Shimura datum $(Q_1,\mc Y_1^+)$ for
$$Q_1=(\mr M_{t,1}\oplus \mr M_{2g,1})\cdot Q,\quad\mc Y_1^+=\mr M_{t,1}(\C)\times \mr M_{2g,1}(\R)\times\widetilde{\pi}_{\mathrm{sab}}(\mc Y^+).$$

Since $\mr M_{t,1}\oplus \mr M_{2g,1}$ is normal in $P_{g,t}$, $(U_N\oplus V_N)=(\mr M_{t,1}\oplus \mr M_{2g,1})\cap N$ is normal in $Q_1$.

By the same argument as in the proof of Corollary \ref{coro_semiabelian_subscheme}, we have a semiabelian subcheme $\mc H$ of $\mc G$ such that the Shimura quotient $\mc G_{/(U_N\oplus V_N)}$ is the quotient $\mc G/\mc H$ as semiabelian scheme over $S$.

Note that $$\rho_{/(U_N\oplus V_N)}(\uni(N(\R)^+ U_N(\C)\widetilde y))=\uni(p_{\sab}(N)(\R)^+\widetilde\pi_{\sab}(\widetilde y))\subset \mc G/\mc H$$ is a multi-section of $S$, and $\uni(N(\R)^+U_N(\C)\widetilde y)\cap \pi_{\sab}^{-1}(s)$ is exactly the preimage of $\pi_{\sab}^{-1}(s)\cap \uni(p_{\sab}(N)(\R)^+\widetilde\pi_{\sab}(\widetilde y))$, which concludes the proof.
\end{proof}

\begin{proof}[Proof of Proposition \ref{prop_relative_ueno}]
To see how the first assertion implies the Ueno locus is Zariski closed, consider the quotient map $f_i:\GGG_{S_i}\to\GGG_{S_i}/\GGG_i$. Then the set
$$\{z\in \mc G_{S_i}\,:\, \dim_z (f_i^{-1}(f_i(z))\cap \mc X)\geq \dim \mc G_i-\dim S_i\}$$
is Zariski closed by the upper semicontinuity, and the relative Ueno locus is the union of all such set.

Now we prove the first assertion, i.e., the finiteness result.
We may assume $\GGG/S$ is principally polarized with a level-$4$ structure by taking a finite étale cover. So there exist modular maps $\mc G\to\mc P_{g,t}^\times$ and $S\rightarrow (\mc A_g^\vee)^{[t]} $. After taking an embedding $K\hookrightarrow \C$, we may also assume $K=\C$. Note that there is no loss of generality since if a semiabelian variety $G$ is defined over $K$, then any semiabelian subvariety $G'\subset G_\C$ is defined over $K$ (see for example the proof of \cite[Proposition 3.7]{Yamaki2018trace} or \cite[Theorem B]{liu2024chowtrace1motiveslangneron}).

We first assume that the modular map is injective. Recall the geometric Zilber--Pink \cite[Theorem 7.1]{BU}. The case $\delta_{\mathrm{ws}}=0$ states that there are finitely many triplets $(Q_i,\mathcal Y_i^+,N_i)$, where $(Q_i,\mathcal Y_i^+)$ is a connected mixed Shimura subdatum and $N_i\trianglelefteq Q_i$ is a normal subgroup with reductive part being semisimple, such that each weakly optimal and weakly special subset of $\XXX$ is of form $\uni(U_{N_i}(\C)N_i(\R)^+\widetilde y_i)$ for some $\widetilde y_i\in \mc Y_i^+$.

Fix an index $i$. We show that for any $\widetilde y\in\mathcal{Y}_i^+$ and $s\in S(K)$, each irreducible component of $\uni(N_i(\R)^+U_{N_i}(\C)\widetilde y)\cap\GGG_s$ is of dimension $$\dim(N_i(\R)^+U_{N_i}(\C)\cap (\mr M_{t,1}(\C)\oplus \mr M_{2g,1}(\R)),$$ which depends only on $i$. Note that it is not necessarily non-empty. For any
$$y'\in \uni(N_i(\R)^+U_{N_i}(\C)\widetilde y)\cap\GGG_s,$$
there is a lifting $\widetilde y'$ such that
$$\uni(N_i(\R)^+U_{N_i}(\C)\widetilde y)=\uni(N_i(\R)^+U_{N_i}(\C)\widetilde y')$$
and
$$\GGG_s=\uni((\mr M_{t,1}(\C)\oplus \mr M_{2g,1}(\R))\widetilde y').$$
Since $\uni$ is a local diffeomorphism, we obtain that that $\uni(N_i(\R)^+U_{N_i}(\C)\widetilde y)\cap\GGG_s$ is of the same dimension as
$$(N_i(\R)^+U_{N_i}(\C)\widetilde y')\cap (\mr M_{t,1}(\C)\oplus \mr M_{2g,1}(\R))\widetilde y')$$
near $\widetilde y'$, which is exactly $\dim(N_i(\R)^+U_{N_i}(\C)\cap (\mr M_{t,1}(\C)\oplus \mr M_{2g,1}(\R))$.

Let $I$ be the set of index satisfying that 
$$\dim(N_i(\R)^+U_{N_i}(\C)\cap (\mr M_{t,1}(\C)\oplus \mr M_{2g,1}(\R)))>0.$$
We claim any maximal coset $Z$ of positive dimension in $\XXX_s$ is an irreducible component of $\uni(N_i(\R)^+U_{N_i}(\C)\widetilde y)\cap\GGG_s$ for some $i\in I$. Indeed, since $Z$ is weakly special by \cite[Proposition 7.2]{ZGAxLindemann}, $Z$ is contained in some weakly optimal weakly special subset $\uni(N_i(\R)^+U_{N_i}(\C)\widetilde y)\subset\XXX$. Therefore $\uni(N_i(\R)^+U_{N_i}(\C)\widetilde y)\cap\GGG_s\subset \mc X_s$, whose irreducible components are cosets.  We thus obtain the claim by the maximality of $Z$.

For any $i\in I$, by Lemma \ref{lemma_fiber_weakly_special_subvariety}, for $S_i=\pi_{\mathrm{sab}}(\uni(\mathcal Y_i^+))\cap S$, there is a subgroup scheme $\mc G_i\subset\mc G_{S_i}$ such that each irreducible component of $\uni(N_i(\R)^+ U_{N_i}(\C)\widetilde y)\cap\mc G_s$ is a coset of $(\mc G_i)_s$ for $s\in S_i(K)$. This concludes the proof under the assumption that the modular map is injective.


In general, we use a self-product argument to reduce to the previous case.

Take an integer $n$ strictly larger than the dimension of any fiber of $S\to(\mc{A}_g^\vee)^{[t]}$. Let $\YYY$ be the image of $$\mc X^{[n]}\rightarrow (\mc P_{g,t}^\times)^{[n]}\hookrightarrow\mc P_{ng,nt}^\times.$$ Denote by $\iota:S\to(\mc A_{ng}^\vee)^{[nt]}$ the moduli map. Note that the dimension of any fiber of $\iota$ is still strictly smaller than $n$. We prove that for any $s\in S(K)$ and subgroup variety $H\subset\GGG_s$, there is a coset of $H^n$ contained in $\YYY_{\iota(s)}$ if and only if there is a coset of $H$ contained in $\XXX_{s'}$ for some $s'\in\iota^{-1}(\iota(s))$. Note that $\iota(s')=\iota(s)$ implies $\GGG_{s'}=\GGG_s$ since $(\mc A_{ng}^\vee)^{[nt]}$ is a fine moduli space. Hence we can view $H$ as a subgroup of $\GGG_{s'}$.

The ``if" part is obvious since $(\XXX_{s'})^n\subset\YYY_{\iota(s)}$ for any $s\in S(K)$. Conversely, assume $H$ satisfies that $H^n+(x_1,\dots,x_n)\subset\YYY_{\iota(s)}$ for $x_1,\dots,x_n\in\XXX_s$. Let $W$ be its preimage in $\mc X^{[n]}$. Then $\dim W\ge n\dim H.$

Note that $W$ lies over $\iota^{-1}(\iota(s))\subset S$. Denote by $D(W,\iota^{-1}(\iota(s)))$ the maximal dimension of fiber of $W$ over $\iota^{-1}(\iota(s))$. The notation applies to other schemes. Then
$$D(W,\iota^{-1}(\iota(s)))\le\sum_{i=1}^n D(p_i(W),\iota^{-1}(\iota(s))),$$
where $p_i:\mc X^{[n]}\to\XXX$ is the $i$-th projection. By our choice, $n>\dim(\iota^{-1}(\iota(s)))$. So 
$$D(W,\iota^{-1}(\iota(s)))\ge\dim W-\dim(\iota^{-1}(\iota(s)))>n(\dim H-1).$$
By the pigeonhole principle,
$$D(p_i(W),\iota^{-1}(\iota(s)))\ge\dim H$$
for some $i$. However, $p_i(W)\subset\GGG_{\iota^{-1}(\iota(s))}=\GGG_s\times\iota^{-1}(\iota(s))$ satisfies $p_i(W)\subset (H+x_i)\times\iota^{-1}(\iota(s))$. So for any $s'\in\iota^{-1}(\iota(s))$ satisfying that $\dim p_i(W)_{s'}=\dim H$, we have $(H+x_i)\subset\mc X_{s'}$.

By the previous case, there exist closed subvarieties $S'_\ell\subset \iota(S)$ and semiabelian subschemes $\mc H'_\ell\subset \mc G'_{S'_\ell}$ such that the proposition statement holds for $\mc Y$, where $\mc G'$ is the image of $\mc G^{[n]}$ in $\mc P^\times_{ng,nt}$. 

Let $\{S_i\}$ be the set $\{\text{irreducible components of }\iota^{-1}(S_\ell')\}$.
In the following, we take the convention of index that $S_i$ is an irreducible component of $\iota^{-1}(S_\ell').$ Let $\mc H_i=(\mc H_{\ell}')\times_{S_\ell'}S_i\subset (\mc G^{[n]})_{S_i}.$
Denote by $s_j:\GGG\to\GGG^{[n]}$ the $j$-th coordinate. 
Then we put 
$$\GGG_i=\text{the neutral component of }\bigcap_{j=1}^ns_j^*\HHH_i,$$
which makes the proposition statement hold for $\XXX$ (here the equality means $\GGG_i$ is the largest semiabelian subscheme such that $\GGG_i^n\subset\HHH_i$). 
\end{proof}

\begin{corollary}
Let $\XXX/S$ be as above. For each geometric point $\bar\eta$ of $S$, the Ueno locus of $\XXX_{\bar\eta}$ is the pull-back of the relative Ueno locus along $\bar\eta$.
\end{corollary}

\begin{proof}
Let $Y$ be the Ueno locus of $\XXX_{\bar\eta}$ and $\ZZZ$ be the relative Ueno locus. For the Zariski closure $\YYY$ of $Y$ in $\XXX$ and each point $s\in S(K)$, $\YYY_s$ is a union of positive dimensional cosets and hence $\YYY_s\subset\ZZZ_s$. So $\YYY\subset\ZZZ$ and $Y\subset\ZZZ_{\bar\eta}$. Conversely, by the above proposition, $\ZZZ_{\bar\eta}$ is a union of positive dimensional cosets. Therefore, $\ZZZ_{\bar\eta}\subset Y$.
\end{proof}

\begin{corollary}\label{coro_MLforpointsgeneralpol}
Let $c'(g+t,d)$ be as in Corollary \ref{MLforpoints}. Let $G$ be a polarized semiabelian variety of dimension $g+t$ over an algebraically closed field $K$ of characteristic $0$. Let $X\subset G$ be a reduced closed subscheme of degree $d$. Let $\Gamma\subset G(K)$ be a subgroup of finite rank. Then
$$\#(X^\circ(K)\cap\Gamma)\le c'(g+t,d)^{\rk(\Gamma)+1}.$$
\end{corollary}

\begin{proof}
We first reduces to the case that $\Gamma$ is finitely generated. Assume that corollary statement holds for any finitely generated group. For a $\Gamma\subset G(K)$ of finite rank, we first take $\Gamma_1$ to be a finitely generated subgroup of $\Gamma$ which is of the same rank as $\Gamma$. For $i\geq 2$, let $$\Gamma_{i}:=\Gamma_1\cap\{x\in\Gamma \text{ is a torsion element of order at most } i\}.$$ Then $\Gamma_1\subset\Gamma_2\subset \Gamma_3\subset\cdots$, and $\Gamma=\cup\Gamma_i.$
Therefore we obtain that $\#X^\circ (K)\cap \Gamma\leq c'(g+t,d)^{\rk(\Gamma)+1}$ as $\#X^\circ (K)\cap \Gamma_i\leq c'(g+t,d)^{\rk(\Gamma)+1}$.

From now on, we assume that $\Gamma$ is finitely generated.
Take a subfield $F\subset K$ finitely generated over $\ovl Q$, over which $G,\Gamma,X$ and the Ueno locus of $X$ are defined. We may assume $F=\ovl \Q(S)$ for some quasi-projective variety $S$ over $\ovl \Q$, hence $F\hookrightarrow K$ induces the geometric generic point $\ovl \eta$ of $S$. 

By shrinking $S$, there is a semiabelian scheme $\GGG$ over $S$ and a closed subscheme $\XXX\subset\GGG$ faithfully flat over $S$ with geometrically integral fibers such that the geometric generic fiber is $X\subset G$.

Assume the contrary that
$$\#(X^\circ(K)\cap\Gamma)\ge c'(g+t,d)^{\rk(\Gamma)+1}+1.$$
Fix $x_1,\dots,x_N\in X^\circ(K)\cap\Gamma$ for $N=c'(g+t,d)^{\rk(\Gamma)+1}+1$. Again by shrinking $S$, we may assume that each $x_i$ extends to a section $\bar x_i:U\to\XXX_U$ such that for any $s\in S(\ovl\Q)$, $\bar x_i(s)$ are distinct and $\bar x_i(s)\in\XXX_s^\circ$, since the relative Ueno locus is Zariski closed. However, since the group generated by $\bar x_i(s)$ has rank at most $r$, this contradicts to Corollary \ref{MLforpoints}..
\end{proof}

\subsection{Proof of the Main theorem}\label{subsection_proof_main_theorem}
\begin{proof}
Without loss of generality, we assume that $K$ is algebraically closed. 

We proceed by induction on $g+t$. If $g+t=0$, there is nothing to prove. Assume that the theorem holds for semiabelian varieties of dimension smaller than $g+t.$

Now we want to prove that it holds for semiabelian varieties of dimension $g+t.$
We first reduces to the principally polarized case.
Assuming that Theorem \ref{ML} holds for principally polarized semiabelian varieties of dimension $g+t$, with the constant $c^\#(g+t,d).$

By Lemma \ref{pp_reduction_lemma}, there is an isogeny $f:G'\to G$ of degree at most $a(g+t,d)$ such that $f^*L$ is proportional to a principal polarization $L'$ on $G'$. Take an irreducible component $X'$ of $f^{-1}(X)$ which maps onto $X$. As in Step 2 of the proof to Corollary \ref{MLforpoints}, $\deg_{L'+M'}(\ovl{X'})\le a(g+t,d)d$. So there are at most $c^\#(g+t,a(g+t,d)d)^{\mathrm{rk}(\Gamma)+1}$ cosets $H_1',H_2',\dots$ contained in $X'$ such that
$$X'(K)\cap f^{-1}(\Gamma)=\bigcup H_i'(K)\cap f^{-1}(\Gamma).$$
Here we used the fact that $\mathrm{rk}(f^{-1}(\Gamma))=\mathrm{rk}(\Gamma).$

Take $H_i=f(H_i)$. They are all cosets contained in $X$. For any point $x\in X(K)\cap\Gamma$, take a lifting $x'\in X'(K)\cap f^{-1}(\Gamma)$. Then $x'\in H_i'(K)$ for some $i$ and hence $x=f(x')\in H_i(K)$. Therefore,
$$X(K)\cap \Gamma\subset\bigcup H_i(K)\cap (\Gamma).$$

Now it suffices to prove the principally polarized case, which can be obtained from the following Lemma \ref{lemma_reduction_ML} and the induction hypothesis.
\end{proof}

\begin{lemma}\label{lemma_reduction_ML}
For any positive integers $g+t,d$, there exist an $N$ and a finite sequence of positive integers $d_1,\dots,d_N$ satisfying the property as follows. Let $G$ be a principally polarized semiabelian variety of dimension $g+t$ over an algebraically closed field of characteristic $0$. Let $X\subset G$ be a closed subvariety of degree $d$. Then there exists a subset $I_X\subset\{1,2,\dots N\}$, non-isogenous quotient maps $p_i:G\to G_i$ to polarized semiabelian subvarieties and closed subvarieties $X_i\subset G_i$ of degree at most $d_i$ for each $i\in I_X$ such that the Ueno locus of $X$ is the union of $f_i^{-1}(X_i)$.
\end{lemma}

\begin{proof}
Apply Proposition \ref{prop_relative_ueno} to each irreducible component of the universal restricted Hilbert scheme. We use the same notation as in the proof of the second assertion of Proposition \ref{prop_relative_ueno}. Fix a polarization on each $\GGG_{S_i}/\GGG_i$ over $S_i$. Take the index set $I_X$ consisting of those $i$ such that $S_i$ contains the point representing $X$. For such $i$, let $p_i$ be the restriction of $f_i$ and $d_i$ be the maximal degree of fibers of $\ZZZ_i$ over $S_i$. The existence of $d_i$ follows from generic flatness and induction on dimension.
\end{proof}

\section{Appendix}\label{Appendix}
In this appendix, we prove the existence of the following ``splitting" isogeny:
\begin{proposition}\label{prop_isogeny_trick} Let $G$ be a semiabelian variety, and $B\subset G^m$ be a semiabelian subvariety. We denote by $p_i$ the $i$-th projection $G^m\rightarrow G.$
Assume one of the following:
\begin{itemize}
    \item $G=\mathbb G_m^t$, $B$ is a subtorus;
    \item $B$ is an abelian subvariety.
\end{itemize}
Then there exists an isogeny $h:G^m\rightarrow G^m$ such that after reordering $G^m$, we have
    \begin{enumerate}
        \item $h(B)=B_1\times \cdots\times B_m$ where $B_i=p_{i}(h(B));$
        \item $B_i\not=0$ if and only if $i\leq m'$ for some $0\leq m'\leq m;$
        \item For each $i=0,\dots, m-1$, $h$ preserves the subgroup $\mr{ker}(p_{1},\dots,p_i)$. If we denote the induced isogeny on $\mr{ker}(p_{1},\dots,p_i)\simeq G^{m-i}$ as $h^{(i)},$ then $h^{(m')}$ is the multiplication by $n$ map for some $n>0;$
        \item For each $i=1,\dots, m$ there exists an isogeny $h_i:G\rightarrow G$ such that 
        $$h_i\circ p_{i}= p_{i}\circ h^{(i-1)}.$$
    \end{enumerate}
\end{proposition}

We first prove the following lemma:
\begin{lemma}\label{lemm_extension_reducibility}
    Assume that one of the following is satisfied:
    \begin{enumerate}
        \item[(i)] $G=\mathbb G_m^t$, and $C\subset G$ is a subtorus;
        \item[(ii)] $G$ is a semiabelian variety, and $C\subset G$ is an abelian subvariety.
    \end{enumerate} Then for every homomorphism $f:C\longrightarrow G'$ of semiabelian varieties, there exists a positive integer $N$ and a homomorphism $F:G\longrightarrow G'$ such that $F|_C=[N]_{G'}\circ f$.
\end{lemma}

\begin{proof}
It suffices to prove that there exists a homomorphism
$$\pi:G\rightarrow C$$ such that $\pi|_C=[N]_C$ for some $N>0$.
Indeed, now define $F:=f\circ\pi:G\rightarrow G'$. Then
$$F|_C=f\circ\pi|_C=f\circ [N]_C=[N]_{G'}\circ f.$$

For case (i), the existence of $\pi$ is due to the fact that any splitting subtorus is a direct factor.

For case (ii), we can obtain $\pi$ by Poincar\'e reducibility if $G$ is an abelian variety. We now give a slight generalization.

Let $A$ be the abelian quotient of $G$, and let $C_A$ be the image of $C$ in $A$. Then $C_A$ is an abelian subvariety of $A$. By Poincar\'e reducibility, there exists an abelian subvariety $D_A\subset A$ such that
$C_A+D_A=A,$ and  $C_A\cap D_A \text{ is finite}.$

Let $D$ be the preimage of $D_A$ in $G$. Then $D$ is a semiabelian subvariety of $G$, and the addition map
\begin{align*}g:C\times D&\rightarrow G,\\ (b,d)&\mapsto b+d,
\end{align*}
is an isogeny. 

Take $N=\#\ker g$. Then the homomorphism
$$[N]_C\circ\operatorname{pr}_C:C\times D\rightarrow C$$
vanishes on $\ker g$, so it factors through $g$, i.e., there exists a homomorphism
$\pi:G\rightarrow C$
such that
$$\pi\circ g=[N]_C\circ\operatorname{pr}_C$$
which implies that $\pi|_C=[N]_C$. This concludes the proof.
\end{proof}

\begin{proof}[Proof of Proposition \ref{prop_isogeny_trick}]
We proceed by induction on $m$. The case $m=1$ is trivial. Assume now $m>1$ and $B\neq 0$. After reordering the factors of $G^m$, we may suppose that $p_1(B)\neq 0.$

Put $$ B_1:=p_1(B),\text{ and } K:=(p_2,\dots,p_m)(\text{neutral component of }B\cap\ker p_1).$$Then $B_1\subset G$ is again either a subtorus of $G=\mathbb G_m^t$ or an abelian subvariety of $G$, and $K\subset G^{m-1}$ is of the same type as $B$.

Since $B$ is either a torus or an abelian variety, by Poincar\'e reducibility, there exists a homomorphism $\tau:B_1\rightarrow B$ of finite kernel, and an integer $n>0$ such that \begin{equation}\label{eq_poincare_reduce}
    (p_1|_{B})\circ\tau=[n]_{B_1}.
\end{equation} Indeed, we may replace $\tau$ by its composition with a quasi-inverse of \( (p_1|_{B})\circ \tau\).

Let $\sigma:=(p_2,\dots,p_m)\circ\tau:B_1\rightarrow G^{m-1}$. Then $ \tau(a)=(na,\sigma(a))$
for $a\in B_1.$
By Lemma \ref{lemm_extension_reducibility}, there exist $M>0$ and a homomorphism $\varphi:G\rightarrow G^{m-1}$
such that $\varphi|_{B_1}=[M]_{G^{m-1}}\circ \sigma.$

Define the automorphism
\begin{align*}
    h_0:G^m&\rightarrow G^m\\
    (x_1,x_2,\dots, x_m)&\mapsto (Mnx_1,Mn(x_2,\dots, x_m)-\varphi(x_1))  
\end{align*}
Then $h_0$ preserves every subgroup $\ker(p_1,\ldots,p_i)$ for $i=1,\dots, m$.

For any $a\in B_1$, set $k:=na-\tau(p_1(a))\in B$. Then $p_1(k)=0$ by \eqref{eq_poincare_reduce}, i.e., $k\in \mr{ker}(p_1)\cap B.$ Therefore \begin{align*}h_0(a)&=Mna-(0,\varphi(p_1(a)))\\&=Mna-(0,M\sigma(p_1(a)))\\&=Mna-(0,M(p_2,\dots,p_m)(\tau(p_1(a))))\\
&=Mna+(0,M(p_2,\dots,p_m)(k-na))\\
&=Mna-Mn(p_2,\cdots,p_m)(a)+Mk\\
&=Mn(p_1(a),0)+Mk,
\end{align*}
which implies that $(p_2,\dots, p_m)(h_0(B))\subset (p_2,\dots,p_m)(\mr{ker}(p_1)\cap B)$. For dimension reasons, we must have $h_0(B)= B\times K.$

Applying the induction hypothesis to $K\subset G^{m-1},$
after reordering the last $m-1$ factors of $G^m$, there exists an isogeny
\[
h':G^{m-1}\longrightarrow G^{m-1}
\]
such that
\begin{itemize}
    \item $h'(K)=K_2\times\cdots\times K_m,$ where $K_i=p_i(h'(K)),$ and $K_i\neq 0$ exactly for $i\le m'\leq m.$
    \item $h'$ preserves subgroups $\ker(p_2,\dots, p_i)$ for $i=1,\dots, m-1$. Note that here we number the factors of $G^{m-1}$ by $2,\dots, m$. Moreover, the induced isogeny $h'^{(m')}=h'|_{\ker(p_2,\dots, p_{m'})}=[n']_{G^{m-m'}}$ for some $n'>0$.
    \item For each $i=2,\dots, m$ there exists an isogeny $h'_i:G\rightarrow G$ such that 
        $$h'_i\circ p_{i}= p_{i}\circ h'^{(i-1)}.$$
\end{itemize}
Define
\[
h:=(\mr{id}_G\times h')\circ h_0.
\]
Then
\[
h(B)=(\mr{id}_G\times h')(B_1\times K)=B_1\times K_2\times\cdots\times K_m.
\]
For $i\ge 2$, put $B_i:=K_i$. We thus obtain (1) and (2).

Notice that $$h^{(i)}=(\mr{id_G}\times h')\circ [Mn]_{\ker(p_1,\dots,p_i)}=[Mn]_{G^{m-i}}\circ h'^{(i)}$$
for $i=1,\dots,m-1$, which proves (3) and (4) after putting $h_1=[Mn]_G$ and $h_i=[Mn]_G\circ h_i'$ for $i=2,\dots, m$.
\end{proof}

\bibliographystyle{abbrv}
\bibliography{Betti}
\end{document}

%% file: macros.tex
\def\R{\mathbb{R}}
\def\Q{\mathbb{Q}}
\def\C{\mathbb{C}}
\def\Z{\mathbb{Z}}

\def\O{\mathcal{O}}

\def\deg{{\mathrm{deg}}}
\def\vol{\widehat{\mathrm{vol}}}

\def\Div{\mathrm{Div}}

\def\mr{\mathrm}
\def\mc{\mathcal}

\def\ovl{\overline}

\def\tor{\mathrm{tor}}
\def\ab{\mathrm{ab}}
\def\sab{\mathrm{sab}}

\def\uni{\mathbf{u}}
\def\biZar{\mathrm{biZar}}
\def\Zar{\mathrm{Zar}}

\newcommand{\spec}{\mathrm{Spec}}

\newcommand{\QQ}{{\mathbb Q}}

\newcommand{\ZZ}{{\mathbb Z}}

\newcommand{\AAA}{{\mathcal A}}

\newcommand{\DDD}{{\mathcal D}}
\newcommand{\EEE}{{\mathcal E}}
\newcommand{\GGG}{{\mathcal G}}
\newcommand{\HHH}{{\mathcal H}}
\newcommand{\LLL}{{\mathcal L}}

\newcommand{\PPP}{{\mathcal P}}
\newcommand{\UUU}{{\mathcal U}}
\newcommand{\XXX}{{\mathcal X}}
\newcommand{\YYY}{{\mathcal Y}}
\newcommand{\ZZZ}{{\mathcal Z}}

\newcommand{\rk}{\mathrm{rk}}
\newcommand{\hilb}{\mathrm{Hilb}}

\newcommand{\Pic}{\widehat{\mathrm{Pic}}}
\newcommand{\Deg}{\widehat{\deg}}

\makeatletter
\newcommand\tint{\mathop{\mathpalette\tb@int{t}}\!\int}
\newcommand\bint{\mathop{\mathpalette\tb@int{b}}\!\int}
\newcommand\tb@int[2]{%
  \sbox\z@{$\m@th#1\int$}%
  \if#2t%
    \rlap{\hbox to\wd\z@{%
      \hfil
      \vrule width .35em height \dimexpr\ht\z@+1.4pt\relax depth -\dimexpr\ht\z@+1pt\relax
      \kern.05em 
    }}
  \else
    \rlap{\hbox to\wd\z@{%
      \vrule width .35em height -\dimexpr\dp\z@+1pt\relax depth \dimexpr\dp\z@+1.4pt\relax
      \hfil
    }}
  \fi
}
\makeatother

\usepackage{etoolbox}
\makeatletter
\newcommand*\suppresschapternumber{%
  \let\@makechapterhead\@makeschapterhead
  \patchcmd{\@chapter}
    {\protect\numberline{\thechapter}}
    {}
    {}{}%
}
\newcommand*\removedotbetweenchapterandsection{%
  \renewcommand\thesection{\thechapter\@arabic\c@section}%
}
\makeatother